\documentclass{amsart}
\usepackage{preamble}

\title[Narrow escape problems beyond spherical holes: slits and tubes]{Narrow escape problems beyond spherical holes:\\ slits and tubes}
\author[Carillo]{Louis Carillo\\
\small{ CERMICS, CNRS, ENPC, Institut Polytechnique de Paris, Marne-la-Vallée, France. \\
MATHERIALS project-team, Inria Paris, France.}} 
\begin{document}

\begin{abstract}
    The narrow escape problem consists in characterizing the exit event of 
    a Brownian motion trapped inside a domain with reflecting boundary and small absorbing holes.
    Using recent advances in the quasi-stationary distribution approach, 
    we first give an explicit set of hypotheses that allows
    one to derive the asymptotics of the exit time and of the exit hole distribution. 
    We then apply it in dimension three to several hole shapes: 
    spherical holes, thin spheroidal holes shrinking to curves on the boundary (slits),
    and holes located at the end of thin tubes attached to the domain.
    We illustrate the accuracy of the results with numerical simulations, 
    and extend them to higher dimensions.
\end{abstract}
\maketitle

\tableofcontents

\section{Introduction}

The narrow escape problem was initially introduced in biology,
in the context of the diffusion of molecules in a cell \cite{HolcmanSchuss04},
in particular to model the diffusion of neurotransmitters in a synapse.
This physical phenomenon can be modeled by a reflected Brownian motion constrained to a domain,
with a small absorbing part on its boundary,
modeling the exit of the molecule, or its reaction with a receptor.
Under the assumption that the absorbing part is small,
the goal is to compute the mean exit time of the molecule
and, when there are several absorbing parts,
the probability of exiting through each of them.
This problem has attracted interest both in the physics community \cite{Singer2006, Benichou2008, GomezCheviakov2015, Grebenkov2024}
and in the mathematical community \cite{AMMARI201266, ChenFriedman2011, neck, Tony2024, Hsu2025b}.
In this work,
we extend the contributions of \cite{Tony2024, nous, narrow2d}
to study different shapes of the absorbing part,
within a similar framework: the quasi-stationary approach.
This introduction is organized as follows.
First, we define the basic setting of the narrow escape problem
and the quasi-stationary distribution in \cref{sub:intro qsd};
then we describe the various hole shapes studied in this work and their connection to the literature in \cref{sub:literature};
finally, the contributions and the outline of the paper are given in \cref{sub:outline}.

\subsection{Introduction to the narrow escape problem and the quasi-stationary distribution}
\label{sub:intro qsd}
In this section, 
we briefly introduce the mathematical fundamentals of the narrow escape problem 
and justify the quasi-stationary approach we use in the rest of the paper.
Let $\Omega \subset \real^3$ be a smooth bounded domain with boundary $\partial \Omega$.
Denote by $(X_t)_{t \geq 0}$ the stochastic process modeling the diffusion of
a particle trapped inside $\Omega$ with reflecting boundary conditions on $\partial \Omega$.
The goal is to study the exit of the particle through a small absorbing part $\Gde \subset \partial \Omega$,
where $\varepsilon > 0$ controls the size of~$\Gde$.
Explicitly,
this exit can be characterized by a pair of random variables: the exit time
\begin{equation}\label{eq:exit time}
    \tau_\varepsilon = \inf\{ t > 0 \, | \, X_t \in \overline{\Gde}\},
\end{equation}
and the exit position 
\begin{equation}\label{eq:exit position}
    X_{\tau_\varepsilon} \in \overline{\Gde}.
\end{equation}

\begin{figure}
    \centering
    \includegraphics[width=0.5\textwidth]{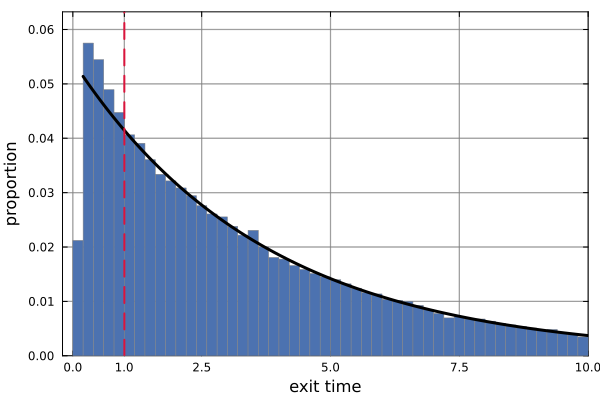}
    \caption{Histogram of the exit time of $100\,000$ particles in a unit ball of dimension $3$
    with $2$ absorbing spherical caps of angle $0.2$ radians on its boundary.
    Notice the exponential tail of the distribution, drawn in black,
    typical of the exit time under the quasi-stationary distribution.
    The vertical red line represents the time from which one can heuristically
    consider the particles to be distributed according to the quasi-stationary distribution.
    In this case, less than $20\%$ of the particles have escaped before this time.
    This proportion shrinks to $0$ in the limit $\varepsilon \to 0$ (see \cite[Section 4.2]{nous}
    for an illustration of this phenomenon).}
    \label{fig:qsd}
\end{figure}

In general,
both of these random variables depend on the initial condition $X_0$ of the process $(X_t)_{t \geq 0}$.
To overcome this difficulty,
in this work we consider the initial condition to be distributed according to the quasi-stationary distribution~$\nu_\varepsilon$.
The quasi-stationary distribution $\nu_\varepsilon$ is a probability measure on~$\Omega$ that is stationary under the dynamics
conditioned on the event that the particle has not escaped yet \cite{ColletQSD, DiGesu2016}.
It has been shown in \cite[Theorem 1.1, Remark 1.2]{nous} that $\nu_\varepsilon$ exists
in a wide class of narrow escape problems and that,
for any initial condition $X_0 \in \Omega$,
\begin{equation*}
    \mathrm{Law}(X_t \, | \, t < \tau_\varepsilon) \underset{t \to +\infty}{\longrightarrow} \nu_\varepsilon.
\end{equation*}
Hence,
at large times,
the law of the process $(X_t)_{t \geq 0}$, conditioned on the fact that it has not escaped yet,
tends to the quasi-stationary distribution (see \cref{fig:qsd} for an illustration).
This relaxation towards $\nu_\varepsilon$ occurs on a time scale that stays bounded
as $\varepsilon \to 0$,
and thus is negligible compared to the exit time $\tau_\varepsilon$, which diverges in this limit.
Additionally,
starting from $\nu_\varepsilon$,
\cite[Proposition 1.3]{nous} shows that both
$\tau_\varepsilon$ and $X_{\tau_\varepsilon}$ can be derived from $\nu_\varepsilon$
(see their explicit relation in \cref{sub:qsd-to-narrow escape}).
Consequently,
the goal of this paper is to study the quasi-stationary distribution $\nu_\varepsilon$,
for different shapes of the absorbing part $\Gde$, and to derive the asymptotic behavior of the exit time~$\tau_\varepsilon$ and exit position~$X_{\tau_\varepsilon}$
in the limit $\varepsilon \to 0$.

\subsection{Hole shapes and bibliography review}
\label{sub:literature}
Let us now describe the different hole shapes studied in this work,
which model various physical phenomena; see the global illustration in \cref{fig:hole shapes}.

\begin{figure}
    \centering
    \includegraphics[width=0.9\textwidth]{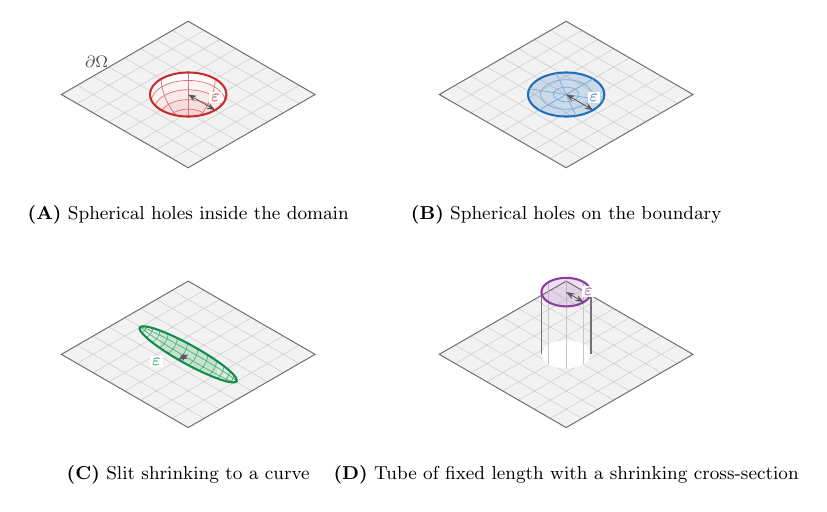}
    \caption{Illustration of the different hole shapes considered in this work. 
    The gray plane represents the boundary of the domain $\partial \Omega$,
    while the colored areas are the holes~$\Gde$.
    The parameter $\varepsilon$ illustrates how the holes shrink.
    }
    \label{fig:hole shapes}
\end{figure}
\paragraph{The narrow escape problem} is the standard setting introduced by \cite{HolcmanSchuss04},
where the holes are geodesic disks on the boundary of the domain that shrink to a point as $\varepsilon \to 0$.
It is represented in blue in \cref{fig:hole shapes} (B).
Some notable contributions in the literature on this version of the
problem are~\cite{wardold, Singer2008, Singer2006, Benichou2008, AMMARI201266, ChenFriedman2011, Tzou2021}
which all focus on the mean exit time or on the exit time from a fixed starting point.
The quasi-stationary distribution approach has been used in \cite{narrow2d}
to study the exit time and exit position in dimension $2$.
Notice two other contributions \cite{nous, Tony2024}, which use the quasi-stationary distribution
but for slightly modified holes that are half-spheres inside the domain;
see \cref{fig:hole shapes} (A).
One goal of this paper is to understand the difference between the two settings.

\paragraph{The narrow slit problem} is a variant of the narrow escape problem
where one considers a hole that shrinks to a curve of finite length
on the boundary of the domain instead of a point.
In practice,
this is built from a prolate spheroid that shrinks to a curve segment (see \cref{fig:hole shapes} (C)).
Notice that this problem is fundamentally different, as the hole
shrinks to an object of dimension
$1$. 
Heuristically,
one could expect the mean exit time to diverge logarithmically,
as in the two-dimensional narrow escape problem,
pointing to the importance of the codimension of the hole instead of the dimension of the domain.

The geometry behind the narrow slit problem is not new and
has been studied in a wide range of applications,
from electromagnetism \cite{slot} and mechanics \cite{BORODACHEV1974108}
to hydrodynamics \cite{Slender-ribbontheory} and acoustics \cite{1938_Morse}.
Yet, to the best of our knowledge,
this regime has not been studied in the context of the narrow escape problem.
However,
using the results on ellipsoidal holes from~\cite{narrowescapepart1, Tzou2021},
we can gain insight into the behavior of the narrow slit problem.
Formally,
taking the ellipticity of the ellipsoid to scale as $\varepsilon^{-1}$
in~\cite[Theorem 1.3]{Tzou2021} leads to a logarithmic divergence of the mean exit time,
which confirms the heuristic prediction above.
Finally, 
note that this problem can be tackled with capacity methods 
to obtain the scaling in $\varepsilon$ of the mean exit time (see~\cite{Felli2021}).

\paragraph{The narrow transition problem} is the variant illustrated in purple in
\cref{fig:hole shapes} (D),
where one considers a tube of fixed length with a shrinking cross-section.
The exit hole is at the end of the tube.
The narrow escape community has considered this problem since \cite{Schussreview};
see a list of recent contributions on this topic in \cite{Berezhkovskii, neck, Hsu2025a, Hsu2025b, yetanothertube}.
Let us also point out the strong connection between this problem and Helmholtz resonators in 
acoustics/electrostatics \cite{Ingard, HOWE1976427, JIMBO1989322, ARRIETA199124}.

Another interpretation of this problem arises from statistical physics,
in the context of transitions from one state to another 
(see \cite{Zwanzig} for a discussion of several problems involving Brownian motions trapped in channels, 
and \cite{FRENKEL199926} for entropy-driven phenomena).
Indeed,
if at the end of the tube there is another domain $\Omega'$,
then the escape time through the tube is also the transition time from $\Omega$ to $\Omega'$.
Transition times are fundamental objects in the study of metastable systems \cite{KRAMERS1940284, reviewkramers, voter}.
Usually, the metastability is induced by an energetic barrier between $\Omega$ and $\Omega'$,
hence the transition time is trivially linked by a factor $2$ to the escape time from $\Omega$ (see \cite[Remark 1]{TonyBoris2022}).
In this work, where $\Omega$ and $\Omega'$ are linked by a narrow tube,
we see that this is not the case,
and the transition time can be much larger than the escape time.

The analogy between the narrow transition problem and the energetic barrier framework can be pushed further.
The transition time for the latter is given by the Eyring--Kramers law \cite{KRAMERS1940284, reviewkramers, TonyBoris2022},
\[
    \mathbb{E}[\tau] \sim \frac{2\pi}{|\lambda_1(z)|} \sqrt{\frac{|\det \nabla^2 V(z)|}{\det \nabla^2 V(x)}} \exp\left(\frac{V(z) - V(x)}{h}\right),
\]
where $h$ is proportional to the temperature and plays the role of the small parameter, $V$ is the energetic potential,
$\lambda_1$ is the negative eigenvalue of the Hessian of $V$ at the saddle point $z$ between the two wells,
and $x$ is the local minimum of $V$ in the well where the process starts.
Formally,
the ratio of the Hessians can be seen as a measure of the volume of the well over the volume of the transition region 
$\frac{|\Omega|}{\varepsilon^2 L}$, where $L$ is the length of the tube and $\varepsilon^2$ its cross-sectional area.  
The factor $|\lambda_1|^{-1}$ is a measure of the rate to traverse the transition region, here $L^2$.
Taking the exponent to $0$ as there are no energetic barriers in the narrow transition problem,
the formula for the mean exit time can be understood as 
\[
    \mathbb{E}[\tau] \sim  L \frac{|\Omega|}{\varepsilon^2}.
\]
In \cref{sub:outline}, 
we see that these heuristics lead to the correct scaling behavior for the mean exit time in the narrow transition problem, 
making a link between metastability arising from an energetic constraint and from an entropic constraint. 

\subsection{Contribution and outline of the paper}
\label{sub:outline}
The contribution of this work is to derive the exit time distribution and the exit hole distribution
for the narrow escape problem (in two variants, A and B), the slit problem (C), and the
transition problem (D),
within a single quasi-stationary distribution framework.
Let us emphasize that the novelty of this work is to demonstrate that the framework introduced in
\cite{Tony2024} can be applied to a wide range of hole shapes,
and to derive asymptotically the exit time distribution and the exit hole distribution from the quasi-stationary distribution,
in the limit $\varepsilon \to 0$.
To do this, 
new quasi-modes have been built to approximate the quasi-stationary distribution for the different hole shapes. 
This technique is explained in \cref{sub:qsd-to-narrow escape}.
The results of this paper are gathered in \cref{tab:summary}, which summarizes the different hole shapes and
the asymptotics of their mean exit time and the exit hole distribution in a reference setting (one hole for the mean exit time, two holes for the exit hole) in the limit $\varepsilon \to 0$.

For the sake of comparison,
we also include a column stating how the surface area of each hole scales with $\varepsilon$.
Notice that the results in \cref{tab:summary} are consistent with the heuristics on the mean exit time developed in \cref{sub:literature}.
Furthermore,
we make explicit the difference between the spherical holes on the boundary studied in \cite{narrow2d}
and the spherical holes studied in \cite{Tony2024, nous}, which are half-spheres inside the domain.
The scaling of the mean exit time differs by a factor $\pi/2$ between the two settings,
giving a rigorous explanation of \cite[Remark 1.5]{nous}.

\begin{table}[h!]
    \centering
    \renewcommand{\arraystretch}{2.2}
    \begin{tabular}{
        @{} c
        >{\raggedright\arraybackslash}p{4.7cm}
        c
        c
        c
        c @{}
    }
        \toprule
        Hole type & Hole description & \makecell{Scaling of the\\surface area} & \makecell{Mean exit time} & \makecell{Exit hole} & \makecell{Reference} \\
        \midrule
        A & Spherical holes inside the domain & $\varepsilon^2$ & $\dfrac{|\Omega|}{2\pi\varepsilon}$ & $\dfrac{1}{3}$ & \cref{prop:narrow escape spherical inside} \\
        B & Spherical holes on the boundary & $\varepsilon^2$ & $\dfrac{|\Omega|}{4\varepsilon}$ & $\dfrac{1}{3}$ & \cref{prop:quasi-mode spherical holes} \\
        C & Prolate spheroid around a curved segment of length $2\ell$ on $\partial\Omega$, half-width
$\varepsilon$ & $\varepsilon \ell$ & $\dfrac{|\Omega|}{2\pi \ell}|\log(\varepsilon)|$ & $\dfrac{1}{2}$ & \cref{prop:narrow escape slit} \\
        D & End of a tube of fixed length $L$ with shrinking cross-section of area $\pi \varepsilon^2$ & $\varepsilon^2$ &
$\dfrac{|\Omega| L}{\pi \varepsilon^2}$ & $\dfrac{1}{5}$ & \cref{prop:quasi-mode tube} \\
        \bottomrule
    \end{tabular}
    \caption{Summary of the different hole shapes considered in this work and their properties.
    The notation A, B, C and D refers to the illustrations in \cref{fig:hole shapes}.
    See \eqref{eq:spherical holes}, \eqref{eq:geodesic spherical holes}, \eqref{eq:slit def} 
    and \eqref{eq:tube def} in \cref{sub:hole shapes notation} for the precise definitions of the holes.
    The mean exit time is $\mathbb{E}_{\nu_\varepsilon}[\tau_\varepsilon]$, the expectation of \eqref{eq:exit time} under the quasi-stationary distribution $\nu_\varepsilon$.
    The exit hole column refers to the probability of exiting through a hole of parameter $\varepsilon$ in a system with another hole of parameter $2\varepsilon$ (with the same $\ell$ for C and the same $L$ for D).
    The reference column gives the proposition in the paper where the result is stated and proved.
    Each result is given as its asymptotic equivalent in the limit $\varepsilon \to 0$.
    }
    \label{tab:summary}
\end{table}

The results presented in \cref{tab:summary} are believed to be derivable in higher dimensions,
at the cost of lengthy computations. 
This is one goal of the numerical experiments in \cref{sec:numerics}, 
which aim to illustrate the results of this work and to explore the higher-dimensional case.

The paper is organized as follows. 
In \cref{sec:notation and results},
we introduce the main notation and explain the quasi-stationary distribution framework.
Then,
in \cref{sec:several narrow escape problems},
we present our results on the spherical holes A and B.
In~\cref{sec:slit}, 
we construct a quasi-mode for the narrow slit problem and derive the asymptotics of the exit time and exit hole distributions.
The same is done for the narrow transition problem in \cref{sec:tubular holes}.
Finally, 
we present numerical experiments in \cref{sec:numerics} to illustrate the results of this work and explore the higher-dimensional case.
Additionally, 
in \cref{app:poincare tube}, we discuss the Poincaré--Wirtinger inequality uniform in $\varepsilon$ in the several geometries of this work. 
We then provide the proof of \cref{thm:narrow escape} in~\cref{app:proof}.
\Cref{sec:app:result Green} gathers the fundamental properties of the Green function considered in this work, 
and \cref{sec:app:integrals} is a technical section on the computation of several integrals.

\section{Notation and detailed method}\label{sec:notation and results}
In this section 
we introduce the mathematical notation, 
first in general in \cref{sub:notation}, 
then for the different hole shapes in \cref{sub:hole shapes notation}. 
Finally,
in \cref{sub:qsd-to-narrow escape}, 
we explain how to use the quasi-stationary distribution to derive the exit time and exit hole distributions,
using a quasi-mode. 

\subsection{Notation}
\label{sub:notation}
In this section,
we define the notation and the mathematical concepts used throughout this work.
We first introduce geometric considerations, before introducing the operator associated with the process $(X_t)_{t \geq 0}$ defined in \eqref{eq:reflected bm}.
Note that in this paper, 
four geometric settings are considered. 
For simplicity, we use the same notation in all four settings, 
with the understanding that the underlying geometric setting is clear from the context.

We recall that $\Omega$ is a smooth bounded domain of $\real^3$ with boundary $\partial \Omega$.
Let $(X_t)_{t \geq 0}$ be the following stochastic process evolving in $\overline{\Omega}$:
\begin{equation}\label{eq:reflected bm}
    \mathrm{d} X_t = \sqrt{2} \, \mathrm{d}B_t - n(X_t) \, \mathrm{d}L_t,
\end{equation}
where $(B_t)_{t \geq 0}$ is a standard Brownian motion in $\real^3$,
$(L_t)_{t \geq 0}$ is the local time of $(X_t)_{t \geq 0}$ on $\partial \Omega$
(see~\cite{Tanaka, ReflectedBrownian}), and
$n$ is the outward unit normal vector to $\partial \Omega$.
This process is a reflected Brownian motion in $\Omega$.
For~$N \geq 1$ and~$\varepsilon > 0$,
we consider a family of $N$ (relatively) open connected sets $\Gke \subset \overline{\Omega}$, $k \in \{ 1, \ldots, N\}$, called the holes,
with a small parameter~$\rek > 0$ that controls their size  (for the \eqref{eq:slit def} holes, $w_\varepsilon^{(k)}$ plays a similar role).
We write~$\Gde = \bigcup_{k=1}^{\ndoors} \Gke$ for their union.
The small hole asymptotics is defined by the limit $\varepsilon \to 0$,
where~$\rek \to 0$ and $|\Gke| \to 0$. In this work,
we emphasize that we only consider the limit where the holes shrink but do not move.
Hence, we always assume that the holes remain separated by a distance bounded below uniformly in $\varepsilon$.

\paragraph{The generator of $(X_t)_{t \geq 0}$.}
The generator of the process $(X_t)_{t \geq 0}$, defined in \eqref{eq:reflected bm}, killed on $\Gde$ is the Laplacian $\Delta$ on $\Omega_\varepsilon$
with Dirichlet boundary conditions on $\Gde$ and Neumann boundary conditions on $\Gne$.
In the following, we work with the nonnegative operator $\mathcal{L}_\varepsilon = -\Delta$,
whose domain is
\begin{equation}\label{eq:domain L}
    \mathcal{D}(\mathcal{L}_\varepsilon) = \Bigl\{ u \in H^1(\Omega_\varepsilon)\, | \, \Delta u \in L^2(\Omega_\varepsilon), \, u = 0 \text{ on } \Gde , \, \partial_n u = 0 \text{ on } \Gne \Bigr\},
\end{equation}
where the normal derivative $\partial_n u$ is defined by duality on $H^{-\frac{1}{2}}(\partial \Omega_\varepsilon)$ (see \cref{app:proof} for an explicit definition).
The operator $\mathcal{L}_\varepsilon$ can be defined as the self-adjoint operator associated with the closed form
\begin{equation}\label{eq:def bilinear form}
    \mathfrak{q}_\varepsilon(u, v) = \int_{\Omega_\varepsilon} \nabla u \cdot \nabla v,
    \qquad u, v \in H^1_{\Gde}(\Omega_\varepsilon),
\end{equation}
with form domain
\[
    H^1_{\Gde}(\Omega_\varepsilon) = \left\{ u \in H^1(\Omega_\varepsilon) \, | \, u = 0 \text{ on } \Gde \right\}.
\]
It is classical that
$\mathcal{L}_\varepsilon$ is self-adjoint,
with a compact resolvent (see \cite[Theorem 1]{Savare} and the Rellich--Kondrachov theorem),
so it has a discrete spectrum.
Furthermore,
its principal eigenvalue is positive and
associated with a positive and continuous eigenfunction \cite[Theorem 8.38]{Gilbrag}.
The density of the quasi-stationary distribution $\nu_\varepsilon$,
introduced in \cref{sub:intro qsd},
is this principal eigenfunction \cite[Theorem 1.1]{nous}.
Explicitly,
denoting this density by~$v_\varepsilon$,
it holds that
\begin{equation}\label{eq:edp density}
    \left\{
        \begin{aligned}
            -\Delta v_\varepsilon & = \lambda_\varepsilon v_\varepsilon && \text{ in } \Omega_\varepsilon, \\
            \partial_n v_\varepsilon & = 0 && \text{ on } \Gne,\\
            v_\varepsilon & = 0 && \text{ on } \Gde,\\
            \int_{\Omega_\varepsilon} v_\varepsilon & = 1.
        \end{aligned}
    \right.
\end{equation}

\subsection{Notation for the different hole shapes}
\label{sub:hole shapes notation}
With this notation,
the hole shapes considered in this work (see \cref{fig:hole shapes} and \cref{sub:literature}) are defined as follows.
For each of the constructions considered in~\eqref{eq:spherical holes} and~\eqref{eq:slit def},
removing the holes from (or, in~\eqref{eq:tube def}, attaching the tubes to) $\Omega$ produces a new domain $\Omega_\varepsilon$,
whose boundary splits into the union
\[
    \partial \Omega_\varepsilon = \overline{\Gde} \cup \overline{\Gne},
\]
where $\Gne$ denotes the reflecting (Neumann) part of the boundary,
inherited from $\partial \Omega$.
Fix $k \in \{1, \ldots, N \}$, $\varepsilon > 0$ 
and $\rek > 0$.

\paragraph{Spherical holes.}
The spherical holes are defined as 
\begin{equation}\label{eq:spherical holes}
\tag{Type A}    \Gke = \partial B (\xk, \rek) \cap \Omega,
\end{equation}
where $\xk$ is a point in $\partial \Omega$,
similarly to \cite{nous}.
We assume that the points $\xk$ are distinct; more precisely, that there exists $\eta > 0$ such that the balls $B(\xk, \eta)$ are disjoint.
Then,
for spherical holes on the boundary,
denote by $d_{\partial \Omega}(\cdot, \cdot)$ the geodesic distance in $\partial \Omega$,
and set
\begin{equation}\label{eq:geodesic spherical holes}
\tag{Type B}    \Gke = \{ x \in \partial \Omega\, | \,  d_{\partial \Omega}(x, \xk) < \rek \}.
\end{equation}
In this work, we choose the geodesic distance in \eqref{eq:geodesic spherical holes}
for the sake of simplicity. Using the underlying Euclidean distance instead
is possible with the same approach, and changes only the error terms in the results.

\paragraph{Slit holes.}
For the slit problem,
let $\Hk \subset \partial \Omega$ be a smooth simple arc.
Since these curves are the limiting sets of the holes, which remain separated,
we also assume that the curves $(\mathcal{H}_j)_{j=1}^N$ are pairwise disjoint.
Consider a smooth parametrization of $\Hk$ through a function
$\gamma_k \colon [-\ell_k, \ell_k] \to \Hk$
such that $|\gamma_k'|=1$, so that $\Hk$ has length $2\ell_k$.

To make the shape of the holes explicit,
a local basis is needed to describe the geometry of the slit: the Frenet coordinates.
For $x$ close enough to $\Hk$, the orthogonal projection of $x$ onto $\Hk$,
denoted $\pi_k(x) \in \Hk$, is uniquely defined.
Hence, there exists $s_x \in [-\ell_k, \ell_k]$ such that $\gamma_k(s_x) = \pi_k(x)$.
Notice that $\pi_k(x)$ may lie at the tip of the slit,
for instance if $x = \gamma_k(\ell_k) + t\, \tau_k(\ell_k)$ for some $t>0$,
with $\tau_k(\cdot)$ the tangent vector to the slit $\Hk$ at $\gamma_k(\cdot)$.
By definition of the Frenet frame,
any point $x$ belonging to a small tubular neighborhood of $\Hk$ can be decomposed as
\begin{equation}\label{eq:decomposition X}
    x = \gamma_k(s_x) + t_\xk\, \tau_k(s_x) + \rho_\xk,
\end{equation}
with $t_\xk \in \real$ and $\rho_\xk \perp \tau_k(s_x)$.
In this setting,
for $\varepsilon$ small enough, the hole is defined as
\begin{equation}\label{eq:slit def}
    \tag{Type C}
    \Gke =
    \left\{ x \in \Omega \, \left| \, \frac{(s_x + t_\xk)^2}{(w_\varepsilon^{(k)})^2 + \ell_k^2} + \frac{|\rho_\xk|^2}{(w_\varepsilon^{(k)})^2} = 1\right.\right\}.
\end{equation}
By design,
the sets $\Gke$ are ellipsoids around the slit $\Hk$ in the basis defined in \eqref{eq:decomposition X} that shrink to $\Hk$ as $\varepsilon \to 0$.
The parameter $w_\varepsilon^{(k)}$ is the half-width of the ellipsoid, as it is its minor semi-axis.
Notice that the holes $\Gke \subset \overline{\Omega}$, 
hence the slit holes are the analogue of the \eqref{eq:spherical holes} holes.
A similar construction could be done for slits fully contained in the boundary like in \eqref{eq:geodesic spherical holes}.

Finally, we define the major semi-axis, 
which is useful in the analysis of the slit problem,
 as
\begin{equation}\label{eq:slit major axis}
    a_\varepsilon^{(k)} = \sqrt{\ell_k^2 + \left(w_\varepsilon^{(k)}\right)^2},
    \qquad \text{ so that } \qquad 
    \left(a_\varepsilon^{(k)} - \ell_k\right)\left(a_\varepsilon^{(k)} + \ell_k\right) = \left(w_\varepsilon^{(k)}\right)^2.
\end{equation}

\paragraph{Tubular holes.}
To construct the tubular holes,
pick a point $\xk \in \partial \Omega$
and a bounded connected smooth planar domain $\mathcal{D}_k \subset T_{\xk}(\partial \Omega)$
containing the origin,
where $T_{\xk}(\partial \Omega)$ is the tangent plane to $\partial \Omega$ at $\xk$,
and set
\begin{equation}\label{eq:tube def D}
    \Dke = \rek \, \mathcal{D}_k,
    \qquad \text{ so that } \qquad
    |\Dke| = (\rek)^2 |\mathcal{D}_k| ,
\end{equation}
where $\rek$ plays the role of the radius of the $k$-th tube.
We denote $\reo = \sum_{j=1}^{\ndoors} r_\varepsilon^{(j)}$.

Denote by $n_k$ the outward unit normal to $\partial \Omega$ at $\xk$ so that points close to $\xk$ are written $x = \xk + \rho + s \, n_k$
with $\rho \in T_{\xk}(\partial \Omega)$ and $s \in \real$.
As $\partial \Omega$ is smooth,
there are $r_0 > 0$ and a smooth function~$h_k \colon B(0,r_0) \to \real$
with $h_k(0) = 0$ and $\nabla h_k(0) = 0$,
such that $\partial \Omega$ is locally the graph $s = h_k(\rho)$
and $\Omega$ lies locally in $\{ s < h_k(\rho)\}$.
In particular, uniformly for $|\rho| \leq C \rek$,
\begin{equation}\label{eq:tube graph}
    h_k = \mathrm{O}\bigl((\rek)^2\bigr),
    \qquad
    \nabla h_k = \mathrm{O}(\rek).
\end{equation}
For $\varepsilon$ small enough that $\overline{\Dke} \subset B(0,r_0)$,
define the tube as
\begin{equation}\label{eq:tube lim}
    \Tke
    = \bigl\{ \xk + \rho + s \, n_k \; : \; \rho \in \Dke, \; h_k(\rho) < s < L_k \bigr\} ,
\end{equation}
with $L_k > 0$ the length of the tube,
and the holes as
\begin{equation}\label{eq:tube def}
\tag{Type D}
    \Gke
    = \bigl\{ \xk + \rho + L_k \, n_k \; : \; \rho \in \Dke \bigr\} .
\end{equation}
We shall always assume that the tubes never intersect each other or the domain itself.

We denote~$\Ske = \overline{\Omega} \cap \overline{\Tke}$ for $k \in \{1, \ldots, N\}$ the $k$-th tube's mouth, 
or explicitly, for $\varepsilon$ small enough
\begin{equation}\label{eq:def tube mouth}
    \Ske = \left\{ \xk + \rho + h_k(\rho) n_k \, \middle| \, \rho \in \Dke \right\},
\end{equation}
where $h_k$ is the function introduced before \eqref{eq:tube graph}.
To be explicit in this case, the resulting domain $\Omega_\varepsilon$ is constructed by plugging the tubes $\Tke$ into the domain $\Omega$:
\begin{equation}\label{eq:tube domain}
    \Omega_\varepsilon = \Omega \cup \bigcup_{k=1}^N \left(\Tke \cup \Ske\right) \setminus \partial \Ske.
\end{equation}
This problem is fundamentally different from the others, 
as  
$\Gke \not \subset \overline{\Omega}$ 
but also because the domain changes with $\varepsilon$ significantly. 
For instance, 
the uniform interior cone condition does not hold, 
hence classical results like \cite[Theorem 1.2]{Ruiz} do not apply directly in order to obtain the Poincaré--Wirtinger inequality uniformly in $\varepsilon$.
This inequality is proven in \cref{lem:poincare tube} using a different approach.

Finally, we introduce, for all $k \in \{1, \ldots, N\}$, the following quantity
\begin{equation}\label{eq:sk}
    \ske = \sup_{\rho \in \Dke} |h_k(\rho)|,
\end{equation}
so that the length of the tube is at most $L_k + \ske$, 
which can be bounded uniformly in $\varepsilon$.

The \eqref{eq:slit def} and \eqref{eq:tube def} notations are illustrated in \cref{fig:coords}.

\begin{figure}
    \centering
    \includegraphics[width=1.0\textwidth]{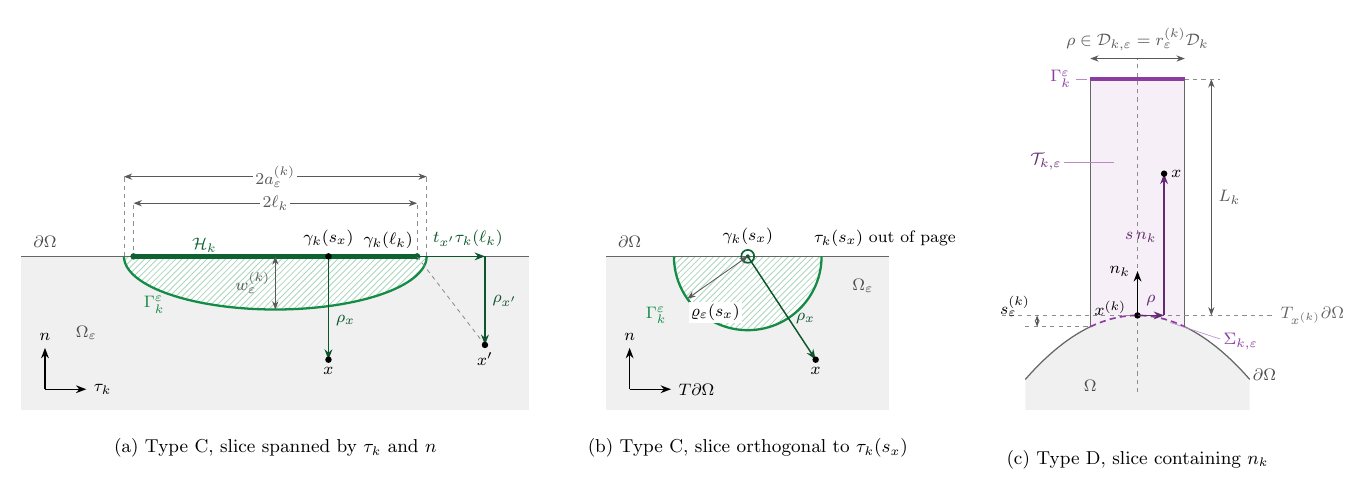}
    \caption{Illustration of the different notations used for the holes \eqref{eq:slit def} and \eqref{eq:tube def}.}
    \label{fig:coords}
\end{figure}
\subsection{From the quasi-stationary distribution to the narrow escape problem}
\label{sub:qsd-to-narrow escape}
In this section, 
we state \cref{thm:narrow escape} that we will apply in different settings throughout the paper, 
which will give us the results claimed in \cref{tab:summary}.

As introduced in \cref{sub:intro qsd}, 
the quasi-stationary distribution $\nu_\varepsilon$ is a key object to study the narrow escape problem. 
Classical results on the quasi-stationary distribution \cite[Proposition 1.3]{nous}
show that, if $X_0 \sim \nu_\varepsilon$, the exit time $\tau_\varepsilon$ and the exit position $X_{\tau_\varepsilon}$ are independent, 
and
\begin{equation}
    \tau_\varepsilon \sim \mathrm{Exp}(\lambda_\varepsilon), \quad \text{ and } \quad 
    \forall k \in \{1, \ldots, N\}, \quad
    \mathbb{P}_{\nu_\varepsilon}\left( X_{\tau_\varepsilon} \in \overline{\Gke} \right) 
    = -\lambda_\varepsilon^{-1} \int_{\overline{\Gke}} \partial_n v_\varepsilon \, \d\sigma. 
\end{equation}
However, 
in a general setting, 
it is impossible to obtain an explicit form for $\nu_\varepsilon$. 
The classical approach is to build a quasi-mode $\varphi_\varepsilon$ 
that verifies a modified version of \eqref{eq:edp density} and leverage the results from it. 
This is the technique applied in \cite{nous}. 
In this work, 
the quasi-mode is built as a sum of contributions of each hole, 
in the form
\begin{equation}\label{eq:def quasi-mode}
    \varphi_\varepsilon = 1 - \sum_{k=1}^{\ndoors} \psi_\varepsilon^{(k)},
\end{equation}
where $\psi_\varepsilon^{(k)}$ is close to $1$ on the $k$-th hole and small on the other holes.
Let us first state the hypotheses needed for the quasi-mode. 

\begin{hypo}\label{hyp:quasi-mode}
    The principal eigenvalue $\lambda_\varepsilon$ of the operator $\mathcal{L}_\varepsilon$ tends to $0$ as $\varepsilon \to 0$:
    \begin{equation}\label{eq:hyp:eig}
        \tag{Eig}
        \lambda_\varepsilon \underset{\varepsilon \to 0}{\longrightarrow} 0.
    \end{equation}
    Furthermore, 
    for each $k \in \{1, \ldots, N\}$,
    there exists a positive number $\Kek$ such that $\Kek \to 0$ as $\varepsilon \to 0$,
    and a function $\psi_\varepsilon^{(k)} \in H^1(\Omega_\varepsilon) \cap C^0(\overline{\Gde})$
    that verify:
    \begin{equation}\label{eq:hyp:PDE quasi-mode}
        \tag{PDE}
        \left\{
            \begin{aligned}
            \Delta \psi_\varepsilon^{(k)} & = C \Kek && \text{ in } \Omega_\varepsilon, \\
            \partial_n \psi_\varepsilon^{(k)} & = 0 && \text{ on } \Gne,\\
            \bigl\|\psi_\varepsilon^{(k)} - \mathbbm{1}_{\overline{\Gke}}\bigr\|_{L^\infty(\Gde)} & = \mathrm{O}(\mathcal{E}_1(\Kall)),
            \end{aligned}
        \right.
    \end{equation}
    with $C > 0$ a constant depending on the volume of the domain and the shape of the holes but not $\varepsilon$, 
    $\Keo = \sum_{k=1}^{\ndoors} \Kek$
    and~$\mathcal{E}_1$ a nonnegative function of the family $\Kall$ such that~$\mathcal{E}_1(\Kall) \to 0$ as $\varepsilon \to 0$.

    Furthermore, 
    for all $k \in \{1, \ldots, N\}$,
    $\psi_\varepsilon^{(k)}$ verifies the following normalization condition  
    \begin{equation}\label{eq:hyp:norm}
        \tag{Norm}
        \|\psi_\varepsilon^{(k)}\|_{L^2(\Omega_\varepsilon)} =  \mathrm{O}(\mathcal{E}_2(\Kall)),
    \end{equation}
    with $\mathcal{E}_2$ a nonnegative function of the family $\Kall$ such that $\mathcal{E}_2(\Kall) \to 0$ as $\varepsilon \to 0$.
\end{hypo}
The two functions $\mathcal{E}_1$ and $\mathcal{E}_2$ are used to quantify the error of the quasi-mode,
and determine the error in the final result of \cref{thm:narrow escape}.
They depend on the whole family $\Kall$, and not only on $\Keo$,
as the error may depend on the relative sizes of the holes (see \cref{sec:tubular holes}).
The Dirichlet condition of \eqref{eq:hyp:PDE quasi-mode} states that 
$\psi_\varepsilon^{(k)}$ is close to $1$ on $\Gke$, 
and that its contribution is negligible on the other holes.

Summing \eqref{eq:hyp:PDE quasi-mode} and \eqref{eq:hyp:norm} over $k$,
and since $\sum_{k=1}^{\ndoors} \mathbbm{1}_{\overline{\Gke}} = 1$ on $\overline{\Gde}$,
the quasi-mode $\varphi_\varepsilon$ defined in~\eqref{eq:def quasi-mode} satisfies
\begin{equation}\label{eq:quasi-mode properties}
    \left\{
        \begin{aligned}
        -\Delta \varphi_\varepsilon & = C \Keo && \text{ in } \Omega_\varepsilon, \\
        \partial_n \varphi_\varepsilon & = 0 && \text{ on } \Gne,\\
        \|\varphi_\varepsilon\|_{L^\infty(\Gde)} & = \mathrm{O}(\mathcal{E}_1(\Kall)),
        \end{aligned}
    \right.
    \qquad \text{ and } \qquad
    \|\varphi_\varepsilon - 1\|_{L^2(\Omega_\varepsilon)} = \mathrm{O}(\mathcal{E}_2(\Kall)).
\end{equation}
Qualitatively,
these hypotheses and this method in general show that 
$\varphi_\varepsilon \simeq |\Omega_\varepsilon| v_\varepsilon$, 
which can be seen from the similarity between \eqref{eq:edp density} and \eqref{eq:quasi-mode properties}.
From the form of the quasi-mode $\varphi_\varepsilon$ in \eqref{eq:def quasi-mode},
one expects that the quasi-stationary distribution $\nu_\varepsilon$ is close to being uniform in $\Omega_\varepsilon$
at small $\varepsilon$.
 
We can now state the key result of this work. 

\begin{thm}\label{thm:narrow escape}
    Consider $(X_t)_{t \geq 0}$, the process introduced in \eqref{eq:reflected bm}. 
    If \cref{hyp:quasi-mode} holds, 
    then the exit time $\tau_\varepsilon$ starting from $\nu_\varepsilon$
    is exponentially distributed with parameter $\lambda_\varepsilon$, and
    \[
        \mathbb{E}_{\nu_\varepsilon}[\tau_\varepsilon] = \frac{1}{\lambda_\varepsilon} = \frac{1}{C \Keo} \left(1 + \mathrm{O}(\mathcal{E}_1(\Kall) + \mathcal{E}_2(\Kall))\right).
    \]
    The exit position is independent of the exit time, and is distributed as, for $k \in \{ 1, \ldots, N \}$
    \[
        \mathbb{P}_{\nu_\varepsilon}(\kappa_\varepsilon = k) = \frac{\Kek}{\Keo} +  \mathrm{O}(\mathcal{E}_1(\Kall) + \mathcal{E}_2(\Kall))
    \]
    where $\Kek$ is given by \cref{hyp:quasi-mode}
    and $\kappa_\varepsilon$ is the index of the exit hole,
    i.e., $\kappa_\varepsilon = k$ if and only if~$X_{\tau_\varepsilon} \in \overline{\Gke}$.
\end{thm}
The proof of this theorem can be performed in a similar manner as in \cite[Theorems 1.4 and 1.6]{nous};
for the sake of completeness, it is given in \cref{app:proof}. 
This proof also relies on the Poincaré--Wirtinger inequality uniform in $\varepsilon$,
which is proven in \cref{app:poincare tube} for the different geometries considered in this work.
Let us finally insist that the techniques developed in this work are limited to the case where a finite number of holes 
are shrinking without moving, 
hence they are well separated in the limit $\varepsilon \to 0$.

\section{Spherical hole shapes}
\label{sec:several narrow escape problems}
In this section, 
we apply \cref{thm:narrow escape} to two different settings: \eqref{eq:spherical holes} and \eqref{eq:geodesic spherical holes},
by constructing a quasi-mode that verifies~\cref{hyp:quasi-mode}, 
to obtain the results claimed in \cref{tab:summary}.
For both of them, 
the main ingredient is the Green function $G$ of the Laplacian on $\Omega$ 
with Neumann boundary condition (see \cref{sec:app:result Green} for results on $G$).
We focus on \eqref{eq:spherical holes} in \cref{sub:sphericalinside},
and then \eqref{eq:geodesic spherical holes} in \cref{sub:spherical}.

\subsection{Spherical holes inside the domain with a center on the boundary}
\label{sub:sphericalinside} 
In this subsection, 
we recall what was done in \cite{nous} for the case of holes defined in \eqref{eq:spherical holes}
for the sake of the comparison.
The authors of \cite{nous} constructed a quasi-mode $\varphi_\varepsilon$ of the form \eqref{eq:def quasi-mode} with
\[
    \psi_\varepsilon^{(k)}(x) = 2 \pi \rek G(x, \xk), \qquad k \in \{1, \ldots, N\},
\]
which verify \cref{hyp:quasi-mode} in this setting with 
$\Kek = \rek$ (so that $\Keo = \reo = \sum_{k=1}^{\ndoors} \rek$), $C = \frac{2\pi}{|\Omega|}$, $\mathcal{E}_1(\Kall) = \Keo |\log(\Keo)|$ and $\mathcal{E}_2(\Kall) = \Keo$ (see \cite[Section 3.1]{nous}).
The continuity on $\overline{\Gde}$ required in \cref{hyp:quasi-mode} is immediate,
as $G(\cdot, \xk)$ is smooth on $\overline{\Omega} \setminus \{\xk\}$ and $\xk \notin \overline{\Gde}$.
They then obtained the following results, which we recall for the sake of comparison.
\begin{prop}\label{prop:narrow escape spherical inside}
    Consider the reflected Brownian motion $(X_t)_{t \geq 0}$ on $\Omega_\varepsilon$ killed on the holes defined in \eqref{eq:spherical holes}, 
    then the mean exit time is given by 
    \begin{equation}\label{eq:lambda epsilon pizza}
        \mathbb{E}_{\nu_\varepsilon}[\tau_\varepsilon] = \frac{|\Omega|}{ 2\pi } \frac{1}{\reo}\Bigl(1 + \mathrm{O}\left(\mathcal{E}(\reo)\right)\Bigr),
    \end{equation}
    with $\mathcal{E} : x \mapsto x |\log(x)|$.
    The exit position is independent of the exit time, and is distributed as, for~$k \in \{ 1, \ldots, N \}$
    \begin{equation}\label{eq:exit position pizza}
        \mathbb{P}_{\nu_\varepsilon}(\kappa_\varepsilon = k) = \frac{\rek}{\reo} +  \mathrm{O}\left(\mathcal{E}(\reo)\right).
    \end{equation}
\end{prop}

\subsection{Spherical holes on the boundary of the domain}
\label{sub:spherical}
In this subsection, we consider holes of type \eqref{eq:geodesic spherical holes}. 
Recall that $d_{\partial \Omega}$ denotes the geodesic distance on $\partial \Omega$. 
Notice that with this hole shape, fully contained in the boundary,
the domain is unchanged: $\Omega_\varepsilon = \Omega$ for all~${\varepsilon > 0}$.

Let us start checking \cref{hyp:quasi-mode} by giving a simple proof that the eigenvalue $\lambda_\varepsilon$ tends to $0$. 
\begin{lem}\label{lem:eigenvalue spherical holes}
    Consider the operator $\mathcal{L}_\varepsilon$ defined on $\Omega_\varepsilon$ with the Dirichlet boundary condition defined in \eqref{eq:geodesic spherical holes},
    then its principal eigenvalue $\lambda_\varepsilon$ tends to $0$ as $\varepsilon \to 0$.
\end{lem}

\begin{proof}
    As $\lambda_\varepsilon$ is the principal eigenvalue of $\mathcal{L}_\varepsilon$,
    it minimizes the Rayleigh quotient by the min-max principle:
    \begin{equation*}
        \lambda_\varepsilon = \min_{u \in H^1_{\Gde}(\Omega_\varepsilon) \setminus \{0\}} \frac{\int_{\Omega_\varepsilon} |\nabla u|^2}{\int_{\Omega_\varepsilon} |u|^2}
,
    \end{equation*}
    where $H^1_{\Gde}(\Omega_\varepsilon)$ is the domain of $\mathfrak{q}_\varepsilon$, the quadratic form associated with $\mathcal{L}_\varepsilon$ defined in \eqref{eq:def bilinear form}.
    This shows that $\lambda_\varepsilon$ is nonnegative.
    For every $k \in \{1, \ldots, N\}$, 
    as $\partial \Omega$ is smooth, 
    for all $x$ in a neighborhood of $\xk$, it holds
    \[
        |x-\xk| \leq d_{\partial \Omega}(x, \xk).
    \]
    Therefore, 
    one has that $\Gke \subset \overline{B(\xk, \rek)} \cap \partial \Omega$.
    Consider now the space of the $H^1(\Omega_\varepsilon)$ 
    functions that vanish on half-balls inside the domain $H^1_{\mathcal{B}_\varepsilon}(\Omega_\varepsilon)$ with
    \[
        \mathcal{B}_\varepsilon = \bigcup_{k=1}^{\ndoors} B(\xk, \rek) \cap \Omega_\varepsilon.
    \]
    It is clear, by extending the functions by $0$, 
    that $H^1_{\mathcal{B}_\varepsilon}(\Omega_\varepsilon) \subset H^1_{\Gde}(\Omega_\varepsilon)$,
    therefore, 
    \begin{equation}\label{eq:rayleigh spherical holes}
        \lambda_\varepsilon 
        \leq \min_{u \in H^1_{\mathcal{B}_\varepsilon}(\Omega_\varepsilon) \setminus \{0\}} \frac{\int_{\Omega_\varepsilon} |\nabla u|^2}{\int_{\Omega_\varepsilon} |u|^2}
        = \min_{u \in H^1_{\mathcal{B}_\varepsilon}(\Omega_\varepsilon) \setminus \{0\}} \frac{\int_{\Omega_\varepsilon \setminus \mathcal{B}_\varepsilon} |\nabla u|^2}{\int_{\Omega_\varepsilon \setminus \mathcal{B}_\varepsilon} |u|^2}.
    \end{equation}
    Notice that the right-hand side of \eqref{eq:rayleigh spherical holes} 
    is exactly the Rayleigh quotient of the eigenvalue associated with the problem with 
    \eqref{eq:spherical holes} holes, described in \cref{sub:sphericalinside}.
    In \cref{prop:narrow escape spherical inside},
    we showed that the eigenvalue of the \eqref{eq:spherical holes} tends to $0$ as~$\varepsilon \to 0$,
    therefore, 
    it is also the case for $\lambda_\varepsilon$ here considered as it is nonnegative.

\end{proof}

We now move on to building the quasi-mode $\varphi_\varepsilon$ that verifies \cref{hyp:quasi-mode} in this setting, 
using again the Green function $G$ of the Laplacian on $\Omega$ with Neumann boundary condition.
For $k \in \{1, \ldots, N\}$, 
denote 
\[
    \rho_\varepsilon^{(k)} \colon y \in \Gke \mapsto \Bigl( \left( \rek\right)^2 - d_{\partial \Omega}(y, \xk)^2 \Bigr)^{-\frac{1}{2}},
\]
and 
\begin{equation}\label{eq:quasi-mode flat holes}
    \psi_\varepsilon^{(k)}(x) = \frac{2}{\pi} \int_{\Gke} \rho_\varepsilon^{(k)}(y) G(x, y)\d \sigma(y), 
    \qquad k \in \{1, \ldots, N\}.
\end{equation}

We can now show that \cref{hyp:quasi-mode} holds for $\psi_\varepsilon^{(k)}$, 
and derive the main results from this section.
\begin{prop}\label{prop:quasi-mode spherical holes}
    The functions $\psi_\varepsilon^{(k)}$ defined in \eqref{eq:quasi-mode flat holes} verify \cref{hyp:quasi-mode} with
    \[
        \Kek = \rek, \qquad \Keo = \reo, \qquad C = \frac{4}{|\Omega|}, \qquad \mathcal{E}_1(\Kall) = \Keo|\log(\Keo)|, \qquad \mathcal{E}_2(\Kall) = \Keo.
    \]
    Furthermore, 
    the mean exit time and exit position of the reflected Brownian motion $(X_t)_{t \geq 0}$ on $\Omega_\varepsilon$ 
    killed on the holes defined in~\eqref{eq:geodesic spherical holes} are given by
    \begin{equation}\label{eq:lambda epsilon spherical holes}
        \mathbb{E}_{\nu_\varepsilon}[\tau_\varepsilon] = \frac{|\Omega|}{ 4 } \frac{1}{\reo}\Bigl(1 + \mathrm{O}\left(\mathcal{E}(\reo)\right)\Bigr),
    \end{equation}
    with $\mathcal{E} : x \mapsto x|\log(x)|$.
    The exit position is independent of the exit time, and is distributed as, for~$k \in \{ 1, \ldots, N \}$
    \begin{equation}\label{eq:exit position spherical holes}
        \mathbb{P}_{\nu_\varepsilon}(\kappa_\varepsilon = k) = \frac{\rek}{\reo} +  \mathrm{O}\left(\mathcal{E}(\reo)\right).
    \end{equation}
\end{prop}

\begin{proof}
    In the light of \cref{lem:eigenvalue spherical holes}, 
    we only need to show that the functions $\psi_\varepsilon^{(k)}$ verify \cref{hyp:quasi-mode} to show the results on the mean exit time and exit position.
    Fix $k \in \{ 1, \ldots, N\}$.

    \paragraph{PDE.}
    The proof relies on the results from \cref{lem:Green solution}.
    Indeed,
    by the (Integrability) estimate of \cref{lem:integrals circular},
    $\rho_\varepsilon^{(k)} \in L^{4/3}(\Gke)$, although $\rho_\varepsilon^{(k)} \notin L^2(\Gke)$.
    Therefore, by \cref{lem:Green solution} with $p = 4/3$,
    $\psi_\varepsilon^{(k)}$ belongs to $H^1(\Omega)$ and is the weak solution of the following PDE:
    \begin{equation}\label{eq:PDE quasi-mode flat holes}
        \left\{
            \begin{aligned}
                \Delta \psi_\varepsilon^{(k)} & = \frac{2}{\pi|\Omega|} \int_{\Gke} \rho_\varepsilon^{(k)} \d \sigma(y) && \text{ in } \Omega, \\
                \partial_n \psi_\varepsilon^{(k)}(x) & = \frac{2}{\pi} \rho_\varepsilon^{(k)}(x) \mathbbm{1}_{\Gke}(x) && \text{ for } x \in \partial \Omega.\\
            \end{aligned}
        \right.
    \end{equation}
    It is clear from \eqref{eq:PDE quasi-mode flat holes} that $\partial_n \psi_\varepsilon^{(k)} = 0$ on $\Gne$ as $\Gne \cap \Gde = \varnothing$.
    Moreover, 
    notice that the following integral has been computed in \cref{lem:integrals circular}, and
    \[
        \int_{\Gke} \rho_\varepsilon^{(k)} \d \sigma(y) = 2\pi \rek,
        \qquad \text{ so that } \qquad 
        \Delta \psi_\varepsilon^{(k)} = \frac{4}{|\Omega|} \rek.
    \]

    \paragraph{Continuity on $\overline{\Gde}$.}
    By \cref{lem:continuity flat holes}, $\psi_\varepsilon^{(k)}$ is continuous on $\overline{\Gde}$,
    so that $\psi_\varepsilon^{(k)} \in H^1(\Omega) \cap C^0(\overline{\Gde})$, as required by \cref{hyp:quasi-mode}.

    \paragraph{Dirichlet estimate.}
    Take $x \in \Gke$,
    then, 
    using \cref{lem:Green scaling exact} for $\rek$ small enough, 
    one can decompose $\psi_\varepsilon^{(k)}$ as
    \begin{equation}\label{eq:varphi expansion spherical holes}
        \psi_\varepsilon^{(k)}(x) = \frac{2}{\pi} \int_{\Gke} \rho_\varepsilon^{(k)} \frac{1}{2\pi |x-y|}\d \sigma(y)
        + \frac{2}{\pi} \int_{\Gke} \rho_\varepsilon^{(k)} g(x, y)\d \sigma(y),
    \end{equation}
    where $g(x, y)$ is bounded by $\mathrm{O}(|\log|x - y||)$ in the limit $|x-y| \to 0$.
    Both integrals in \eqref{eq:varphi expansion spherical holes} have been computed in \cref{lem:integrals circular}. 
    Indeed, by the (Newtonian potential) estimate from \cref{lem:integrals circular},
    \begin{equation}\label{eq:newton potential estimate spherical}
        \frac{2}{\pi} \int_{\Gke} \rho_\varepsilon^{(k)} \frac{1}{2\pi |x-y|}\d \sigma(y) 
        = \frac{2}{\pi} \frac{\pi}{2} + \mathrm{O}(\rek).
    \end{equation}
    By the (Logarithmic potential) estimate from \cref{lem:integrals circular}, 
    \begin{equation}\label{eq:log potential estimate spherical}
        \left| \int_{\Gke} \rho_\varepsilon^{(k)} g(x, y)\d \sigma(y) \right|
        \leq C \int_{\Gke} \rho_\varepsilon^{(k)} \left| \log|x-y|\right|\d \sigma(y) 
        = \mathrm{O}(\rek|\log(\rek)|).
    \end{equation}
    Combining \eqref{eq:varphi expansion spherical holes}, \eqref{eq:newton potential estimate spherical} and \eqref{eq:log potential estimate spherical}
    leads to 
    \[
        \psi_\varepsilon^{(k)}(x) = 1 + \mathrm{O}(\rek) + \mathrm{O}(\rek|\log(\rek)|) = 1 + \mathrm{O}(\rek |\log(\rek)|), \qquad \text{for all } x \in \Gke.
    \]
    Now take $x \in \Gamma_{j}^\varepsilon$ for $j \neq k$.
    As the holes are well separated,
    $G(x, \cdot)$ is bounded on $\Gke$ uniformly in $x$ and~$\varepsilon$,
    so that, by the (mass) integral from \cref{lem:integrals circular}, 
    $\psi_\varepsilon^{(k)}(x) = \mathrm{O}(\rek)$.
    Since $x \mapsto x |\log(x)|$ is increasing on $(0, \e^{-1})$ and $\rek \leq \reo$, 
    this shows that 
    \[
        \bigl\|\psi_\varepsilon^{(k)} - \mathbbm{1}_{\overline{\Gke}}\bigr\|_{L^\infty(\Gde)} = \mathrm{O}(\reo |\log(\reo)|),
    \]
    which is the Dirichlet condition in \eqref{eq:hyp:PDE quasi-mode} with $\mathcal{E}_1(\Kall) = \reo|\log(\reo)|$.

    \paragraph{Normalization.}
    Using the Cauchy--Schwarz inequality between $\sqrt{\rho_\varepsilon^{(k)}}$ and $\sqrt{\rho_\varepsilon^{(k)}} G(x, \cdot)$,
    one has that, for~$x \in \Omega$, 
    \begin{equation}\label{eq:cauchy-schwarz}
        \Bigl( \int_{\Gke} \rho_\varepsilon^{(k)}(y) G(x,y) \, \d \sigma(y) \Bigr)^2 \leq \Bigl( \int_{\Gke} \rho_\varepsilon^{(k)}(y) \, \d \sigma(y) \Bigr) \Bigl( \int_{\Gke} \rho_\varepsilon^{(k)}(y) \, G(x,y)^2 \, \d \sigma(y) \Bigr).
    \end{equation}
    It is clear by (mass) from \cref{lem:integrals circular} that the first integral 
    in the right-hand side is $\mathrm{O}(\rek)$ in the limit $\varepsilon \to 0$.
    Integrating over $x \in \Omega$ and using Fubini--Tonelli, 
    \begin{equation}\label{eq:fubini}
        \int_{\Omega} \int_{\Gke} \rho_\varepsilon^{(k)}(y) \, G(x,y)^2 \, \d \sigma(y) \, \d x 
        = \int_{\Gke} \rho_\varepsilon^{(k)}(y) \Bigl( \int_\Omega G(x,y)^2 \, \d x \Bigr) \d \sigma(y).
    \end{equation}
    Using \eqref{eq:cauchy-schwarz}, \eqref{eq:fubini}, \cref{lem:Green L2 bound} and (mass) from \cref{lem:integrals circular}, 
    one has that there exists a constant $C > 0$ such that 
    \begin{equation}\label{eq:single hole L2 bound}
        \bigl\| \psi_\varepsilon^{(k)} \bigr\|_{L^2(\Omega)}^2
        = \frac{4}{\pi^2} \int_\Omega \Bigl( \int_{\Gke} \rho_\varepsilon^{(k)}(y) G(x,y) \, \d\sigma(y) \Bigr)^2 \d x
        \leq C \left( \rek\right)^2,
    \end{equation}
    so that $\|\psi_\varepsilon^{(k)}\|_{L^2(\Omega)} = \mathrm{O}(\reo)$.
    This shows that \eqref{eq:hyp:norm} is satisfied with $\mathcal{E}_2(\Kall) = \reo$.
\end{proof}

\section{Slit-like holes}
\label{sec:slit}
In this section, 
we consider the case of slit-like holes defined in \eqref{eq:slit def},
whose half-width $w_\varepsilon^{(k)}$ is the natural small parameter of this section.

In particular, the foci of these ellipsoids are located at the tips $s_x + t_\xk = \pm \ell_k$ of the slit,
which makes the computations explicit (see \cref{lem:computation slit}).
Finally, notice that the underlying domain $\Omega_\varepsilon$ is then modified as
\[
    \Omega_\varepsilon = \Omega \setminus \bigcup_{k=1}^{\ndoors}
    \left\{ x \in \overline{\Omega} \, \left| \, \frac{(s_x + t_\xk)^2}{(a_\varepsilon^{(k)})^2} + \frac{|\rho_\xk|^2}{(w_\varepsilon^{(k)})^2} \leq 1\right.\right\}.
\]

We are now in a position to introduce a quasi-mode $\varphi_\varepsilon$ that verifies \cref{hyp:quasi-mode} in this setting.
For $k \in \{1, \ldots, N\}$, 
denote
\begin{equation}\label{eq:quasi-mode-slit}
    \psi_\varepsilon^{(k)}(x) = \frac{\pi}{|\log w_\varepsilon^{(k)}|} \int_{\Hk} G(x, y) \, \d y, \qquad x \in \Omega,
\end{equation}
where $G$ is, as in \cref{sec:several narrow escape problems}, the Green function of the Laplacian on $\Omega$ 
with Neumann boundary condition of \cref{def:Neumann Green}.
The construction of $\psi_\varepsilon^{(k)}$ is very similar to the one used in \cref{sub:spherical}:
a single-layer potential built from $G$.
The main difference is that the layer is here supported on the curves $\Hk$
(with their arc-length measure $\d y$),
instead of the two-dimensional patches $\Gke$ of \cref{sub:spherical}.
This drastically changes the behavior of the singularity of the quasi-mode close to the holes.

Let us first check that the functions $\psi_\varepsilon^{(k)}$ verify the \eqref{eq:hyp:PDE quasi-mode} and \eqref{eq:hyp:norm} conditions of \cref{hyp:quasi-mode}.
\begin{lem}\label{lem:edp quasi-mode slit}
    The functions $\psi_\varepsilon^{(k)}$ defined in \eqref{eq:quasi-mode-slit} verify \eqref{eq:hyp:PDE quasi-mode} and \eqref{eq:hyp:norm} of \cref{hyp:quasi-mode} with
    \[
        \Kek = \frac{\ell_k}{|\log w_\varepsilon^{(k)}|}, \qquad C = \frac{2\pi}{|\Omega|}, \qquad \mathcal{E}_1(\Kall) = \mathcal{E}_2(\Kall) = \Keo.
    \]
\end{lem}

\begin{proof}
    Fix $k \in \{1, \ldots, N\}$.

    \paragraph{PDE.}
    Fix $\varepsilon > 0$.
    By construction of $\Omega_\varepsilon$,
    the closure $\overline{\Omega_\varepsilon}$ stays at a positive distance from the slit:
    in the coordinates \eqref{eq:decomposition X}, every point of the half-spheroid $\Gke$
    is at a positive distance from $\Hk$,
    bounded below by some constant depending only on $w_\varepsilon^{(k)}$ and the geometry of $\Hk$, so that
    \[
        \delta_\varepsilon = \mathrm{dist}\bigl(\overline{\Omega_\varepsilon}, \Hk\bigr) > 0 .
    \]
    By \cref{lem:Green David} (extended by continuity to $\overline{\Omega_\varepsilon}$),
    $|G(x, y)| \leq K \delta_\varepsilon^{-1}$ and $|\nabla_x G(x, y)| \leq K \delta_\varepsilon^{-2}$
    for all $x \in \overline{\Omega_\varepsilon}$ and $y \in \Hk$.
    Hence, by dominated convergence, $\psi_\varepsilon^{(k)}$ is continuous on $\overline{\Omega_\varepsilon}$,
    and, differentiating under the integral sign in \eqref{eq:quasi-mode-slit}, $\psi_\varepsilon^{(k)} \in W^{1, \infty}(\Omega_\varepsilon) \subset H^1(\Omega_\varepsilon)$.
    Moreover, for each $y \in \Hk$, $G(\cdot, y)$ is smooth on $\overline{\Omega_\varepsilon}$,
    with $\Delta_x G(\cdot, y) = 1/|\Omega|$ in $\Omega_\varepsilon$ and $\partial_n G(\cdot, y) = 0$ on $\Gne$ by \cref{def:Neumann Green}.
    Integrating these relations in $y \in \Hk$ against test functions, with the Fubini theorem justified by the above bounds,
    shows that $\psi_\varepsilon^{(k)} \in H^1(\Omega_\varepsilon) \cap C^0(\overline{\Gde})$ satisfies, in the weak sense,
    \[
        \Delta \psi_\varepsilon^{(k)}
        = \frac{\pi}{|\log w_\varepsilon^{(k)}|} \int_{\Hk} \frac{1}{|\Omega|} \, \d y
        = \frac{\pi}{|\log w_\varepsilon^{(k)}|} \, \frac{2 \ell_k}{|\Omega|}
        = C \Kek
        \quad \text{in } \Omega_\varepsilon,
        \qquad
        \partial_n \psi_\varepsilon^{(k)} = 0
        \quad \text{on } \Gne,
    \]
    where we used $\int_{\Hk} \d y = 2 \ell_k$, and $\Gne \subset \partial \Omega \setminus \Hk$ for the Neumann condition.

    \paragraph{Dirichlet estimate.}
    The proof of this estimate is postponed to \cref{lem:computation slit} in the appendix below, 
    where we show that
    \begin{equation}\label{eq:slit integral}
        \int_{\mathcal{H}_k} G(x,y) \, \d y = \frac{1}{\pi} |\log w_\varepsilon^{(k)}| + \mathrm{O}(1), \qquad x \in \Gke,
    \end{equation}
    so that, by the very choice of the prefactor $\pi / |\log w_\varepsilon^{(k)}|$ in \eqref{eq:quasi-mode-slit},
    \begin{equation}\label{eq:slit dirichlet 1}
        \psi_\varepsilon^{(k)}(x) = 1 + \mathrm{O}\left( \frac{1}{|\log w_\varepsilon^{(k)}|}\right), \qquad x \in \Gke.
    \end{equation}
    As the curves $\mathcal{H}_j$ are pairwise disjoint, hence at a positive distance from one another that does not vanish as $\varepsilon \to 0$,
    for $j \neq k$ the integral in \eqref{eq:slit integral} is bounded uniformly for $x \in \Gamma_j^\varepsilon$.
    Therefore,
    \begin{equation}\label{eq:slit dirichlet 2}
        \psi_\varepsilon^{(k)}(x) = \mathrm{O}\left( \frac{1}{|\log w_\varepsilon^{(k)}|}\right), \qquad x \in \Gde \setminus \Gke.
    \end{equation}
    The Dirichlet estimate follows from \eqref{eq:slit dirichlet 1} and \eqref{eq:slit dirichlet 2}.

    \paragraph{Normalization.} 
    Notice from \cref{lem:Green L2 bound} that 
    \[
        \sup_{y \in \partial \Omega} \int_\Omega G(x,y)^2 \, \d x \leq M < +\infty,
    \]
    where $M$ is a constant independent of $\varepsilon$.
    Thus, 
    the $L^2$ norm of $\psi_\varepsilon^{(k)}$ satisfies
    \begin{equation}\label{eq:psi L2 norm slit}
        \|\psi_\varepsilon^{(k)}\|_{L^2(\Omega_\varepsilon)} \leq \frac{\pi}{|\log w_\varepsilon^{(k)}|} 
        \left( \int_{\Omega_\varepsilon} \left( \int_{\Hk} G(x,y) \, \d y \right)^2 \, \d x \right)^{\frac{1}{2}}.
    \end{equation}
    Using the Cauchy--Schwarz inequality, the inclusion $\Omega_\varepsilon \subset \Omega$, and Fubini--Tonelli, we can bound the integral in \eqref{eq:psi L2 norm slit} 
    by constants independent of $\varepsilon$.
    Hence,
    \[
        \|\psi_\varepsilon^{(k)}\|_{L^2(\Omega_\varepsilon)} = \mathrm{O}\left( \frac{1}{|\log w_\varepsilon^{(k)}|} \right) = \mathrm{O}(\Keo),
    \]
    which is \eqref{eq:hyp:norm} with~$\mathcal{E}_2(\Kall) = \Keo$.
    By the Cauchy--Schwarz inequality applied to $y \mapsto 1$ and $y \mapsto G(x, y)$ on $\Hk$,
    for $x \in \Omega_\varepsilon$,
    \[
        \left( \int_{\Hk} G(x,y) \, \d y \right)^2 \leq |\Hk| \int_{\Hk} G(x,y)^2 \, \d y = 2 \ell_k \int_{\Hk} G(x,y)^2 \, \d y.
    \]
    Integrating over $x \in \Omega_\varepsilon \subset \Omega$,
    using Fubini--Tonelli and \cref{lem:Green L2 bound},
    \[
        \int_{\Omega_\varepsilon} \left( \int_{\Hk} G(x,y) \, \d y \right)^2 \, \d x \leq 2 \ell_k \int_{\Hk} \Bigl( \int_\Omega G(x,y)^2 \, \d x \Bigr) \d y \leq 2 \ell_k \int_{\Hk} M \, \d y = 4 M \ell_k^2,
    \]
    so that $\|u_k\|_{L^2(\Omega)} = \mathrm{O}(1)$,
    a constant independent of $\varepsilon$ -- as could be expected, since $u_k$ itself does not depend on $\varepsilon$, only the prefactor in front of it does.
    Hence,
    \[
        \|\psi_\varepsilon^{(k)}\|_{L^2(\Omega_\varepsilon)} \leq \frac{\pi}{|\log w_\varepsilon^{(k)}|} \|u_k\|_{L^2(\Omega)} = \mathrm{O}\left( \frac{1}{|\log w_\varepsilon^{(k)}|} \right) = \mathrm{O}(\Keo),
    \]
    which is \eqref{eq:hyp:norm} with~$\mathcal{E}_2(\Kall) = \Keo$.
\end{proof}

It remains to check that the eigenvalue $\lambda_\varepsilon$ tends to $0$, and we can then conclude with the main result of this section.

\begin{prop}\label{prop:narrow escape slit}
    Consider the reflected Brownian motion $(X_t)_{t \geq 0}$ on $\Omega_\varepsilon$ killed on the holes defined in~\eqref{eq:slit def},
    then, in the limit as $\varepsilon \to 0$, the mean exit time is given by
    \begin{equation}\label{eq:lambda epsilon slit}
        \mathbb{E}_{\nu_\varepsilon}[\tau_\varepsilon] = \frac{|\Omega|}{2\pi} \frac{1}{\Keo} \Bigl( 1 + \mathrm{O}(\Keo) \Bigr).
    \end{equation}
    Then the exit position is independent of the exit time, and is distributed as, for $k \in \{ 1, \ldots, N\}$,
    \begin{equation}\label{eq:exit position slit}
        \mathbb{P}_{\nu_\varepsilon}(\kappa_\varepsilon = k) = \frac{\Kek}{\Keo} + \mathrm{O}(\Keo).
    \end{equation}
\end{prop}
\begin{proof}
    In view of \cref{lem:edp quasi-mode slit}, it remains to prove that $\lambda_\varepsilon \to 0$ as $\varepsilon \to 0$ to be able to apply \cref{thm:narrow escape}.
    As in \cref{lem:eigenvalue spherical holes},
    we rely on the Rayleigh quotient
    \begin{equation}\label{eq:rayleigh quotient slit}
        \lambda_\varepsilon = \min_{u \in H^1_{\Gde}(\Omega_\varepsilon) \setminus \{0\}} \frac{\int_{\Omega_\varepsilon} |\nabla u|^2}{\int_{\Omega_\varepsilon} |u|^2}.
    \end{equation}
    As $\varphi_\varepsilon$, 
    the quasi-mode defined in \eqref{eq:def quasi-mode} using the $(\psi_\varepsilon^{(k)})_{k=1}^N$ of \eqref{eq:quasi-mode-slit},
    belongs to $H^1(\Omega_\varepsilon)$ but does not vanish on $\Gde$,
    we cannot use it directly as a trial function;
    we replicate here the techniques used in \cite[Proof of item (ii) of Theorem 1.2]{Tony2024},
    and consider, for $\alpha > 0$,
    \[
        \varphi_\varepsilon^\alpha(x) = 1 - \alpha \sum_{k=1}^N \psi_\varepsilon^{(k)}(x).
    \]
    By the Dirichlet estimate of \cref{lem:edp quasi-mode slit},
    $\sum_{k=1}^N \psi_\varepsilon^{(k)} = 1 + \mathrm{O}(\Keo)$ on $\Gde$, so
    \[
        \varphi_\varepsilon^\alpha = (1-\alpha) + \alpha \, \mathrm{O}(\Keo), \qquad \text{ uniformly on } \Gde.
    \]
    Hence, for a fixed $\alpha^* > 1$ (e.g. $\alpha^*=2$), and $\varepsilon$ small enough, $\varphi_\varepsilon^{\alpha^*}$ is (strictly) negative on $\Gde$.
    Denote~$\phi_\varepsilon^{\alpha^*} = \bigl(\varphi_\varepsilon^{\alpha^*}\bigr)_+$ the positive part of $\varphi_\varepsilon^{\alpha^*}$:
    as it vanishes on $\Gde$, it belongs to the form domain $H^1_{\Gde}(\Omega_\varepsilon)$.

    By the same computation as in the normalization part of the proof of \cref{lem:edp quasi-mode slit} (with the extra factor $\alpha^*$),
    \begin{equation}\label{eq:alpha qm slit l2}
        \|\phi_\varepsilon^{\alpha^*}\|_{L^2(\Omega_\varepsilon)}^2 = |\Omega_\varepsilon| + \mathrm{O}(\Keo).
    \end{equation}
    For the gradient, using $\nabla \varphi_\varepsilon^{\alpha^*} = -\alpha^* \sum_{k=1}^N \nabla \psi_\varepsilon^{(k)}$,
    the discrete Cauchy--Schwarz inequality
    and \cref{lem:Green David} to bound $|\nabla_x G(x,y)| \leq K |x-y|^{-2}$,
    \[
        \int_{\Omega_\varepsilon} |\nabla \phi_\varepsilon^{\alpha^*}|^2
        \leq \int_{\Omega_\varepsilon} |\nabla \varphi_\varepsilon^{\alpha^*}|^2
        \leq N \bigl(\alpha^*\bigr)^2 \sum_{k=1}^N \frac{\pi^2}{|\log w_\varepsilon^{(k)}|^2} \int_{\Omega_\varepsilon} \Bigl( \int_{\Hk} K |x - y|^{-2} \, \d y \Bigr)^2 \, \d x.
    \]
    Fix $k$.
    Since $\Hk$ is a smooth embedded arc, $\int_{\Hk} |x-y|^{-2} \, \d y \leq C \, d(x, \Hk)^{-1}$.
    The points of the spheroid closest to $\Hk$ are its tips,
    at distance $a_\varepsilon^{(k)} - \ell_k = (w_\varepsilon^{(k)})^2 / (a_\varepsilon^{(k)} + \ell_k)$ by \eqref{eq:slit major axis},
    so that $d(x, \Hk) \geq c \, (w_\varepsilon^{(k)})^2$ for $x \in \Omega_\varepsilon$.
    As $|\{ x \in \real^3 \, | \, d(x, \Hk) < t \}| \leq C t^2$, it holds that
    \[
        \int_{\Omega_\varepsilon} \Bigl( \int_{\Hk} K |x-y|^{-2} \, \d y \Bigr)^2 \, \d x
        \leq C \int_{\Omega_\varepsilon} d(x, \Hk)^{-2} \, \d x
        \leq C \int_{c \, (w_\varepsilon^{(k)})^2}^{r_0} \frac{\d t}{t} + C
        = \mathrm{O}\bigl( |\log w_\varepsilon^{(k)}| \bigr),
    \]
    for some fixed $r_0 > 0$.
    Plugging this bound into the previous display,
    \begin{equation}\label{eq:nabla slit}
        \int_{\Omega_\varepsilon} |\nabla \phi_\varepsilon^{\alpha^*}|^2
        = \mathrm{O}\left( \sum_{k=1}^N \frac{|\log w_\varepsilon^{(k)}|}{|\log w_\varepsilon^{(k)}|^2} \right)
        = \mathrm{O}(\Keo).
    \end{equation}
    As $\phi_\varepsilon^{\alpha^*}$ belongs to the form domain $H^1_{\Gde}(\Omega_\varepsilon)$,
    plugging \eqref{eq:alpha qm slit l2} and \eqref{eq:nabla slit} into \eqref{eq:rayleigh quotient slit} shows that
    \[
        \lambda_\varepsilon \leq \frac{\mathrm{O}(\Keo)}{1+\mathrm{O}(\Keo)} = \mathrm{O}(\Keo) \underset{\varepsilon \to 0}{\longrightarrow} 0.
    \]
    Having checked both hypotheses \eqref{eq:hyp:eig} and \eqref{eq:hyp:PDE quasi-mode}--\eqref{eq:hyp:norm} of \cref{hyp:quasi-mode},
    \cref{thm:narrow escape} together with \cref{lem:edp quasi-mode slit} directly give \eqref{eq:lambda epsilon slit} and \eqref{eq:exit position slit}.
\end{proof}

\section{Tubular holes}
\label{sec:tubular holes}
In this section we consider the narrow escape problem for the hole shape
defined in \eqref{eq:tube def}.
Unlike in the previous sections,
the functions $\psi_\varepsilon^{(k)}$ are not given by an explicit formula:
$\psi_\varepsilon^{(k)}$ is defined in \cref{sub:definition of the quasi-mode}
as the solution of a Neumann problem where a flux $\Phi_\varepsilon^{(k)}$ leaves $\Omega_\varepsilon$ uniformly through the $k$-th hole $\Gke$ only.
We make this construction rigorous in \cref{prop:flux expansion}
by computing the energy matrix of the family $(\psi_\varepsilon^{(k)})_{k=1}^N$,
whose off-diagonal entries measure the interaction between the tubes through the bulk.
In \cref{sub:tube:dirichlet error},
we read the means of $\psi_\varepsilon^{(k)}$ on the holes from this matrix
and control the fluctuations around these means (\cref{prop:dirichlet tube}).
We conclude with \cref{prop:quasi-mode tube} in \cref{sub:tube:proof of thm}.

\subsection{Definition of the quasi-mode}
\label{sub:definition of the quasi-mode}
Throughout this section,
with a slight abuse of notation,
we still denote $\mathfrak{q}_\varepsilon(u, v) = \int_{\Omega_\varepsilon} \nabla u \cdot \nabla v$ as in \eqref{eq:def bilinear form},
now for $u, v \in H^1(\Omega_\varepsilon)$.

For $k \in \{1, \ldots, N\}$ and $\Phi \in \real$,
consider the Neumann problem where a flux $\Phi$ leaves $\Omega_\varepsilon$ uniformly through the $k$-th hole only:
\begin{equation}\label{eq:tube neumann problem}
    \left\{
        \begin{aligned}
            \Delta \psi^{(k)}_\Phi & = \frac{\Phi}{|\Omega_\varepsilon|} && \text{in } \Omega_\varepsilon, \\
            \partial_n \psi^{(k)}_\Phi & = 0 && \text{on } \partial \Omega_\varepsilon \setminus \Gke, \\
            \partial_n \psi^{(k)}_\Phi & = \frac{\Phi}{|\Gke|} && \text{on } \Gke .
        \end{aligned}
    \right.
\end{equation}
A weak solution of \eqref{eq:tube neumann problem} is a function $\psi^{(k)}_\Phi \in H^1(\Omega_\varepsilon)$ such that
\begin{equation}\label{eq:weak tube neumann problem}
    \forall v \in H^1(\Omega_\varepsilon), \qquad
    \mathfrak{q}_\varepsilon(\psi^{(k)}_\Phi, v)
    = \Phi \, \mathfrak{l}_k(v)
    = - \frac{\Phi}{|\Omega_\varepsilon|} \int_{\Omega_\varepsilon} v
        + \frac{\Phi}{|\Gke|} \int_{\Gke} v ,
\end{equation}
where $\mathfrak{l}_k$ is the continuous linear form on $H^1(\Omega_\varepsilon)$ defined by the right-hand side of \eqref{eq:weak tube neumann problem}.
As the compatibility condition is verified in \eqref{eq:tube neumann problem}, 
it has a unique zero-mean weak solution (see~\cite[Theorem~8.8 and Remark~8.11]{Salsa}).
By linearity, $\psi^{(k)}_\Phi = \Phi \, \psi^{(k)}_1$.

Take $\Phi = 1$ and $v = \psi^{(k)}_1$ in \eqref{eq:weak tube neumann problem},
as $\psi^{(k)}_1$ has zero mean on $\Omega_\varepsilon$,
it leads to
\[
    \mathfrak{q}_\varepsilon(\psi^{(k)}_1, \psi^{(k)}_1)
    = \frac{1}{|\Gke|} \int_{\Gke} \psi^{(k)}_1,
\]
hence its mean on $\Gke$ is related to its energy.
In view of the Dirichlet condition in the hypothesis
\eqref{eq:hyp:PDE quasi-mode},
it is natural to impose that the mean of the quasi-mode $\psi_\varepsilon^{(k)}$ on $\Gke$ is $1$.
This motivates the choice of the flux $\Phi_\varepsilon^{(k)}$ as
\begin{equation}\label{eq:def flux tube}
    \Phi_\varepsilon^{(k)} = \bigl\| \nabla \psi^{(k)}_1 \bigr\|_{L^2(\Omega_\varepsilon)}^{-2},
    \qquad \text{ so that } \qquad
    \frac{1}{|\Gke|} \int_{\Gke} \psi^{(k)}_{\Phi_\varepsilon^{(k)}} = 1 .
\end{equation}
We finally define the quasi-mode $\psi_\varepsilon^{(k)}$ as
\begin{equation}\label{eq:def quasi-mode tube}
    \psi_\varepsilon^{(k)} = \psi^{(k)}_{\Phi_\varepsilon^{(k)}},
\end{equation}
i.e., the unique zero-mean weak solution of \eqref{eq:tube neumann problem} with $\Phi = \Phi_\varepsilon^{(k)}$, 
and therefore, 
\begin{equation}\label{eq:tube energy psi}
    \mathfrak{q}_\varepsilon(\psi_\varepsilon^{(k)}, \psi_\varepsilon^{(k)}) = \Phi_\varepsilon^{(k)}.
\end{equation}

Moreover,
taking $v = \psi_\varepsilon^{(j)}$ in \eqref{eq:weak tube neumann problem} for some $j \in \{ 1, \ldots, N\}$,
as $\psi_\varepsilon^{(j)}$ has zero mean on $\Omega_\varepsilon$,
\begin{equation}\label{eq:tube energy matrix}
    \mathfrak{q}_\varepsilon \bigl(\psi_\varepsilon^{(k)}, \psi_\varepsilon^{(j)} \bigr)
    = \Phi_\varepsilon^{(k)} \, \frac{1}{|\Gke|} \int_{\Gke} \psi_\varepsilon^{(j)}.
\end{equation}
Hence the symmetric energy matrix
$\mathrm{Q}^\varepsilon = \bigl( \mathfrak{q}_\varepsilon( \psi_\varepsilon^{(i)}, \psi_\varepsilon^{(j)} ) \bigr)_{1 \leq i, j \leq N}$
contains both the fluxes $\Phi_\varepsilon^{(k)}$, on its diagonal,
and the means of each $\psi_\varepsilon^{(k)}$ on the other holes, off its diagonal.
The following proposition gives its asymptotics and fixes the value of the fluxes $\Phi_\varepsilon^{(k)}$.

\begin{prop}\label{prop:flux expansion}
    There are $C, \varepsilon_0 > 0$ such that, for all $\varepsilon < \varepsilon_0$ and $j \neq k \in \{1, \ldots, N\}$,
    \begin{equation}\label{eq:flux expansion}
        \Phi_\varepsilon^{(k)} = (\rek)^2 \frac{ |\mathcal{D}_k|}{L_k}
        \bigl( 1 + \mathrm{O}(\rek) \bigr),
        \qquad
        \bigl| \mathfrak{q}_\varepsilon \bigl( \psi_\varepsilon^{(j)}, \psi_\varepsilon^{(k)} \bigr) \bigr|
        \leq C \bigl( r_\varepsilon^{(j)} \rek \bigr)^{3/2} .
    \end{equation}
\end{prop}

\begin{proof}
    In this proof, $C$ denotes a constant independent of $\varepsilon$ that may change from line to line,
    and we use the index $i \in \{1, \ldots, N\}$ for the tubes.
    Let us first state an estimate used throughout the proof.
    From \eqref{eq:sk} and \eqref{eq:tube graph}, $|h_i| \leq s_\varepsilon^{(i)} = \mathrm{O}((r_\varepsilon^{(i)})^2)$ on $\mathcal{D}_{i, \varepsilon}$, so that
    \begin{equation}\label{eq:tube uniform field energy}
        |\mathcal{T}_{i, \varepsilon}|
        = \int_{\mathcal{D}_{i, \varepsilon}} \bigl( L_i - h_i(\rho) \bigr) \, \d\rho
        = L_i |\mathcal{D}_{i, \varepsilon}| \bigl(1 + \mathrm{O}((r_\varepsilon^{(i)})^2)\bigr).
    \end{equation}
    Fix $k \in \{1, \ldots, N\}$.
    We bound $\Phi_\varepsilon^{(k)}$ from above with a trial function,
    and from below with a trial flux,
    which we then use to estimate the off-diagonal entries of $\mathrm{Q}^\varepsilon$.

    \paragraph{Upper bound.}
    Let $v \in H^1(\Omega_\varepsilon)$ with $\mathfrak{l}_k(v) \neq 0$.
    Taking $v$ in \eqref{eq:weak tube neumann problem} with $\Phi = \Phi_\varepsilon^{(k)}$,
    the Cauchy--Schwarz inequality and \eqref{eq:tube energy psi} give
    $\Phi_\varepsilon^{(k)} |\mathfrak{l}_k(v)| = |\mathfrak{q}_\varepsilon(\psi_\varepsilon^{(k)}, v)| \leq \sqrt{\Phi_\varepsilon^{(k)}} \sqrt{\mathfrak{q}_\varepsilon(v, v)}$,
    that is
    \begin{equation}\label{eq:tube dirichlet principle}
        \Phi_\varepsilon^{(k)} \leq \frac{\mathfrak{q}_\varepsilon(v, v)}{\mathfrak{l}_k(v)^2} .
    \end{equation}
    Let us now construct a $v$ that gives the upper bound.
    We choose $v = 0$ outside $\Tke$ and, in the local coordinates $(\rho, s)$ of
    \eqref{eq:tube lim},
    \[
        v(\rho, s) = \frac{s - h_k(\rho)}{L_k}
        \qquad \text{in } \Tke .
    \]
    As $v$ vanishes on the mouth $\Ske$, it is continuous across it, and $v \in H^1(\Omega_\varepsilon)$.
    On $\Gke = \{ s = L_k \}$,
    $v = (L_k - h_k(\rho)) / L_k$,
    so that, as $|\Gke| = |\Dke|$,
    its mean on $\Gke$ is $|\Tke| / (L_k |\Dke|) = 1 + \mathrm{O}((\rek)^2)$ by \eqref{eq:tube uniform field energy}.
    Moreover,
    as $0 \leq v \leq C$ in $\Tke$ and $|\Tke| \leq C (\rek)^2$,
    the mean of $v$ on $\Omega_\varepsilon$ is bounded by $C (\rek)^2$,
    so that $\mathfrak{l}_k(v) = 1 + \mathrm{O}((\rek)^2)$.
    Finally,
    in $\Tke$,
    $\nabla v = (n_k - \nabla h_k) / L_k$
    with $\nabla h_k \perp n_k$, so
    $|\nabla v|^2 = (1 + \mathrm{O}((\rek)^2)) / L_k^2$
    by \eqref{eq:tube graph},
    and $\mathfrak{q}_\varepsilon(v, v) = |\Tke| \, (1 + \mathrm{O}((\rek)^2)) / L_k^2$.
    Plugging these estimates into \eqref{eq:tube dirichlet principle},
    and using \eqref{eq:tube uniform field energy}, we obtain
    \begin{equation}\label{eq:tube upper bound}
        \Phi_\varepsilon^{(k)}
        \leq \frac{|\Dke|}{L_k} \bigl( 1 + C (\rek)^2 \bigr).
    \end{equation}

    \paragraph{Lower bound.}
    This part is based on an idea of Prager--Synge (see the original paper \cite{MR25902}
    and \cite{MR4132116} for a modern presentation).
    Assume there exists $F \in L^2(\Omega_\varepsilon; \real^3)$ which satisfies
    \begin{equation}\label{eq:tube admissible flux}
        \int_{\Omega_\varepsilon} F \cdot \nabla v
        = \mathfrak{l}_k(v)
        \qquad \text{for every } v \in H^1(\Omega_\varepsilon).
    \end{equation}
    Taking $v = \psi_\varepsilon^{(k)}$ in \eqref{eq:tube admissible flux}
    and comparing it with \eqref{eq:weak tube neumann problem} for $\Phi = \Phi_\varepsilon^{(k)}$,
    one has 
    \[
        \Phi_\varepsilon^{(k)} \int_{\Omega_\varepsilon} F \cdot \nabla \psi_\varepsilon^{(k)} = 
        \Phi_\varepsilon^{(k)} \, \mathfrak{l}_k(\psi_\varepsilon^{(k)})
        = \mathfrak{q}_\varepsilon(\psi_\varepsilon^{(k)}, \psi_\varepsilon^{(k)}).
    \]
    Hence, expanding the square and using \eqref{eq:tube energy psi},
    \begin{equation}\label{eq:tube PragerSynge}\begin{aligned}
        \int_{\Omega_\varepsilon} \bigl| \nabla \psi_\varepsilon^{(k)} - \Phi_\varepsilon^{(k)} F \bigr|^2
        & = \mathfrak{q}_\varepsilon \bigl( \psi_\varepsilon^{(k)}, \psi_\varepsilon^{(k)} \bigr)
        - 2 \, \mathfrak{q}_\varepsilon \bigl( \psi_\varepsilon^{(k)}, \psi_\varepsilon^{(k)} \bigr)
        + \bigl( \Phi_\varepsilon^{(k)} \bigr)^2 \int_{\Omega_\varepsilon} |F|^2\\
        & = \bigl( \Phi_\varepsilon^{(k)} \bigr)^2 \int_{\Omega_\varepsilon} |F|^2 - \Phi_\varepsilon^{(k)} .
    \end{aligned}\end{equation}
    Therefore, to show the desired lower bound,
    it suffices to construct $F$
    that verifies \eqref{eq:tube admissible flux} and with $\left\| F \right\|_{L^2(\Omega_\varepsilon)}^2$ bounded by
    $L_k / |\Dke| + C / \rek$.

    In each tube $\mathcal{T}_{i, \varepsilon}$,
    using the coordinates $(\rho, s)$ of \eqref{eq:tube lim}, we define
    \begin{equation}\label{eq:def tube trial flux}
        F_k = f_{k, i} \, n_i \quad \text{in } \mathcal{T}_{i, \varepsilon},
        \qquad \text{ with } \qquad
        f_{k, i}(\rho, s) = \frac{\delta_{ik}}{|\Dke|} - \frac{L_i - s}{|\Omega_\varepsilon|} .
    \end{equation}
    In $\Omega$, we take $F_k = \nabla \hat f_k$,
    where $\hat f_k$ extends the fluxes $f_{k, i}$ to the bulk:
    it is the zero-mean solution of
    \begin{equation}\label{eq:tube trial bulk}
        \left\{
            \begin{aligned}
                \Delta \hat f_k & = \frac{1}{|\Omega_\varepsilon|} && \text{in } \Omega, \\
                \partial_n \hat f_k & = f_{k, i} \, (n_i \cdot n) && \text{on } \Sigma_{i, \varepsilon}, \quad i \in \{1, \ldots, N\}, \\
                \partial_n \hat f_k & = 0 && \text{elsewhere on } \partial\Omega .
            \end{aligned}
        \right.
    \end{equation}
    From \cite[Proposition 2.3]{aramaki2018existence},
    since the sources in \eqref{eq:tube trial bulk} are bounded,
    there exists a unique $\hat f_k \in W^{1, 4}(\Omega) \subset C^0(\overline{\Omega})$ solving \eqref{eq:tube trial bulk} if and only if
    \begin{equation}\label{eq:tube compatibility condition}
        \sum_{i=1}^N \int_{\Sigma_{i, \varepsilon}} f_{k, i} \, (n_i \cdot n) \, \d\sigma = \frac{|\Omega|}{|\Omega_\varepsilon|} .
    \end{equation}
    Let us show that this condition is satisfied.
    For a fixed $i \in \{1, \ldots, N\}$,
    from the parametrization of $\Sigma_{i, \varepsilon}$ in \eqref{eq:def tube mouth},
    one has $\d\sigma = \sqrt{1+|\nabla h_i|^2}\,\d\rho$ on $\Sigma_{i, \varepsilon}$.
    Moreover, since $\Omega$
    lies in $\{s < h_i(\rho)\}$, $n = (n_i - \nabla h_i)/\sqrt{1+|\nabla h_i|^2}$.
    As~$\nabla h_i \perp n_i$, this gives $(n_i\cdot n)\,\d\sigma = \d\rho$, hence
    \[
        \int_{\Sigma_{i, \varepsilon}} f_{k, i} \, (n_i \cdot n) \, \d\sigma
        = \int_{\mathcal{D}_{i, \varepsilon}} \Bigl( \frac{\delta_{ik}}{|\Dke|} - \frac{L_i - h_i(\rho)}{|\Omega_\varepsilon|} \Bigr) \d\rho
        = \delta_{ik} - \frac{|\mathcal{T}_{i, \varepsilon}|}{|\Omega_\varepsilon|}.
    \]
    Summing over $i$, and noticing that $\sum_{i=1}^N |\mathcal{T}_{i, \varepsilon}| = |\Omega_\varepsilon| - |\Omega|$,
    leads to \eqref{eq:tube compatibility condition}.

    Let us now show that $F_k$ satisfies \eqref{eq:tube admissible flux}.
    Let $v \in H^1(\Omega_\varepsilon)$.
    The idea is to integrate by parts in $\Omega$ and in each $\mathcal{T}_{i, \varepsilon}$ separately.
    For a fixed $i \in \{1, \ldots, N\}$,
    $F_k$ is smooth up to the boundary in~$\mathcal{T}_{i, \varepsilon}$,
    with $\nabla \cdot F_k = \partial_s f_{k, i} = 1 / |\Omega_\varepsilon|$.
    Moreover, $F_k$ is tangent to the lateral boundary of $\mathcal{T}_{i, \varepsilon}$,
    the outward normal of $\mathcal{T}_{i, \varepsilon}$ is $n_i$ on $\Gamma_i^\varepsilon$ and $-n$ on $\Sigma_{i, \varepsilon}$,
    and $f_{k, i} = \delta_{ik} / |\Dke|$ on $\Gamma_i^\varepsilon$, so that
    \begin{equation}\label{eq:tube flux tube part}
        \int_{\mathcal{T}_{i, \varepsilon}} F_k \cdot \nabla v
        = - \frac{1}{|\Omega_\varepsilon|} \int_{\mathcal{T}_{i, \varepsilon}} v
        + \frac{\delta_{ik}}{|\Dke|} \int_{\Gamma_i^\varepsilon} v
        - \int_{\Sigma_{i, \varepsilon}} f_{k, i} \, (n_i \cdot n) \, v \, \d\sigma .
    \end{equation}
    In $\Omega$, as $v|_\Omega \in H^1(\Omega)$,
    the weak formulation of \eqref{eq:tube trial bulk} reads
    \begin{equation}\label{eq:tube flux bulk part}
        \int_{\Omega} F_k \cdot \nabla v
        = \int_\Omega \nabla \hat f_k \cdot \nabla v
        = - \frac{1}{|\Omega_\varepsilon|} \int_\Omega v
        + \sum_{i=1}^{\ndoors} \int_{\Sigma_{i, \varepsilon}} f_{k, i} \, (n_i \cdot n) \, v \, \d\sigma .
    \end{equation}
    Summing \eqref{eq:tube flux bulk part} and \eqref{eq:tube flux tube part} over $i$
    leads to \eqref{eq:tube admissible flux}, 
    as $|\Gke| = |\Dke|$.

    It remains to bound $\left\| F_k \right\|_{L^2(\Omega_\varepsilon)}^2$.
    Let us first compute the integral over the tubes.
    For $i \in \{1, \ldots, N\}$ and $(\rho, s)$ such that $\rho \in \mathcal{D}_{i, \varepsilon}$ and $h_i(\rho) \leq s \leq L_i$,
    one has~$0 \leq L_i - s \leq L_i + s_\varepsilon^{(i)} \leq C$,
    and $|\Dke| = (\rek)^2 |\mathcal{D}_k|$,
    so that
    \begin{equation}\label{eq:tube bound jk}
        \Bigl| f_{k, i} - \frac{\delta_{ik}}{|\Dke|} \Bigr| \leq C,
        \qquad \text{and} \qquad
        |f_{k, i}| \leq \frac{\delta_{ik}}{|\Dke|} + C \leq \frac{C}{(r_\varepsilon^{(i)})^2}.
    \end{equation}
    Integrating $|f_{k, i}|^2$ over $\mathcal{T}_{i, \varepsilon}$,
    using $|\mathcal{T}_{i, \varepsilon}| \leq C (r_\varepsilon^{(i)})^2$ and \eqref{eq:tube uniform field energy},
    and summing over $i$ leads to
    \begin{equation}\label{eq:tube energy tubes}
        \sum_{i=1}^{\ndoors} \int_{\mathcal{T}_{i, \varepsilon}} |F_k|^2
        \leq \frac{|\Tke|}{|\Dke|^2} + C
        \leq \frac{L_k}{|\Dke|} + C.
    \end{equation}
    Now for the bulk part,
    taking $v = \hat f_k$ in the weak formulation of \eqref{eq:tube trial bulk},
    as $\hat f_k$ has zero mean on $\Omega$, $(n_i \cdot n) \, \d\sigma = \d\rho$ on $\Sigma_{i, \varepsilon}$,
    and $\int_{\mathcal{D}_{i, \varepsilon}} |f_{k, i}| \leq \delta_{ik} + C |\mathcal{D}_{i, \varepsilon}|$ by \eqref{eq:tube bound jk},
    \begin{equation}\label{eq:tube energy bulk trace}
        \int_\Omega |F_k|^2
        = \sum_{i=1}^N \int_{\Sigma_{i, \varepsilon}} f_{k, i} \, \hat f_k \, (n_i \cdot n) \, \d\sigma
        \leq C \, \| \hat f_k \|_{L^\infty(\Omega)} .
    \end{equation}
    By \cref{lem:Green solution},
    $\hat f_k(x) = \int_{\partial \Omega} G(x, y) \, g(y) \, \d\sigma(y)$ for $x \in \Omega$,
    where~$G$ is the Green function of \cref{def:Neumann Green}
    and $g = \partial_n \hat f_k$ is given by \eqref{eq:tube trial bulk}.
    Hence, by \cref{lem:Green David} and \eqref{eq:tube bound jk},
    \begin{equation}\label{eq:tube J pointwise}
        |\hat f_k(x)|
        \leq C \sum_{i=1}^{\ndoors} \Bigl( \frac{\delta_{ik}}{|\Dke|} + 1 \Bigr) \int_{\Sigma_{i, \varepsilon}} \frac{\d\sigma(y)}{|x - y|},
        \qquad x \in \Omega.
    \end{equation}
    To compute the integral in \eqref{eq:tube J pointwise}, fix $i \in \{1, \ldots, N\}$ and $x \in \Omega$,
    and let $\rho_x$ be the orthogonal projection of $x - x^{(i)}$ onto $T_{x^{(i)}}(\partial \Omega)$.
    For $y = x^{(i)} + \rho + h_i(\rho) n_i \in \Sigma_{i, \varepsilon}$,
    one has $|x - y| \geq |\rho_x - \rho|$.
    Since $\rho \mapsto |\rho_x - \rho|^{-1}$ is radially decreasing around $\rho_x$,
    its integral over $\mathcal{D}_{i, \varepsilon}$ is at most its integral over the disc $B(\rho_x, a_i)$ of the same area,
    that is with $\pi a_i^2 = |\mathcal{D}_{i, \varepsilon}| = (r_\varepsilon^{(i)})^2 |\mathcal{D}_i|$.
    Therefore, using that $\d\sigma \leq 2 \, \d\rho$ for $\varepsilon$ small enough,
    \[
        \int_{\Sigma_{i, \varepsilon}} \frac{\d\sigma(y)}{|x - y|}
        \leq 2 \int_{\mathcal{D}_{i, \varepsilon}} \frac{\d\rho}{|\rho_x - \rho|}
        \leq 2 \int_{B(\rho_x, a_i)} \frac{\d\rho}{|\rho_x - \rho|}
        = 4 \pi a_i
        = 4 \sqrt{\pi |\mathcal{D}_i|} \, r_\varepsilon^{(i)} .
    \]
    Plugging this into \eqref{eq:tube J pointwise} gives
    $\| \hat f_k \|_{L^\infty(\Omega)} \leq C / \rek$,
    and \eqref{eq:tube energy bulk trace} becomes
    \begin{equation}\label{eq:tube energy bulk}
        \int_\Omega |F_k|^2 = \int_\Omega |\nabla \hat f_k|^2 \leq \frac{C}{\rek} .
    \end{equation}
    Finally,
    the combination of \eqref{eq:tube energy tubes} and \eqref{eq:tube energy bulk} gives,
    as $|\Dke| = (\rek)^2 |\mathcal{D}_k|$,
    \[
        \bigl\| F_k \bigr\|_{L^2(\Omega_\varepsilon)}^2
        \leq \frac{L_k}{|\Dke|} + \frac{C}{\rek}
        \leq \frac{L_k}{|\Dke|} \bigl( 1 + C \rek \bigr),
        \qquad \text{so that} \qquad
        \Phi_\varepsilon^{(k)} \geq \frac{|\Dke|}{L_k} \bigl( 1 - C \rek \bigr).
    \]
    Together with \eqref{eq:tube upper bound},
    this gives the first estimate of \eqref{eq:flux expansion}.
    Moreover,
    \eqref{eq:tube PragerSynge} with $F = F_k$, \eqref{eq:tube upper bound} and the previous bound on $\| F_k \|_{L^2(\Omega_\varepsilon)}$ give
    \begin{equation}\label{eq:tube remainder}
        \bigl\| \nabla \psi_\varepsilon^{(k)} - \Phi_\varepsilon^{(k)} F_k \bigr\|_{L^2(\Omega_\varepsilon)}^2
        = \Phi_\varepsilon^{(k)} \Bigl( \Phi_\varepsilon^{(k)} \bigl\| F_k \bigr\|_{L^2(\Omega_\varepsilon)}^2 - 1 \Bigr)
        \leq \Phi_\varepsilon^{(k)} \Bigl( \bigl( 1 + C (\rek)^2 \bigr) \bigl( 1 + C \rek \bigr) - 1 \Bigr)
        \leq C (\rek)^3 .
    \end{equation}

    \paragraph{Interactions.}
    Let $j \neq k \in \{1, \ldots, N\}$.
    By \eqref{eq:tube admissible flux} and \eqref{eq:weak tube neumann problem},
    $\Phi_\varepsilon^{(k)} \int_{\Omega_\varepsilon} F_k \cdot \nabla \psi_\varepsilon^{(j)} = \Phi_\varepsilon^{(k)} \mathfrak{l}_k(\psi_\varepsilon^{(j)}) = \mathfrak{q}_\varepsilon(\psi_\varepsilon^{(j)}, \psi_\varepsilon^{(k)})$,
    and similarly when exchanging $j$ and $k$.
    Hence, as in \eqref{eq:tube PragerSynge}, expanding the product gives
    \[
        \int_{\Omega_\varepsilon} \bigl( \nabla \psi_\varepsilon^{(j)} - \Phi_\varepsilon^{(j)} F_j \bigr) \cdot \bigl( \nabla \psi_\varepsilon^{(k)} - \Phi_\varepsilon^{(k)} F_k \bigr)
        = \mathfrak{q}_\varepsilon \bigl( \psi_\varepsilon^{(j)}, \psi_\varepsilon^{(k)} \bigr)
        - 2 \, \mathfrak{q}_\varepsilon \bigl( \psi_\varepsilon^{(j)}, \psi_\varepsilon^{(k)} \bigr)
        + \Phi_\varepsilon^{(j)} \Phi_\varepsilon^{(k)} \int_{\Omega_\varepsilon} F_j \cdot F_k ,
    \]
    that is,
    \begin{equation}\label{eq:expansion cross terms}
        \mathfrak{q}_\varepsilon \bigl( \psi_\varepsilon^{(j)}, \psi_\varepsilon^{(k)} \bigr)
        = \Phi_\varepsilon^{(j)} \Phi_\varepsilon^{(k)} \int_{\Omega_\varepsilon} F_j \cdot F_k
        - \int_{\Omega_\varepsilon} \bigl( \nabla \psi_\varepsilon^{(j)} - \Phi_\varepsilon^{(j)} F_j \bigr) \cdot \bigl( \nabla \psi_\varepsilon^{(k)} - \Phi_\varepsilon^{(k)} F_k \bigr) .
    \end{equation}

    Let us now bound the right-hand side in \eqref{eq:expansion cross terms}.
    By the Cauchy--Schwarz inequality and \eqref{eq:tube remainder},
    the last integral is bounded by $C (r_\varepsilon^{(j)} \rek)^{3/2}$.
    In each tube $\mathcal{T}_{i, \varepsilon}$,
    as $\delta_{ij} \delta_{ik} = 0$,
    \eqref{eq:tube bound jk} gives $|f_{j, i} f_{k, i}| \leq C (r_\varepsilon^{(i)})^{-2}$,
    so that, with $|\mathcal{T}_{i, \varepsilon}| \leq C (r_\varepsilon^{(i)})^2$,
    \[
        \left|\int_{\mathcal{T}_{i, \varepsilon}} F_j \cdot F_k\right| \leq C.
    \]
    In $\Omega$,
    the Cauchy--Schwarz inequality and \eqref{eq:tube energy bulk} give
    $|\int_\Omega \nabla \hat f_j \cdot \nabla \hat f_k| \leq C (r_\varepsilon^{(j)} \rek)^{-1/2}$.
    Hence, as $\Phi_\varepsilon^{(j)} \Phi_\varepsilon^{(k)} \leq C (r_\varepsilon^{(j)} \rek)^2$ by the first estimate of \eqref{eq:flux expansion},
    \[
        |\mathfrak{q}_\varepsilon(\psi_\varepsilon^{(j)}, \psi_\varepsilon^{(k)})| \leq C (r_\varepsilon^{(j)} \rek)^{3/2},
    \]
    which is the second estimate of \eqref{eq:flux expansion}.
\end{proof}

\subsection{The Dirichlet error}
\label{sub:tube:dirichlet error}
In this subsection,
we control the trace of $\psi_\varepsilon^{(k)}$ on each hole $\Gamma_j^\varepsilon$, $j \in \{1, \ldots, N\}$,
by splitting it into its mean
\begin{equation}\label{eq:tube means matrix}
    \mathrm{M}^\varepsilon_{jk} = \frac{1}{|\Gamma_j^\varepsilon|} \int_{\Gamma_j^\varepsilon} \psi_\varepsilon^{(k)}
\end{equation}
and its fluctuation around this mean.
By \eqref{eq:tube energy matrix},
the matrix of the means is
$\mathrm{M}^\varepsilon = \mathrm{diag}(\Phi_\varepsilon^{(1)}, \ldots, \Phi_\varepsilon^{(N)})^{-1} \, \mathrm{Q}^\varepsilon$.
In particular, $\mathrm{M}^\varepsilon_{kk} = 1$ by \eqref{eq:def flux tube},
and, for $j \neq k$, \cref{prop:flux expansion} gives
\[
    |\mathrm{M}^\varepsilon_{jk}| = \frac{|\mathrm{Q}^\varepsilon_{jk}|}{\Phi_\varepsilon^{(j)}} \leq C \Bigl( \frac{(\rek)^3}{r_\varepsilon^{(j)}} \Bigr)^{1/2}.
\]
Hence, denoting
\begin{equation}\label{eq:tube eta}
    \eta_\varepsilon = \Bigl( \frac{(\reo)^3}{\min_{1 \leq j \leq \ndoors} r_\varepsilon^{(j)}} \Bigr)^{1/2} \geq \reo,
\end{equation}
it holds that, for all $j, k \in \{1, \ldots, N\}$,
\begin{equation}\label{eq:tube means}
    \bigl| \mathrm{M}^\varepsilon_{jk} - \delta_{jk} \bigr| \leq C \eta_\varepsilon.
\end{equation}
For the fluctuations around the mean,
we expand $\psi_\varepsilon^{(k)}$ on the Neumann eigenfunctions of the cross-section in the upper half of the $j$-th tube:
the corresponding modes are explicit, and controlled by the energy of~$\psi_\varepsilon^{(k)}$.

\begin{prop}\label{prop:dirichlet tube}
    There is $C > 0$ such that, for $\varepsilon$ small enough and $j, k \in \{1, \ldots, N\}$,
    the function $\psi_\varepsilon^{(k)}$ is continuous on $\overline{\Gamma_j^\varepsilon}$ and
    \begin{equation}\label{eq:dirichlet tube}
        \bigl\| \psi_\varepsilon^{(k)} - \mathrm{M}^\varepsilon_{jk} \bigr\|_{L^\infty(\Gamma_j^\varepsilon)}
        \leq C r_\varepsilon^{(j)} \, \bigl\| \nabla \psi_\varepsilon^{(k)} \bigr\|_{L^2(\Omega_\varepsilon)}
        \leq C (\reo)^2,
        \qquad \text{so that} \qquad
        \bigl\| \psi_\varepsilon^{(k)} - \delta_{jk} \bigr\|_{L^\infty(\Gamma_j^\varepsilon)}
        \leq C \eta_\varepsilon ,
    \end{equation}
    with $\eta_\varepsilon$ defined in \eqref{eq:tube eta}.
\end{prop}

\begin{proof}
    Fix $j, k \in \{1, \ldots, N\}$,
    denote $\psi = \psi_\varepsilon^{(k)}$ and $r = r_\varepsilon^{(j)}$,
    and use the coordinates $(\rho, s)$ of \eqref{eq:tube lim} in the $j$-th tube.
    The last two estimates of \eqref{eq:dirichlet tube} follow from the first one:
    by \eqref{eq:tube energy psi} and \eqref{eq:flux expansion},
    $\| \nabla \psi \|_{L^2(\Omega_\varepsilon)} = (\Phi_\varepsilon^{(k)})^{1/2} \leq C \rek$,
    which gives the second inequality as $r \rek \leq (\reo)^2$,
    and then the triangle inequality and \eqref{eq:tube means} give
    $\| \psi - \delta_{jk} \|_{L^\infty(\Gamma_j^\varepsilon)} \leq C (\reo)^2 + |\mathrm{M}^\varepsilon_{jk} - \delta_{jk}| \leq C \eta_\varepsilon$,
    as $(\reo)^2 \leq \reo \leq \eta_\varepsilon$ for $\varepsilon$ small enough.
    It remains to show the first inequality of \eqref{eq:dirichlet tube}.

    \paragraph{Transverse modes.}
    Denote by $0 = \mu_0 < \mu_1 \leq \mu_2 \leq \ldots$ the Neumann eigenvalues of $\mathcal{D}_j$,
    and by $(\phi_n)_{n \geq 0}$ an associated $L^2(\mathcal{D}_j)$-orthonormal basis of eigenfunctions, $\phi_0$ being constant.
    By scaling, the functions~$\phi_n^\varepsilon(\rho) = r^{-1} \phi_n(\rho / r)$
    form an $L^2(\mathcal{D}_{j, \varepsilon})$-orthonormal basis of Neumann eigenfunctions of $\mathcal{D}_{j, \varepsilon} = r \mathcal{D}_j$,
    with eigenvalues $\mu_n / r^2$, that is
    \begin{equation}\label{eq:tube transverse eigen}
        \forall v \in H^1(\mathcal{D}_{j, \varepsilon}), \qquad
        \int_{\mathcal{D}_{j, \varepsilon}} \nabla \phi_n^\varepsilon \cdot \nabla v
        = \frac{\mu_n}{r^2} \int_{\mathcal{D}_{j, \varepsilon}} \phi_n^\varepsilon \, v .
    \end{equation}
    The goal of this step is to decompose $\psi$ along the modes $\phi_n^\varepsilon$ in the upper half of the tube.
    As $s_\varepsilon^{(j)} = \mathrm{O}(r^2)$ by \eqref{eq:sk} and \eqref{eq:tube graph},
    for $\varepsilon$ small enough the cross-section of $\mathcal{T}_{j, \varepsilon}$ at any height $s \in (L_j/2, L_j)$ is $\mathcal{D}_{j, \varepsilon} = r \mathcal{D}_j$.
    For $n \geq 1$, set
    \begin{equation}\label{eq:tube transverse modes}
        a_n(s) = \int_{\mathcal{D}_{j, \varepsilon}} \psi(\rho, s) \, \phi_n^\varepsilon(\rho) \, \d\rho,
        \qquad s \in [L_j/2, L_j].
    \end{equation}
    As $\psi \in H^1(\mathcal{D}_{j, \varepsilon} \times (L_j/2, L_j))$,
    its restrictions to lines are absolutely continuous \cite[Theorem 4.21]{EvansGariepy}:
    up to modifying $\psi$ on a zero-measure set,
    for a.e.\ $\rho \in \mathcal{D}_{j, \varepsilon}$,
    the function $s \mapsto \psi(\rho, s)$ is absolutely continuous on $[L_j/2, L_j]$.
    Hence, the function $a_n$ is in $C^0([L_j/2, L_j]) \cap H^1(L_j/2, L_j)$.

    Take a smooth function $\chi \in C^\infty([L_j/2, L_j])$ vanishing near $L_j/2$,
    and let $w = \chi(s) \phi_n^\varepsilon(\rho)$ in $\mathcal{D}_{j, \varepsilon} \times (L_j/2, L_j)$, extended by $0$, so that $w \in H^1(\Omega_\varepsilon)$.
    As $\phi_n^\varepsilon$ has zero mean on $\mathcal{D}_{j, \varepsilon}$, $\int_{\Omega_\varepsilon} w = \int_{\Gke} w = 0$, that is $\mathfrak{l}_k(w) = 0$.
    Hence, by \eqref{eq:weak tube neumann problem},
    as $\nabla w = \chi \nabla_\rho \phi_n^\varepsilon + \chi' \phi_n^\varepsilon \, n_j$,
    and by \eqref{eq:tube transverse eigen} with $v = \psi(\cdot, s)$,
    \begin{equation}\label{eq:tube weak ode}\begin{aligned}
        0 = \mathfrak{q}_\varepsilon(\psi, w)
        & = \int_{L_j/2}^{L_j} \Bigl( \chi'(s) \int_{\mathcal{D}_{j, \varepsilon}} \partial_s \psi(\cdot, s) \, \phi_n^\varepsilon
        + \chi(s) \int_{\mathcal{D}_{j, \varepsilon}} \nabla_\rho \psi(\cdot, s) \cdot \nabla_\rho \phi_n^\varepsilon \Bigr) \d s\\
        &= \int_{L_j/2}^{L_j} \Bigl( \chi' a_n' + \frac{\mu_n}{r^2} \chi \, a_n \Bigr) \d s,
    \end{aligned}\end{equation}
    where we used the definition of $a_n$ in \eqref{eq:tube transverse modes} and the expression of $a_n'$.

    First, taking $\chi \in C^\infty_c((L_j/2, L_j))$ in \eqref{eq:tube weak ode}
    shows that 
    \begin{equation}\label{eq:ode an}
        a_n'' = \mu_n r^{-2} a_n,
    \end{equation}
    in the sense of distributions on~$(L_j/2, L_j)$.
    As $a_n$ is continuous, this shows that $a_n \in C^2([L_j/2, L_j])$ and that \eqref{eq:ode an} holds pointwise on $[L_j/2, L_j]$.
    Then, for $\chi$ such that $\chi(L_j/2) = 0$
    and $\chi(L_j) = 1$,
    an integration by parts gives
    \[
        \int_{L_j/2}^{L_j} \chi' a_n' = \chi(L_j) a_n'(L_j) - \int_{L_j/2}^{L_j} \chi a_n''.
    \]
    Using \eqref{eq:ode an} and \eqref{eq:tube weak ode} shows that $a_n'(L_j) = 0$.
    Therefore, solving \eqref{eq:ode an} with the boundary condition $a_n'(L_j) = 0$ gives
    \begin{equation}\label{eq:tube expansion}
        a_n(s) = a_n(L_j) \cosh \Bigl( \frac{\sqrt{\mu_n}}{r} (L_j - s) \Bigr),
        \qquad s \in [L_j/2, L_j].
    \end{equation}

    \paragraph{Energy bound.}
    By \eqref{eq:tube transverse eigen},
    the functions $r \mu_n^{-1/2} \phi_n^\varepsilon$, $n \geq 1$,
    are orthonormal for the scalar product~$(\theta, \theta') \mapsto \int_{\mathcal{D}_{j, \varepsilon}} \nabla \theta \cdot \nabla \theta'$.
    Hence, integrating Bessel's inequality over $s \in (L_j/2, L_j)$,
    and using \eqref{eq:tube transverse modes} and~\eqref{eq:tube expansion} with $\cosh^2 \xi \geq \xi^2$,
    \begin{equation}\label{eq:tube coefficients}
        \frac{L_j^3}{24 \, r^4} \sum_{n \geq 1} \mu_n^2 \, a_n(L_j)^2
        \leq \sum_{n \geq 1} \frac{\mu_n}{r^2} \int_{L_j/2}^{L_j} a_n(s)^2 \, \d s
        \leq \bigl\| \nabla \psi \bigr\|_{L^2(\Omega_\varepsilon)}^2 .
    \end{equation}

    \paragraph{Uniform bound on the hole.}
    It remains to turn the bound \eqref{eq:tube coefficients} on the coefficients of the trace of $\psi$
    into a uniform bound on $\Gamma_j^\varepsilon$.
    To work on the fixed cross-section $\mathcal{D}_j$, let
    \[
        f(y) = \psi(r y, L_j) - \mathrm{M}^\varepsilon_{jk},
        \qquad y \in \mathcal{D}_j,
    \]
    so that the left-hand side of \eqref{eq:dirichlet tube} is~$\| f \|_{L^\infty(\mathcal{D}_j)}$,
    and denote~$f_n = \int_{\mathcal{D}_j} f \, \phi_n$ its coordinates in the basis $(\phi_n)_{n \geq 0}$.
    It is clear by definition that $f$ has zero mean over $\mathcal{D}_j$, that is $f_0 = 0$.
    Then, for $n \geq 1$, as $\phi_n$ has zero mean,
    the same change of variables and the definitions of $\phi_n^\varepsilon$ and $a_n$ give
    \[
        f_n
        = \int_{\mathcal{D}_j} \psi(r y, L_j) \, \phi_n(y) \, \d y
        = \frac{1}{r^2} \int_{\mathcal{D}_{j, \varepsilon}} \psi(\rho, L_j) \, \phi_n(\rho / r) \, \d\rho
        = \frac{1}{r} \int_{\mathcal{D}_{j, \varepsilon}} \psi(\rho, L_j) \, \phi_n^\varepsilon(\rho) \, \d\rho
        = \frac{a_n(L_j)}{r} .
    \]
    Moreover, for $n \geq 1$, $\mu_n \geq \mu_1 > 0$ gives $1 + \mu_n \leq (1 + \mu_1^{-1}) \mu_n$.
    Hence, by \eqref{eq:tube coefficients},
    \begin{equation}\label{eq:tube trace coefficients}
        \sum_{n \geq 1} (1 + \mu_n)^2 f_n^2
        \leq \bigl( 1 + \mu_1^{-1} \bigr)^2 r^{-2} \sum_{n \geq 1} \mu_n^2 \, a_n(L_j)^2
        \leq \frac{24 \, (1 + \mu_1^{-1})^2}{L_j^3} \, r^2 \bigl\| \nabla \psi \bigr\|_{L^2(\Omega_\varepsilon)}^2 .
    \end{equation}
    In particular,
    the series $g = \sum_{n \geq 1} (1 + \mu_n) f_n \phi_n$ converges in $L^2(\mathcal{D}_j)$,
    and $\| g \|_{L^2(\mathcal{D}_j)}^2$ is the left-hand side of \eqref{eq:tube trace coefficients}.

    Let us show that $f$ is the weak solution of the Neumann problem $- \Delta u + u = g$ in $\mathcal{D}_j$, $\partial_n u = 0$ on $\partial \mathcal{D}_j$.
    By classical theory, this problem has a unique weak solution $u \in H^1(\mathcal{D}_j)$, that verifies
    \[
        \forall v \in H^1(\mathcal{D}_j), \qquad \int_{\mathcal{D}_j} \nabla u \cdot \nabla v + \int_{\mathcal{D}_j} u \, v = \int_{\mathcal{D}_j} g \, v.
    \]
    Taking $v = \phi_n$,
    and as $\int_{\mathcal{D}_j} \nabla \phi_n \cdot \nabla u = \mu_n \int_{\mathcal{D}_j} \phi_n \, u$ by definition of the Neumann eigenpair $(\mu_n, \phi_n)$,
    one gets $(1 + \mu_n) \int_{\mathcal{D}_j} u \, \phi_n = \int_{\mathcal{D}_j} g \, \phi_n = (1 + \mu_n) f_n$ for every $n \geq 0$.
    Hence $u$ and $f$ have the same coordinates in the Hilbert basis $(\phi_n)$ of $L^2(\mathcal{D}_j)$, thus $u = f$.

    Finally, as $\mathcal{D}_j$ is smooth,
    the $H^2$ regularity of the Neumann problem \cite[Theorem 9.26]{bezis} gives $f \in H^2(\mathcal{D}_j)$ with $\| f \|_{H^2(\mathcal{D}_j)} \leq C \| g \|_{L^2(\mathcal{D}_j)}$,
    and the Sobolev embedding $H^2(\mathcal{D}_j) \subset C(\overline{\mathcal{D}_j})$, valid in dimension $2$,
    shows that $f$ is continuous on $\overline{\mathcal{D}_j}$.
    Together with \eqref{eq:tube trace coefficients}, this gives
    \[
        \| f \|_{L^\infty(\mathcal{D}_j)}
        \leq C \| f \|_{H^2(\mathcal{D}_j)}
        \leq C \| g \|_{L^2(\mathcal{D}_j)}
        \leq C r \bigl\| \nabla \psi \bigr\|_{L^2(\Omega_\varepsilon)},
    \]
    where $C$ only depends on $\mathcal{D}_j$ and $L_j$.
    By definition of $f$, this shows that $\psi$ is continuous on $\overline{\Gamma_j^\varepsilon}$,
    and gives the first inequality of \eqref{eq:dirichlet tube}.
\end{proof}

\subsection{The quasi-mode is admissible}
\label{sub:tube:proof of thm}
We can now show that \cref{hyp:quasi-mode} holds for the functions $\psi_\varepsilon^{(k)}$ defined in \eqref{eq:def quasi-mode tube},
and derive the main result of this section.

\begin{prop}\label{prop:quasi-mode tube}
    Assume that the radii satisfy
    \begin{equation}\label{eq:tube radii}
        (\reo)^3 = \mathrm{o}\Bigl( \min_{1 \leq j \leq \ndoors} r_\varepsilon^{(j)} \Bigr)
        \qquad \text{as } \varepsilon \to 0,
    \end{equation}
    that is, $\eta_\varepsilon \to 0$ with $\eta_\varepsilon$ defined in \eqref{eq:tube eta}.
    Then the functions $\psi_\varepsilon^{(k)}$ defined in \eqref{eq:def quasi-mode tube} verify \cref{hyp:quasi-mode} with
    \begin{equation}\label{eq:tube constants}
        \Kek = \frac{\Phi_\varepsilon^{(k)}}{|\Omega_\varepsilon|}
        = \frac{(\rek)^2}{|\Omega|} \frac{|\mathcal{D}_k|}{L_k} \bigl( 1 + \mathrm{O}(\reo) \bigr),
        \qquad
        C = 1,
    \end{equation}
    and
    \begin{equation}\label{eq:tube errors}
        \mathcal{E}_1(\Kall) = \frac{\Bigl( \sum_{j=1}^{\ndoors} \bigl( K_\varepsilon^{(j)} \bigr)^{1/2} \Bigr)^{3/2}}{\min_{1 \leq j \leq \ndoors} \bigl( K_\varepsilon^{(j)} \bigr)^{1/4}},
        \qquad
        \mathcal{E}_2(\Kall) = \sum_{j=1}^{\ndoors} \bigl( K_\varepsilon^{(j)} \bigr)^{1/2} .
    \end{equation}
    Furthermore,
    the mean exit time and exit position of the reflected Brownian motion $(X_t)_{t \geq 0}$ on $\Omega_\varepsilon$
    killed on the holes defined in~\eqref{eq:tube def} are given by
    \begin{equation}\label{eq:tube eigenvalue}
        \mathbb{E}_{\nu_\varepsilon}[\tau_\varepsilon]
        = |\Omega|
        \Bigl( \sum_{k=1}^{\ndoors} \frac{(\rek)^2 |\mathcal{D}_k|}{L_k} \Bigr)^{-1}
        \bigl( 1 + \mathrm{O}(\eta_\varepsilon) \bigr).
    \end{equation}
    Then the exit position is independent of the exit time, and is distributed as, for $k \in \{1, \ldots, N\}$,
    \begin{equation}\label{eq:tube exit law}
        \mathbb{P}_{\nu_\varepsilon}\bigl( \kappa_\varepsilon = k \bigr)
        = \frac{(\rek)^2 |\mathcal{D}_k| / L_k}{\sum_{j=1}^{\ndoors} (r_\varepsilon^{(j)})^2 |\mathcal{D}_j| / L_j}
        + \mathrm{O}(\eta_\varepsilon) .
    \end{equation}
\end{prop}

\begin{rema}
    When the radii are comparable, i.e.\ $\min_j r_\varepsilon^{(j)} \geq c \reo$,
    then $\eta_\varepsilon \leq C \reo$, and the errors in \eqref{eq:tube eigenvalue} and \eqref{eq:tube exit law} are $\mathrm{O}(\reo)$.
    For radii $r_\varepsilon^{(j)} = \varepsilon^{\alpha_j}$,
    \eqref{eq:tube radii} holds if and only if $\max_j \alpha_j < 3 \min_j \alpha_j$,
    and $\eta_\varepsilon = \varepsilon^{(3 \min_j \alpha_j - \max_j \alpha_j)/2}$.
\end{rema}

\begin{proof}
    The asymptotics of $\Kek$ in \eqref{eq:tube constants} follow from \eqref{eq:flux expansion},
    as $|\Omega_\varepsilon| = |\Omega| + \mathrm{O}((\reo)^2)$.
    In particular, there is $c > 0$ such that $c (\rek)^2 \leq \Kek \leq c^{-1} (\rek)^2$ for all $k$,
    so that, by \eqref{eq:tube eta} and \eqref{eq:tube errors},
    $\mathcal{E}_1(\Kall)$ is bounded above and below by constant multiples of $\eta_\varepsilon$,
    and $\mathcal{E}_2(\Kall)$ by constant multiples of $\reo$.
    As $\reo \leq \eta_\varepsilon$, the errors of \cref{thm:narrow escape} are thus $\mathrm{O}(\eta_\varepsilon)$, and they vanish by \eqref{eq:tube radii}.
    Fix $k \in \{1, \ldots, N\}$ and let us check all the hypotheses in \cref{hyp:quasi-mode}.

    \paragraph{PDE.}
    By \eqref{eq:tube neumann problem} with $\Phi = \Phi_\varepsilon^{(k)}$,
    $\Delta \psi_\varepsilon^{(k)} = \Phi_\varepsilon^{(k)} / |\Omega_\varepsilon| = \Kek$ in $\Omega_\varepsilon$
    and $\partial_n \psi_\varepsilon^{(k)} = 0$ on $\Gne$.

    \paragraph{Dirichlet estimate.}
    By \cref{prop:dirichlet tube},
    $\psi_\varepsilon^{(k)}$ is continuous on $\overline{\Gde}$ and
    $\|\psi_\varepsilon^{(k)} - \mathbbm{1}_{\overline{\Gke}}\|_{L^\infty(\Gde)} \leq C \eta_\varepsilon$,
    which is the Dirichlet condition of \eqref{eq:hyp:PDE quasi-mode} with $\mathcal{E}_1$ given in \eqref{eq:tube errors}.

    \paragraph{Normalization.}
    As $\psi_\varepsilon^{(k)}$ has zero mean on $\Omega_\varepsilon$,
    \cref{lem:poincare tube}, \eqref{eq:tube energy psi} and \eqref{eq:flux expansion} give
    \[
        \|\psi_\varepsilon^{(k)}\|_{L^2(\Omega_\varepsilon)}
        \leq C \|\nabla \psi_\varepsilon^{(k)}\|_{L^2(\Omega_\varepsilon)}
        = C \sqrt{\Phi_\varepsilon^{(k)}}
        \leq C \reo,
    \]
    which is \eqref{eq:hyp:norm} with $\mathcal{E}_2$ given in \eqref{eq:tube errors}.

    \paragraph{Eigenvalue.}
    Both techniques used in \cref{lem:eigenvalue spherical holes} and in the proof of \cref{prop:narrow escape slit} apply to prove \eqref{eq:hyp:eig}, 
    with no additional complexity.

    \paragraph{Conclusion.}
    Having checked both hypotheses \eqref{eq:hyp:eig} and \eqref{eq:hyp:PDE quasi-mode}--\eqref{eq:hyp:norm} of \cref{hyp:quasi-mode},
    \cref{thm:narrow escape} holds.
\end{proof}

\section{Numerical simulations}
\label{sec:numerics}
This section has two goals:
checking the asymptotics of \cref{tab:summary} in dimension $3$, including the exit hole distribution,
and testing their natural extension to dimensions $d \geq 4$, 
for which we do not provide a proof, but which can be studied with the same methods.
The method is described in \cref{sub:algorithm},
the three-dimensional results are presented in \cref{sub:results 3d},
and the higher-dimensional ones in \cref{sub:results higher dim}.

\subsection{Method}
\label{sub:algorithm}

\paragraph{The Walk on Stars.}
Narrow escape simulations are costly, since $\tau_\varepsilon$ diverges as $\varepsilon \to 0$ \cite{Caginalp, nous}.
An Euler--Maruyama discretization of \eqref{eq:reflected bm} must have steps smaller than the holes:
it needs a time step $h \ll \varepsilon^2$, hence it requires more than $\mathbb{E}[\tau_\varepsilon]/\varepsilon^2$ steps per trajectory.
In \cite[Section 4.1]{nous},
this cost was reduced by Walk on Spheres (WoS) steps \cite{Muller1956, Sabelfeld1991} away from the boundary,
but Euler--Maruyama steps with $h = 10^{-11}$ remained necessary near the reflecting boundary.
Finite element methods, on the other hand, are impractical to mesh domains of dimension $d \geq 4$ and are computationally expensive for delicate geometries such as the tubes \eqref{eq:tube def}.
We therefore use the \textit{Walk on Stars} (WoSt) algorithm \cite{walkonstars},
which extends WoS to Neumann boundary conditions
and belongs to a recent family of grid-free Monte Carlo solvers \cite{SawhneyCrane2020, Sawhney2022, Miller2024}.
WoSt has neither time step nor mesh:
it only queries distances to the boundary and intersections of rays with it, in any dimension.

Let us give the main ideas of the \textit{Walk on Stars} algorithm, 
following \cite[Section 4]{walkonstars}.
The mean exit time $u(x) = \mathbb{E}_x[\tau_\varepsilon]$ solves
\begin{equation}\label{eq:mean exit time pde}
    \left\{
    \begin{aligned}
        -\Delta u & = 1 && \text{in } \Omega_\varepsilon,\\
        u & = 0 && \text{on } \Gde,\\
        \partial_n u & = 0 && \text{on } \Gne.
    \end{aligned}
    \right.
\end{equation}
Fix $x \in \Omega_\varepsilon$ and $R > 0$ such that the ball $B(x, R)$ does not meet $\Gde$,
and let $\mathrm{St}(x, R) = B(x, R) \cap \Omega_\varepsilon$.
We assume that $\mathrm{St}(x, R)$ is star-shaped with respect to $x$
which can be ensured by reducing $R$.
Denote by
\[
    G_R(x, y) = \frac{|y - x|^{2-d} - R^{2-d}}{(d-2)\, |\mathbb{S}^{d-1}|}
\]
the Dirichlet Green function of $B(x, R)$.
Green's formula on $\mathrm{St}(x, R)$ and \eqref{eq:mean exit time pde} give
\begin{equation}\label{eq:green representation}
    u(x) = \int_{\mathrm{St}(x, R)} G_R(x, y) \, \d y
    + \int_{\partial \mathrm{St}(x, R)} u(z) \, \bigl( -\partial_{n_z} G_R(x, z) \bigr) \, \d\sigma(z).
\end{equation}
Following the same approach as \cite{walkonstars}, we can write both integrals of \eqref{eq:green representation} 
in polar coordinates around $x$ and obtain a recursive formula for $u(x)$ that can be used for sampling:
\begin{equation}\label{eq:wost recursion}
    u(x) = \mathbb{E}_\nu \Bigl[ \tau_R\bigl(t(\nu)\bigr) + u\bigl(x + t(\nu)\, \nu\bigr) \Bigr],
    \qquad \nu \sim \mathcal{U}(\mathbb{S}^{d-1}),
\end{equation}
where
\begin{equation}\label{eq:time increment}
    \tau_R(t) = |\mathbb{S}^{d-1}| \int_0^t G_R(x, x + \rho\, \nu)\, \rho^{d-1} \, \d\rho
    = \frac{1}{d-2} \left( \frac{t^2}{2} - \frac{t^d}{d\, R^{d-2}} \right).
\end{equation}
Formula \eqref{eq:wost recursion} expresses $u(x)$ through the value of $u$ at the random point $x_1 = x + t(\nu)\, \nu$.
Applying it again at $x_1$, then at $x_2$, and so on, gives
\[
    u(x_0) = \mathbb{E}\Bigl[ \sum_{n \geq 0} \tau_{R_n}(t_n) \Bigr],
\]
where the sum runs over the steps of the random walk $(x_n)_{n \geq 0}$ until it reaches $\Gde$, where $u = 0$.
A similar representation holds for the exit hole $\kappa_\varepsilon$, 
which is the index $k$ of the hole on which it has escaped.
This is the \textit{Walk on Stars} algorithm, detailed in \cref{alg:wost}.

\begin{algorithm}
\caption{Walk on Stars for the mean exit time and the exit hole.}
\label{alg:wost}
\begin{algorithmic}[1]
\Require starting point $x_0 \in \Omega_\varepsilon$, shell width $\eta > 0$
\State $x \gets x_0$, \quad $T \gets 0$
\Loop
    \If{$d(x, \Gke) < \eta$ for some $k \in \{1, \ldots, N\}$}
        \State \Return $(T, k)$ \Comment{killed on the $k$-th hole}
    \EndIf
    \State $R \gets d(x, \Gde)$ \Comment{reduced if $\mathrm{St}(x, R)$ is not star-shaped}
    \If{$x \in \Gne$}
        \State draw $\nu$ uniformly among the inward directions at $x$
    \Else
        \State draw $\nu$ uniformly on $\mathbb{S}^{d-1}$
    \EndIf
    \State $s \gets$ distance from $x$ to $\Gne$ in the direction $\nu$
    \State $t \gets \min\{R, s\}$
    \State $T \gets T + \tau_R(t)$ \Comment{see \eqref{eq:time increment}}
    \State $x \gets x + t\, \nu$ \Comment{$x \in \Gne$ if $t = s$}
\EndLoop
\end{algorithmic}
\end{algorithm}

\paragraph{Explanation of \cref{alg:wost}.}
The walk is constructed as follows.
At the current position $x$, we take the radius $R = d(x, \Gde)$, 
so that the ball $B(x, R)$ does not intersect the killing region.
We then draw a direction $\nu$:
uniformly on the whole sphere $\mathbb{S}^{d-1}$ if $x$ lies inside $\Omega_\varepsilon$,
and uniformly on the inward half-sphere if $x \in \Gne$.
The latter follows from \eqref{eq:green representation}:
at a smooth point of $\Gne$, its left-hand side becomes $u(x)/2$, while the inward directions only cover half of the sphere,
and the two factors $1/2$ cancel
(see \cite[Footnote 1 and Section 4.4]{walkonstars}).
Next, we follow the ray from $x$ in the direction $\nu$ until it leaves $\mathrm{St}(x, R)$:
either it reaches the sphere $\partial B(x, R)$ ($t = R$),
or it hits $\Gne$ first ($t = s < R$).
The time $T$ is then increased by $\tau_R(t)$, 
the integral of $G_R(x, \cdot)$ along the segment $[x, x + t\nu]$.
If $t = R$, this is $\tau_R(R) = R^2/(2d)$, the mean exit time of $(X_t)_{t \geq 0}$ from $B(x, R)$,
and the step is a classical Walk on Spheres step \cite{Muller1956};
if $t < R$, only the part of the ball lying in $\Omega_\varepsilon$ contributes, and $\tau_R(t) < R^2/(2d)$.
We compute $\tau_R(t)$ in closed form, instead of sampling a point with density proportional to $G_R(x, \cdot)$ as in \cite{walkonstars}, which reduces the variance.
Finally, the walker moves to $x + t\nu$.

Since $t \leq d(x, \Gde)$, a step of the walk never crosses $\Gde$,
similarly to Walk on Spheres.
We therefore kill the walk as soon as it enters the shell $\{ d(\cdot, \Gke) < \eta \}$ of some hole $k$,
where $\eta$ is a small parameter, in practice $\eta = 10^{-4} \varepsilon$.
This introduces a relative bias of order $\eta/\varepsilon$ \cite{MascagniHwang2003}, well below the statistical errors
(see~\cite{BinderBraverman2012} for the influence of $\eta$ on the number of steps of Walk on Spheres).

\paragraph{Initial distribution.}
WoSt simulates no trajectory, hence cannot sample $\nu_\varepsilon$.
All walks start from the uniform distribution on $\Omega_\varepsilon$, which is a proxy for $\nu_\varepsilon$.
This is justified by the fact that the density $v_\varepsilon$ of $\nu_\varepsilon$ is close to $\varphi_\varepsilon / |\Omega_\varepsilon|$, where the quasi-mode $\varphi_\varepsilon$ is itself 
built from a constant function on $\Omega_\varepsilon$, plus a small correction $\psi_\varepsilon^{(k)}$.
In the rest of this section, 
$\mathbb{E}$ and $\mathbb{P}$ without subscript refer to this uniform initial distribution, 
while $\mathbb{E}_{\nu_\varepsilon}$ and $\mathbb{P}_{\nu_\varepsilon}$ refer to the quasi-stationary one as in the previous sections.

As a check,
starting instead from the center of the ball (Type B, $d = 3$),
which shifts $\mathbb{E}[\tau_\varepsilon]$ by an additive constant \cite[Section 4.2]{nous},
gives the same mean exit times within statistical error for $\varepsilon \leq 10^{-2}$.

\paragraph{Implementation and statistics.}
The code is written in Rust \cite{rust}, generic in the dimension $3 \leq d \leq 16$,
and parallelized over the walks with Rayon \cite{rayon}.
The domains are the unit ball and the cube $[-1, 1]^d$.
Distances and ray intersections are computed in closed form,
except for the spheroidal holes of Type C,
whose exact distance is computed by the robust bisection of \cite{Eberly2013},
accurate down to $w = 10^{-8}$.
With $n$ walks per hole size, the relative standard error on $\mathbb{E}[\tau_\varepsilon]$ is close to $n^{-1/2}$,
that is $0.16\%$ for $n = 4 \cdot 10^5$;
error bars in the figures are one standard error.
In total,
the $64$ experiments, including the validations and variations of the tube length,
comprise $563$ parameter values and $5.9 \cdot 10^7$ walks,
that is $1\,786$ core-hours on a single $32$-core node.
A table of the parameters of each experiment is given in \cref{tab:numerics}.

\begin{table}[h!]
    \centering
    \small
    \setlength{\tabcolsep}{4pt}
    \begin{tabular}{@{} c c l >{\raggedright\arraybackslash}p{6.5cm} l l @{}}
        \toprule
        Fig. & $d$ & $\Omega$ & Holes & Smallest size & Walks \\
        \midrule
        \ref{fig:time 3d} & $3$ & cube & A, B: one hole centered on the face $x_1 = 1$ & $\varepsilon = 10^{-6}$ & $4 \cdot 10^5$ \\
         & $3$ & cube & C: one slit centered on the face $x_1 = 1$ & $w = 10^{-8}$ & $2 \cdot 10^5$ \\
         & $3$ & ball & D: one tube at $e_1$ & $\varepsilon = 1.1 \cdot 10^{-3}$ & $10^5$ \\
        \addlinespace
        \ref{fig:exit 3d} & $3$ & ball & A, B: radii $\varepsilon$, $2\varepsilon$ at $\pm e_1$ & $\varepsilon = 3 \cdot 10^{-6}$ & $4 \cdot 10^5$ \\
         & $3$ & cube & C: two slits on the faces $x_1 = \pm 1$ & $w = 10^{-8}$ & $2 \cdot 10^5$ \\
         & $3$ & ball & D: two tubes of radii $\varepsilon$, $2\varepsilon$ at $\pm e_1$ & $\varepsilon = 1.6 \cdot 10^{-3}$ & $4 \cdot 10^5$ \\
        \addlinespace
        \ref{fig:hd AB} & $3$--$16$ & cube & radius $\varepsilon$ at $x_1 = 1$ & $\varepsilon = 2 \cdot 10^{-7}$ & $3 \cdot 10^4$ \\
        \ref{fig:hd slit} & $4$--$10$ & cube & half-width $w$ on the face $x_1 = 1$ & $w = 1.5 \cdot 10^{-6}$ & $3 \cdot 10^4$ \\
        \ref{fig:hd tube} & $4$--$6$ & ball & one tube of radius $\varepsilon$ at $e_1$ & $\varepsilon = 5 \cdot 10^{-3}$ & $2$--$3 \cdot 10^4$ \\
        \ref{fig:hd slit} & $4$ & cube & one planar slit on the face $x_1 = 1$ & $w = 10^{-8}$ & $2 \cdot 10^5$ \\
        \bottomrule
    \end{tabular}
    \caption{Numerical setup.
    Slits lie along $e_2$ (along $e_2$ and $e_3$ for planar slits), with half-length $\ell = 1/2$; tubes have length $L = 1$.
    Hole sizes are geometrically spaced, with $8$ to $16$ values per experiment;
    the smallest size is the minimum over all dimensions.
    The number of walks per size is reduced (down to $10^4$) for the two or three smallest sizes.
    The tubes in dimension $3$ in \cref{fig:hd tube} are those of \cref{fig:time 3d}.}
    \label{tab:numerics}
\end{table}

\subsection{Results in dimension \texorpdfstring{$3$}{3}}
\label{sub:results 3d}

\begin{figure}
    \centering
    \includegraphics[width=0.7\textwidth]{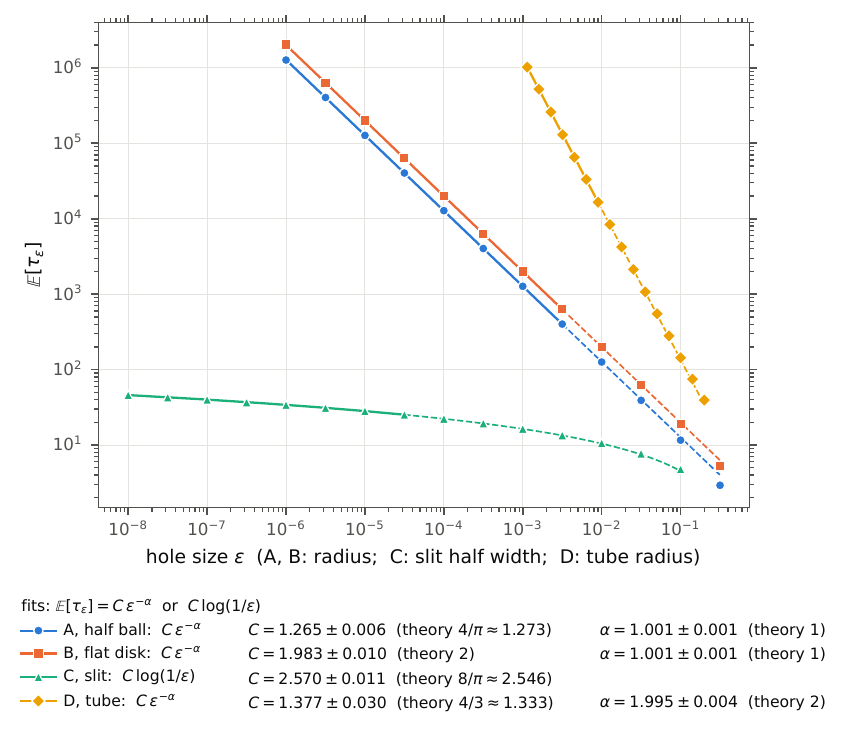}
    \caption{Mean exit time $\mathbb{E}[\tau_\varepsilon]$ for a single hole of each type in dimension $3$:
    A, B and C on a face of the cube $[-1,1]^3$ ($|\Omega| = 8$),
    D at the end of a tube of length $L = 1$ attached to the unit ball ($|\Omega| = 4\pi/3$).
    Lines: weighted least-squares fits, solid on the fitted range and dashed outside.
    Theoretical values from \cref{tab:summary}.}
    \label{fig:time 3d}
\end{figure}

\paragraph{Mean exit time.}
\Cref{fig:time 3d} shows $\mathbb{E}[\tau_\varepsilon]$ for a single hole of each type,
with fits of the form predicted by \cref{tab:summary} on the smallest holes.
The fitted exponents agree with \cref{tab:summary} within $0.3\%$,
and the prefactors within $1\%$ for A, B and C, and within $3.3\%$ ($1.5$ standard errors) for D.

\begin{figure}
    \centering
    \includegraphics[width=0.7\textwidth]{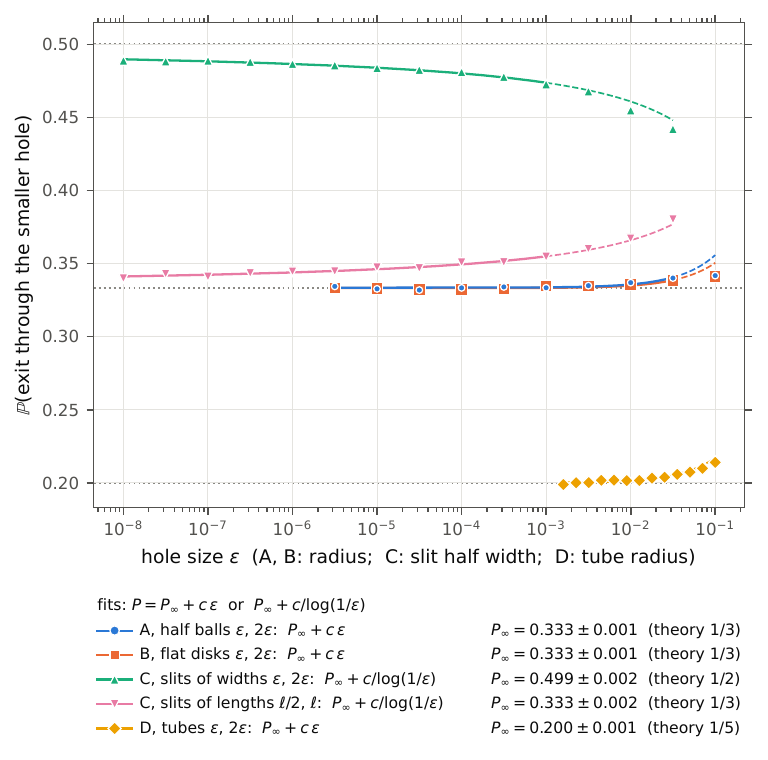}
    \caption{Probability of exiting through the smaller of two holes in dimension $3$:
    A (half-balls), B (geodesic caps) and D in the unit ball, with sizes $\varepsilon$ and $2\varepsilon$ at the poles $\pm e_1$;
    C on opposite faces of the cube, either with half-widths $w$ and $2w$ (same $\ell$),
    or with half-lengths $\ell/2$ and $\ell$ (same $w$).
    Dotted lines: limits $1/5$, $1/3$ and $1/2$.
    Lines: fits.}
    \label{fig:exit 3d}
\end{figure}

\paragraph{Exit hole.}
\Cref{fig:exit 3d} shows the probability of exiting through the smaller of two holes.
By Propositions~\ref{prop:narrow escape spherical inside}, \ref{prop:quasi-mode spherical holes} and~\ref{prop:quasi-mode tube},
it tends to $1/3$ for A and B
and to $1/5$ for D.
For slits, in the case of two slits of half-widths $w$ and $2w$, they are asymptotically equivalent ($1/2$),
whereas in the case of two slits of half-lengths $\ell/2$ and $\ell$, they are not ($1/3$).
This is the three-dimensional counterpart of the two-dimensional phenomenon of \cite[Remark 1.7]{nous}.

\subsection{Extending the results to higher dimensions}
\label{sub:results higher dim}
The constructions of \cref{sec:several narrow escape problems,sec:slit,sec:tubular holes} have formal analogues in any dimension $d \geq 4$.
The expected scalings are given in \cref{tab:capacities}.
Only for~\eqref{eq:spherical holes} have these results been proved, in \cite[Theorems 1.4 and 1.6]{nous}. 
\Cref{fig:hd AB,fig:hd slit,fig:hd tube} 
illustrate that the scalings of \cref{tab:capacities} are verified up to $d = 16$.
In \cref{fig:hd tube}, the simulations are stopped at $d=6$ as 
the \textit{Walk on Stars} algorithm is not efficient for non-convex geometries. 
Around the connection with the tube, 
$R$ has to be small so that $\mathrm{St}(x,R)$ stays star-shaped,
hence the number of steps of the walk is much higher than in the other geometries.

\begin{table}[h!]
    \centering
    \small
    \renewcommand{\arraystretch}{1.5}
    \begin{tabular}{@{} l l c @{}}
        \toprule
        Type & Limit set & $\mathbb{E}[\tau_\varepsilon] \propto$ \\
        \midrule
        A & point & $\varepsilon^{-(d-2)}$ \\
        B & point & $\varepsilon^{-(d-2)}$ \\
        C, $d \geq 4$ & segment & $w^{-(d-3)}$ \\
        C, planar, $d = 4$ & disk & $\log(1/w)$ \\
        D & --- & $\varepsilon^{-(d-1)}$ \\
        \bottomrule
    \end{tabular}
    \caption{Leading-order scaling of the mean exit time in dimension $d$.}
    \label{tab:capacities}
\end{table}

\begin{figure}
    \centering
    \includegraphics[width=0.8\textwidth]{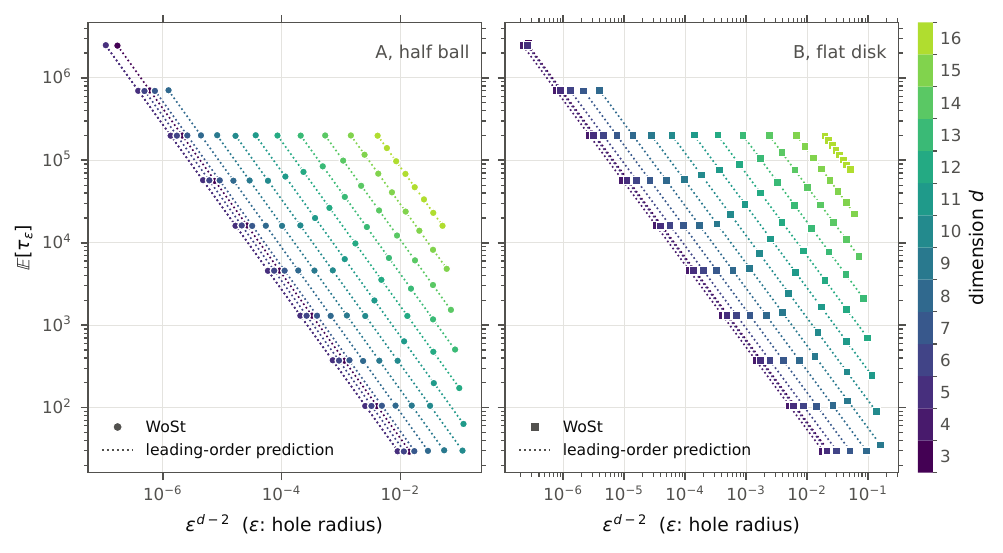}
    \caption{Types A (left) and B (right) in the cube $[-1,1]^d$, $3 \leq d \leq 16$:
    $\mathbb{E}[\tau_\varepsilon]$ against $\varepsilon^{d-2}$. Dotted: prediction.}
    \label{fig:hd AB}
\end{figure}

\begin{figure}
    \centering
    \begin{subfigure}{0.49\textwidth}
        \includegraphics[width=\textwidth]{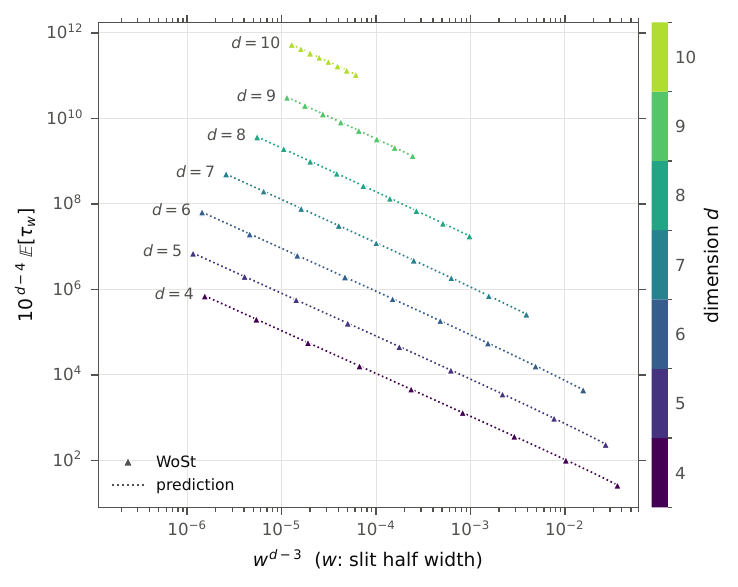}
    \end{subfigure}
    \hfill
    \begin{subfigure}{0.49\textwidth}
        \includegraphics[width=\textwidth]{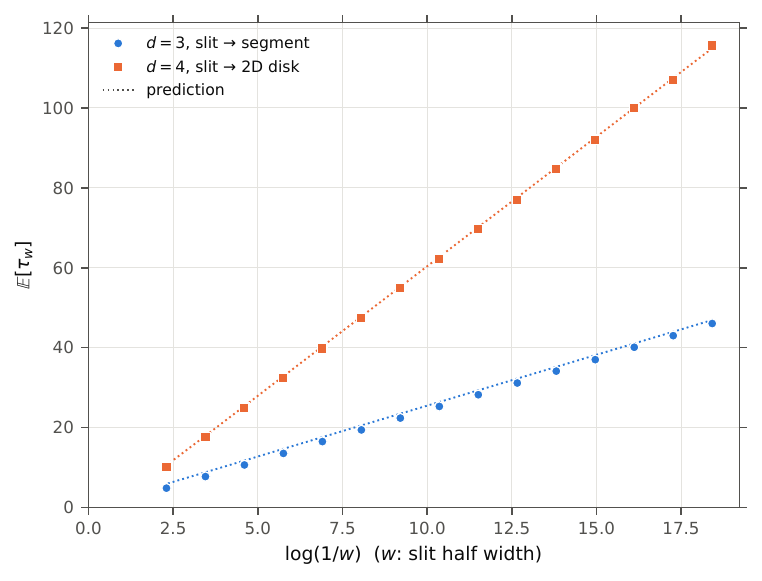}
    \end{subfigure}
    \caption{Type C.
    Left: line slits in the cube $[-1,1]^d$, $4 \leq d \leq 10$, $\mathbb{E}[\tau_w]$ against $w^{d-3}$ (curves offset by $10^{d-4}$).
    Right: holes shrinking to a set of codimension $2$, a segment in $d = 3$ and a disk in $d = 4$, $\mathbb{E}[\tau_w]$ against $\log(1/w)$.
    Dotted: prediction.}
    \label{fig:hd slit}
\end{figure}

\begin{figure}
    \centering
    \includegraphics[width=0.6\textwidth]{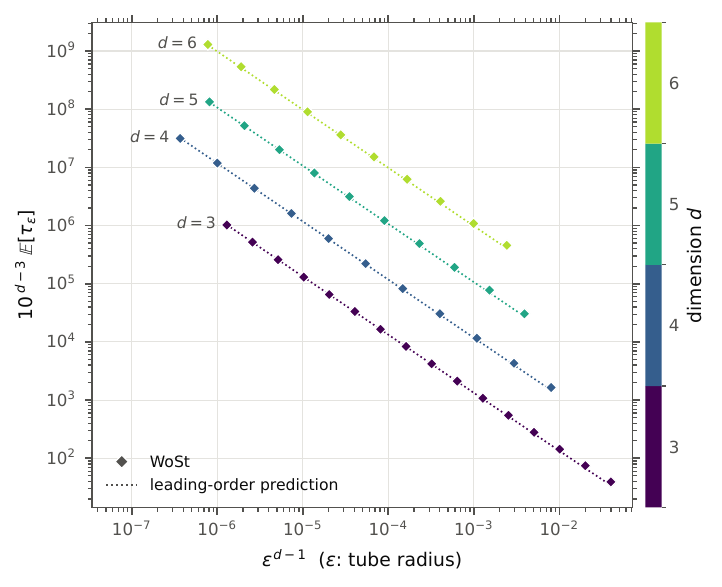}
    \caption{Type D: tube of radius $\varepsilon$ and length $L = 1$ attached to the unit ball of $\real^d$, $3 \leq d \leq 6$,
    $\mathbb{E}[\tau_\varepsilon]$ against $\varepsilon^{d-1}$ (curves offset by $10^{d-3}$). Dotted: prediction.}
    \label{fig:hd tube}
\end{figure}

\clearpage

\section*{Acknowledgments and AI statement}
Part of this work is funded by the Agence
Nationale de la Recherche through the projects IPSO (ANR-23-CE40-0027) and DySLoS (ANR-25-CE40-6875).
The author has received funding from the European Research Council (ERC) 
under the European Union’s Horizon 2020 research and innovation programme (grant agreement No 810367), project EMC2.

LLMs (Claude Opus 5.5) have been used to improve the general writing and presentation of the manuscript and correct typographical errors. 
Through this process a simpler yet more elegant proof of the results presented in \cref{sub:definition of the quasi-mode} was suggested by the LLM. 
The tikz figures were generated using Claude Opus. 
The author has verified the correctness of the results and proofs presented in this manuscript.

\appendix
\section{On the \texorpdfstring{$\varepsilon$}{epsilon}-uniform Poincaré--Wirtinger inequality}
\label{app:poincare tube}

In this section, we prove that the Poincaré--Wirtinger inequality holds for all geometries considered in this work with a constant
independent of $\varepsilon$.
For the \eqref{eq:spherical holes}, \eqref{eq:geodesic spherical holes} and \eqref{eq:slit def} geometries, 
this is a classical result, see~\cite[Theorem~1.2]{Ruiz},
as the domains are bounded and satisfy the uniform cone condition.
This is not the case for the tubular holes geometry \eqref{eq:tube def}.
However, the uniformity of the constant in the Poincaré--Wirtinger inequality on domains with thin channels has been studied in the literature (see~\cite{Arrieta1995, JimboKosugi2009}).
In particular, the convergence of the Neumann spectrum of the operator defined on~$\Omega_\varepsilon$ to the one of $\Omega$
by \cite{Arrieta1995} would be sufficient to prove the uniform Poincaré--Wirtinger inequality, 
and it would lead to a constant $\max\{\mu_2(\Omega)^{-1/2}, 2\ell/\pi\}$, where $\ell$ is the length of the longest tube and 
$\mu_2(\Omega)$ is the first non-zero Neumann eigenvalue of $\Omega$.
However, the author of~\cite{Arrieta1995} assume that the domains are smooth, which is not the case for $\Omega_\varepsilon$ for the holes geometry considered in~\cref{sec:tubular holes}.
We therefore provide a self-contained proof of the uniform Poincaré--Wirtinger inequality
in this setting with a constant independent of $\varepsilon$.

\begin{lem}\label{lem:poincare tube}
    Consider the geometry of \cref{sec:tubular holes}, introduced in~\eqref{eq:tube def} with the domains $\Omega_\varepsilon$ defined in~\eqref{eq:tube domain}.
    There is $C > 0$, independent of $\varepsilon$, such that for any
    $u \in H^1(\Omega_\varepsilon)$, it holds
    \[
        \left\|u - \frac{1}{|\Omega_\varepsilon|}\int_{\Omega_\varepsilon} u\right\|_{L^2(\Omega_\varepsilon)}
        \leq C \|\nabla u\|_{L^2(\Omega_\varepsilon)}.
    \]
\end{lem}

\begin{proof}
    Fix $u \in H^1(\Omega_\varepsilon)$ and denote $\bar u = \frac{1}{|\Omega|}\int_\Omega u$. 
    Since the mean over
    $\Omega_\varepsilon$ minimizes $k \mapsto \|u - k\|_{L^2(\Omega_\varepsilon)}$
    over constants, it suffices to bound $\|u - \bar{u}\|_{L^2(\Omega_\varepsilon)}$.
    Since $\Tke$ are pairwise disjoint, 
    by the definition of $\Omega_\varepsilon$ in~\eqref{eq:tube domain}, we have
    \begin{equation}\label{eq:poincare split}
        \|v\|_{L^2(\Omega_\varepsilon)}^2
        = \|v\|_{L^2(\Omega)}^2 + \sum_{k=1}^{\ndoors}\|v\|_{L^2(\Tke)}^2
        \qquad (v \in L^2(\Omega_\varepsilon)),
    \end{equation}
    and the same identity holds for $\nabla u$. We bound each term separately.

    \paragraph{Step 1 (the fixed domain).}
    Since $\Omega$ is bounded, connected, Lipschitz and independent of
    $\varepsilon$, there is a constant $C_P > 0$ such that the Poincaré--Wirtinger inequality holds:
    \begin{equation}\label{eq:poincare fixed}
        \|u - \bar u\|_{L^2(\Omega)} \leq C_P \|\nabla u\|_{L^2(\Omega)}.
    \end{equation}
    For the same reason, there is a constant $C_{\mathrm{tr}} > 0$ such that the trace theorem holds, 
    for any $v \in H^1(\Omega)$,
    \begin{equation}\label{eq:trace fixed}
        \|v\|_{L^2(\partial\Omega)}^2
        \leq C_{\mathrm{tr}} \bigl( \|v\|_{L^2(\Omega)}^2 + \|\nabla v\|_{L^2(\Omega)}^2 \bigr),
    \end{equation}
    which is useful in the next step.

    \paragraph{Step 2 (the tubes).}
    Fix $k \in \{1, \ldots, \ndoors\}$ and denote $\ell_k = L_k + \sup_{\varepsilon} \ske$ 
    the maximal length of the $k$-th tube. We claim that
    \begin{equation}\label{eq:poincare tube}
        \|u - \bar u\|_{L^2(\Tke)}^2
        \leq 2\ell_k \, \|u - \bar u\|_{L^2(\Ske)}^2
        + \ell_k^2 \, \|\partial_s u\|_{L^2(\Tke)}^2 ,
    \end{equation}
    where $\partial_s u = \nabla u \cdot n_k$. 

    For fixed $\varepsilon$, $\Tke$ is a Lipschitz domain. Hence functions in
    $C^1(\overline{\Tke})$ are dense in $H^1(\Tke)$, and both sides
    of~\eqref{eq:poincare tube} are continuous, by
    the trace theorem on $\Tke$. It is therefore enough to
    prove~\eqref{eq:poincare tube} with~$u - \bar u$ replaced by some
    $w \in C^1(\overline{\Tke})$. We use the coordinates
    $x = \xk + \rho + s\,n_k$ of~\eqref{eq:tube lim}, and 
    write~$w(\rho, s) = w(\xk + \rho + s\,n_k)$. For $\rho \in \Dke$ and
    $s \in (h_k(\rho), L_k)$, the fundamental theorem of calculus gives
    \[
        w(\rho, s) = w\bigl(\rho, h_k(\rho)\bigr)
        + \int_{h_k(\rho)}^{s} \partial_s w(\rho, \sigma)\,\d\sigma .
    \]
    By $(a+b)^2 \leq 2a^2 + 2b^2$ and the Cauchy--Schwarz inequality,
    \[
        w(\rho, s)^2 \leq 2\, w\bigl(\rho, h_k(\rho)\bigr)^2
        + 2\bigl(s - h_k(\rho)\bigr) \int_{h_k(\rho)}^{L_k} |\partial_s w(\rho, \sigma)|^2 \,\d\sigma .
    \]
    We integrate in $s$ over $(h_k(\rho), L_k)$, an interval of length at most
    $\ell_k$, and use $\int_{h_k(\rho)}^{L_k} 2(s - h_k(\rho))\,\d s
    = (L_k - h_k(\rho))^2 \leq \ell_k^2$. Then we integrate in $\rho$ over $\Dke$.
    This yields
    \[
        \|w\|_{L^2(\Tke)}^2
        \leq 2\ell_k \int_{\Dke} w\bigl(\rho, h_k(\rho)\bigr)^2 \,\d\rho
        + \ell_k^2 \|\partial_s w\|_{L^2(\Tke)}^2 .
    \]
    Finally, $\Ske$ is the graph of
    $\rho \mapsto \xk + \rho + h_k(\rho)\,n_k$ over $\Dke$, so its surface element
    is
    \[
        \d \sigma = \sqrt{1 + |\nabla h_k(\rho)|^2}\,\d\rho \geq \d\rho .
    \]
    Therefore
    $\int_{\Dke} w(\rho, h_k(\rho))^2\,\d\rho \leq \|w\|_{L^2(\Ske)}^2$,
    which proves~\eqref{eq:poincare tube}.

    \paragraph{Step 3 (conclusion).}
    From their definition~\eqref{eq:def tube mouth}, the mouths $\Ske$ are pairwise
    disjoint subsets of $\partial\Omega$, so that
    $\sum_{k} \|u - \bar u\|_{L^2(\Ske)}^2 \leq \|u - \bar u\|_{L^2(\partial\Omega)}^2$.
    We sum~\eqref{eq:poincare tube} over $k$ and use $|\partial_s u| \leq |\nabla u|$
    together with~\eqref{eq:trace fixed} and~\eqref{eq:poincare fixed}. This gives
    \begin{equation}\label{eq:poincare poincare fixed}
        \sum_{k=1}^{\ndoors} \|u - \bar u\|_{L^2(\Tke)}^2
        \leq 2\ell\, C_{\mathrm{tr}} \bigl(1 + C_P^2\bigr) \|\nabla u\|_{L^2(\Omega)}^2
        + \ell^2 \sum_{k=1}^{\ndoors} \|\nabla u\|_{L^2(\Tke)}^2,
    \end{equation}
    where $\ell = \max_k \ell_k$ is the maximal length of all the tubes.
    Finally, we combine~\eqref{eq:poincare fixed} and~\eqref{eq:poincare poincare fixed} with~\eqref{eq:poincare split} to obtain
    \[
        \|u - \bar u\|_{L^2(\Omega_\varepsilon)}^2 \leq K \|\nabla u\|_{L^2(\Omega_\varepsilon)}^2,
        \qquad
        K = \max\Bigl\{ C_P^2 + 2\ell\, C_{\mathrm{tr}} \bigl(1 + C_P^2\bigr),\ \ell^2 \Bigr\} .
    \]
    The lemma holds with $C = K^{1/2}$.
\end{proof}

\section{Proof of \texorpdfstring{\cref{thm:narrow escape}}{the main theorem}}
\label{app:proof}
We provide here the proof of \cref{thm:narrow escape} in the general case,
following the steps made in \cite[Section 3]{nous}.
Throughout this section,
$(\cdot, \cdot)_{\Omega_\varepsilon}$ denotes the $L^2(\Omega_\varepsilon)$ scalar product,
$\langle \cdot, \cdot \rangle$ denotes the duality bracket between $H^{-\frac{1}{2}}(\partial \Omega_\varepsilon)$ and $H^{\frac{1}{2}}(\partial \Omega_\varepsilon)$,
and functions of $H^1(\Omega_\varepsilon)$ are identified with their traces on $\partial \Omega_\varepsilon$.

Let us first give a meaning to the normal derivatives introduced in \eqref{eq:domain L} and \eqref{eq:hyp:PDE quasi-mode}.
For $u \in H^1(\Omega_\varepsilon)$ such that $\Delta u \in L^2(\Omega_\varepsilon)$, we define the normal derivative $\partial_n u$
as an element of $H^{-\frac{1}{2}}(\partial \Omega_\varepsilon)$ by,
for all~$v \in H^{\frac{1}{2}}(\partial \Omega_\varepsilon)$,
\begin{equation}\label{eq:normal derivative}
    \langle \partial_n u , v \rangle
    = \int_{\Omega_\varepsilon} \nabla u \cdot \nabla w + \left(\Delta u\right)w,
\end{equation}
where $w$ is any extension of $v$ in $H^1(\Omega_\varepsilon)$, such that $\operatorname{tr}_{\partial \Omega_\varepsilon} w = v$.
Notice that the right-hand side of \eqref{eq:normal derivative} is null for zero trace functions,
hence the definition is indeed independent of the choice of $w$.
The Neumann condition $\partial_n u = 0$ on $\Gne$ is then understood as $\langle \partial_n u, v \rangle = 0$
for all $v \in H^{\frac{1}{2}}(\partial \Omega_\varepsilon)$ vanishing on $\Gde$.
Applying \eqref{eq:normal derivative} twice,
the gradient terms cancel and one obtains the Green formula:
for all $u_1, u_2 \in H^1(\Omega_\varepsilon)$ such that $\Delta u_1, \Delta u_2 \in L^2(\Omega_\varepsilon)$,
\begin{equation}\label{eq:green formula}
    \langle \partial_n u_1, u_2 \rangle - \langle \partial_n u_2, u_1 \rangle
    = \bigl( \Delta u_1, u_2 \bigr)_{\Omega_\varepsilon} - \bigl( u_1, \Delta u_2 \bigr)_{\Omega_\varepsilon}.
\end{equation}
In particular,
if $u_1 = 0$ on $\Gde$ and $\partial_n u_2 = 0$ on $\Gne$,
then $\langle \partial_n u_2, u_1 \rangle = 0$.

We now provide a technical lemma.
\begin{lem}\label{lem:quasi-mode scalar products}
    Under \cref{hyp:quasi-mode},
    consider the quasi-mode $\varphi_\varepsilon$ defined in \eqref{eq:def quasi-mode}
    and the density of the quasi-stationary distribution $v_\varepsilon$ defined in \eqref{eq:edp density}.
    Then, the following estimates hold
    \begin{equation}\label{eq:bound scalar}
        \bigl( v_\varepsilon, \varphi_\varepsilon \bigr)_{\Omega_\varepsilon} = 1 + \mathrm{O}(\mathcal{E}_2(\Kall)),
    \end{equation}
    and
    \begin{equation}\label{eq:bound normal}
        \langle \partial_n v_\varepsilon , \varphi_\varepsilon \rangle = \mathrm{O}(\lambda_\varepsilon \mathcal{E}_1(\Kall)).
    \end{equation}
\end{lem}

\begin{proof}
    We start by showing that the $L^2$ norm of $v_\varepsilon$ is bounded uniformly in $\varepsilon$.
    By the uniform in $\varepsilon$ Poincaré--Wirtinger inequality on $\Omega_\varepsilon$
    (see \cite[Theorem 1.2]{Ruiz} for \eqref{eq:spherical holes}--\eqref{eq:slit def} and \cref{lem:poincare tube} for \eqref{eq:tube def}),
    there exists $C_{\mathrm{PW}} > 0$ independent of $\varepsilon$ such that
    \begin{align}\label{pw}
        \left\|v_\varepsilon-\frac{1}{|\Omega_\varepsilon|}\left( v_\varepsilon  , 1\right)_{\Omega_\varepsilon}\right\|_{L^2(\Omega_\varepsilon)}^2
        \leq C_{\mathrm{PW}} \|\nabla v_\varepsilon\|_{L^2(\Omega_\varepsilon)}^2 = C_{\mathrm{PW}} \lambda_\varepsilon \| v_\varepsilon\|_{L^2(\Omega_\varepsilon)}^2,
    \end{align}
    where we used \eqref{eq:edp density} to obtain the equality.
    Since $\left\lVert v_\varepsilon \right\rVert_{L^1(\Omega_\varepsilon)} = 1$, the left-hand side satisfies
    \begin{align}\label{lhs}
        \left\|v_\varepsilon-\frac{1}{|\Omega_\varepsilon|}\left( v_\varepsilon  , 1\right)_{\Omega_\varepsilon}\right\|_{L^2(\Omega_\varepsilon)}^2 = \|v_\varepsilon\|_{L^2(\Omega_\varepsilon)}^2-\frac{1}{|\Omega_\varepsilon|} .
    \end{align}
    Plugging \eqref{pw} into
    \eqref{lhs},
    one obtains
    \begin{align}\label{mino}
        \|v_\varepsilon\|_{L^2(\Omega_\varepsilon)}^2 (1 - C_{\mathrm{PW}} \lambda_\varepsilon)
        \leq  |\Omega_\varepsilon|^{-1}.
    \end{align}
    Since, for all kinds of holes it holds
    \[
        |\Omega| = |\Omega_\varepsilon| + \mathrm{o}\left( 1\right),
    \]
    and since $1 - C_{\mathrm{PW}} \lambda_\varepsilon \geq \frac{1}{2}$ for $\varepsilon$ sufficiently small by
    \eqref{eq:hyp:eig},
    the inequality \eqref{mino} leads to
    \begin{equation} \label{eq:bound u_0}
        \|v_\varepsilon\|_{L^2(\Omega_\varepsilon)} \leq M,
    \end{equation}
    for some constant $M>0$, independent of $\varepsilon$.

    Let us now move on to the estimates claimed by this lemma.
    From the Cauchy--Schwarz inequality, \eqref{eq:bound u_0},
    the~$L^1$ normalization of $v_\varepsilon$
    and \eqref{eq:quasi-mode properties}, we have
    \[
        \left| \bigl( v_\varepsilon, \varphi_\varepsilon \bigr)_{\Omega_\varepsilon} - 1\right|
        = \left| \bigl( v_\varepsilon, \varphi_\varepsilon - 1 \bigr)_{\Omega_\varepsilon}\right|
        \leq \left\| v_\varepsilon \right\|_{L^2(\Omega_\varepsilon)} \left\| \varphi_\varepsilon  - 1\right\|_{L^2(\Omega_\varepsilon)}
        = \mathrm{O}(\mathcal{E}_2(\Kall)).
    \]

    For the normal derivative,
    recall from \cite[Proposition 1.3]{nous} that $-\lambda_\varepsilon^{-1} \partial_n v_\varepsilon$ is the law of $X_{\tau_\varepsilon}$.
    Hence,
    $\partial_n v_\varepsilon$ can be identified with a nonpositive Radon measure supported on $\overline{\Gde}$,
    of total mass $\left| \partial_n v_\varepsilon(\overline{\Gde})\right| = \lambda_\varepsilon$,
    so that, for all $g \in C^0(\overline{\Gde})$,
    \begin{equation}\label{eq:measure bound}
        \left| \int_{\overline{\Gde}} g \, \d (\partial_n v_\varepsilon) \right|
        \leq \lambda_\varepsilon \|g\|_{L^\infty(\Gde)}.
    \end{equation}
    Applying \eqref{eq:measure bound} to $g = \varphi_\varepsilon$,
    together with the Dirichlet condition of \eqref{eq:quasi-mode properties},
    gives \eqref{eq:bound normal}.
\end{proof}

We can now move on to the proof of \cref{thm:narrow escape}.
\begin{proof}[Proof of \cref{thm:narrow escape}]
    From \cite[Proposition 1.3]{nous},
    starting from $\nu_\varepsilon$,
    the exit time $\tau_\varepsilon$ is exponentially distributed with parameter $\lambda_\varepsilon$,
    it is independent of the exit position,
    and, for $k \in \{1, \ldots, N\}$,
    \begin{equation}\label{eq:proof exit law}
        \mathbb{P}_{\nu_\varepsilon}(\kappa_\varepsilon = k)
        = - \frac{1}{\lambda_\varepsilon} \int_{\overline{\Gde}} \mathbbm{1}_{\overline{\Gke}} \, \d (\partial_n v_\varepsilon),
    \end{equation}
    where $\mathbbm{1}_{\overline{\Gke}} \in C^0(\overline{\Gde})$, as the holes are separated.
    It therefore remains to estimate $\lambda_\varepsilon$ and the right-hand side of \eqref{eq:proof exit law}.
    For brevity,
    we denote $\mathcal{E} = \mathcal{E}_1(\Kall) + \mathcal{E}_2(\Kall)$,
    so that $\mathcal{E} \to 0$ as $\varepsilon \to 0$.

    \paragraph{Mean exit time.}
    Apply \eqref{eq:green formula} with $u_1 = v_\varepsilon$ and $u_2 = \varphi_\varepsilon$.
    As $v_\varepsilon = 0$ on $\Gde$ and $\partial_n \varphi_\varepsilon = 0$ on $\Gne$,
    one has $\langle \partial_n \varphi_\varepsilon, v_\varepsilon \rangle = 0$,
    so that, using \eqref{eq:edp density}, \eqref{eq:quasi-mode properties} and $\int_{\Omega_\varepsilon} v_\varepsilon = 1$,
    \begin{equation}\label{eq:proof eigenvalue identity}
        \langle \partial_n v_\varepsilon, \varphi_\varepsilon \rangle
        = \bigl( \Delta v_\varepsilon, \varphi_\varepsilon \bigr)_{\Omega_\varepsilon} - \bigl( v_\varepsilon, \Delta \varphi_\varepsilon \bigr)_{\Omega_\varepsilon}
        = - \lambda_\varepsilon \bigl( v_\varepsilon, \varphi_\varepsilon \bigr)_{\Omega_\varepsilon} + C \Keo.
    \end{equation}
    Plugging \eqref{eq:bound scalar} and \eqref{eq:bound normal} into \eqref{eq:proof eigenvalue identity} leads to
    \[
        \lambda_\varepsilon \bigl( 1 + \mathrm{O}(\mathcal{E}_2(\Kall)) \bigr) + \mathrm{O}\bigl(\lambda_\varepsilon \mathcal{E}_1(\Kall)\bigr)
        = \lambda_\varepsilon \bigl( 1 + \mathrm{O}(\mathcal{E}) \bigr)
        = C \Keo.
    \]
    Isolating $\lambda_\varepsilon$ gives, for $\varepsilon$ small enough,
    \begin{equation}\label{eq:proof eigenvalue}
        \lambda_\varepsilon = C \Keo \bigl( 1 + \mathrm{O}(\mathcal{E}) \bigr),
        \qquad
        \mathbb{E}_{\nu_\varepsilon}[\tau_\varepsilon] = \frac{1}{\lambda_\varepsilon} = \frac{1}{C \Keo} \bigl( 1 + \mathrm{O}(\mathcal{E}) \bigr),
    \end{equation}
    which is the first claim.

    \paragraph{Exit position.}
    Fix $k \in \{1, \ldots, N\}$.
    Following \cite[Section 3.3]{nous},
    we compare $v_\varepsilon$ to the quasi-mode of the problem where the $k$-th hole has been removed:
    \[
        \varphi_\varepsilon^{(-k)} = 1 - \sum_{j \neq k} \psi_\varepsilon^{(j)} = \varphi_\varepsilon + \psi_\varepsilon^{(k)},
    \]
    with $\psi_\varepsilon^{(j)}$ given by \cref{hyp:quasi-mode}.
    Summing \eqref{eq:hyp:PDE quasi-mode} and \eqref{eq:hyp:norm} over $j \neq k$ as in \eqref{eq:quasi-mode properties},
    $-\Delta \varphi_\varepsilon^{(-k)} = C(\Keo - \Kek)$ in $\Omega_\varepsilon$,
    $\partial_n \varphi_\varepsilon^{(-k)} = 0$ on $\Gne$, $\| \varphi_\varepsilon^{(-k)} - 1 \|_{L^2(\Omega_\varepsilon)} = \mathrm{O}(\mathcal{E}_2(\Kall))$
    and, since $1 - \mathbbm{1}_{\overline{\Gke}} = \sum_{j \neq k} \mathbbm{1}_{\overline{\Gamma_j^\varepsilon}}$ on $\overline{\Gde}$,
    \begin{equation}\label{eq:proof removed hole}
        \bigl\| \varphi_\varepsilon^{(-k)} - \mathbbm{1}_{\overline{\Gke}} \bigr\|_{L^\infty(\Gde)} 
        = \Bigl\| \sum_{j \neq k} \bigl( \mathbbm{1}_{\overline{\Gamma_j^\varepsilon}} - \psi_\varepsilon^{(j)} \bigr) \Bigr\|_{L^\infty(\Gde)} 
        = \mathrm{O}(\mathcal{E}_1(\Kall)).
    \end{equation}
    Proceeding as in \eqref{eq:proof eigenvalue identity} with $\varphi_\varepsilon^{(-k)}$ instead of $\varphi_\varepsilon$,
    \begin{equation}\label{eq:proof removed identity}
        \langle \partial_n v_\varepsilon, \varphi_\varepsilon^{(-k)} \rangle
        = - \lambda_\varepsilon \bigl( v_\varepsilon, \varphi_\varepsilon^{(-k)} \bigr)_{\Omega_\varepsilon} + C \bigl( \Keo - \Kek \bigr)
        = - C \Kek + \bigl( C \Keo - \lambda_\varepsilon \bigr) - \lambda_\varepsilon \bigl( v_\varepsilon, \varphi_\varepsilon^{(-k)} - 1 \bigr)_{\Omega_\varepsilon}.
    \end{equation}
    The second term in the right-hand side of \eqref{eq:proof removed identity} is a $\mathrm{O}(\Keo \mathcal{E})$ by \eqref{eq:proof eigenvalue}.
    By the Cauchy--Schwarz inequality, \eqref{eq:bound u_0} and \eqref{eq:proof eigenvalue},
    the last term is bounded by $M \lambda_\varepsilon \| \varphi_\varepsilon^{(-k)} - 1 \|_{L^2(\Omega_\varepsilon)} = \mathrm{O}(\Keo \mathcal{E}_2(\Kall))$.
    Therefore,
    decomposing $\mathbbm{1}_{\overline{\Gke}} = \varphi_\varepsilon^{(-k)} + \bigl( \mathbbm{1}_{\overline{\Gke}} - \varphi_\varepsilon^{(-k)} \bigr)$
    and using \eqref{eq:measure bound} with \eqref{eq:proof removed hole} for the second part,
    \[
        \int_{\overline{\Gde}} \mathbbm{1}_{\overline{\Gke}} \, \d (\partial_n v_\varepsilon)
        = \langle \partial_n v_\varepsilon, \varphi_\varepsilon^{(-k)} \rangle + \mathrm{O}(\lambda_\varepsilon \mathcal{E})
        = - C \Kek + \mathrm{O}(\Keo \mathcal{E}).
    \]
    Finally,
    plugging this estimate and \eqref{eq:proof eigenvalue} into \eqref{eq:proof exit law},
    and since $\Kek \leq \Keo$,
    \[
        \mathbb{P}_{\nu_\varepsilon}(\kappa_\varepsilon = k)
        = \frac{C \Kek + \mathrm{O}(\Keo \mathcal{E})}{C \Keo \bigl( 1 + \mathrm{O}(\mathcal{E}) \bigr)}
        = \frac{\Kek}{\Keo} + \mathrm{O}(\mathcal{E}),
    \]
    which concludes the proof.
\end{proof}

\section{On the Green function of the Neumann problem}
\label{sec:app:result Green}
Several quasi-modes in this work depend on the Green function of the Laplacian with Neumann boundary conditions. 
In this section, 
we recall its definition and some of its properties. 
\begin{defi}\label{def:Neumann Green}
    For any $y \in \partial \Omega$, and $p \in (1, 3/2)$, 
    we denote $G(\cdot, y)$ the zero-mean weak $W^{1, p}(\Omega)$ solution of the following problem:
    \begin{equation}\label{eq:PDE Green}
        \left\{
            \begin{aligned}
            \Delta_x G(\cdot, y) = \frac{1}{|\Omega|} \quad & \text{ in } \Omega, \\
            \partial_n G(\cdot, y) = \delta_y \quad & \text{ on } \partial \Omega.
            \end{aligned}
        \right.
    \end{equation}
\end{defi}
The existence, uniqueness and regularity of $G$ are a consequence 
of \cite[Proposition 2.3]{aramaki2018existence}.
Through this definition it is clear that $G$ can be used to solve the following Neumann problem:
given $g \in C^0(\partial \Omega)$, the zero-mean classical solution of the following problem 
\begin{equation}\label{eq:smooth Neumann problem}
    \left\{
        \begin{aligned}
        \Delta u & = \frac{1}{|\Omega|} \int_{\partial \Omega} g \quad && \text{ in } \Omega, \\
        \partial_n u & = g \quad && \text{ on } \partial \Omega,
        \end{aligned}
    \right.
\end{equation}
can be expressed as 
\begin{equation}\label{eq:solution Neumann problem}
    u(x) = \int_{\partial \Omega} G(x, y) g(y) \, \d \sigma(y).
\end{equation}

This function has been studied in the context of the narrow escape problem 
in \cite{Singer2008, ChenFriedman2011, nous} and 
as itself in \cite{GreenFunctionDavid, lindsay2026dimensionalneumanngreensfunction}.
In particular, 
let us recall the following result from \cite{GreenFunctionDavid}
\begin{lem}\label{lem:Green David}
    There exists $K > 0$ such that for any $x \in \Omega$ and $y \in \partial \Omega$,
    the Green function $G$ defined in \cref{def:Neumann Green} and its derivatives can be bounded as 
    \[
        \left| G(x, y)\right| \leq K |x-y|^{-1}
        \qquad \text{ and }\qquad
        \left| \nabla_x G(x, y)\right| \leq K |x-y|^{-2} .
    \]
\end{lem}
\begin{proof}
    Direct consequence of \cite[Equation 4.3]{GreenFunctionDavid}, 
    as $\Omega \times \partial \Omega 
    \subset \overline{\Omega} \times \overline{\Omega} \setminus \{x=y\}$.
\end{proof}
We also recall the main result of \cite{lindsay2026dimensionalneumanngreensfunction}.
\begin{lem}\label{lem:Green scaling exact}
    In the limit $|x-y| \to 0$, 
    the function $G$ defined in \cref{def:Neumann Green} 
    can be written as 
    \[
        G(x, y) = \frac{1}{2\pi}\frac{1}{|x-y|} + \mathrm{O} \left( \left| \log |x-y|\right| \right),
    \]
    where the $\mathrm{O}$ refers to the limit $|x-y| \to 0$, uniformly in $y \in \partial \Omega$ and $x \in \overline{\Omega} \setminus \{y\}$. 
\end{lem} 

With these estimates,
we now extend \eqref{eq:smooth Neumann problem} and \eqref{eq:solution Neumann problem} to the case of $g \in L^p(\partial \Omega)$, $p \in (1, \infty)$.
\begin{lem}\label{lem:Green solution}
    Fix $p \in (1, \infty)$.
    For any $g \in L^p(\partial \Omega)$,
    consider the function $\varphi$ defined on $\Omega$ by
    \[
        \varphi(x) = \int_{\partial \Omega} G(x, y) \, g(y) \, \d\sigma(y).
    \]
    Then $\varphi$ is the zero-mean $W^{1, p}(\Omega)$ weak solution of
    \begin{equation}\label{eq:weak Neumann problem}
        \left\{
            \begin{aligned}
            \Delta \varphi & = \frac{1}{|\Omega|} \int_{\partial \Omega} g \quad && \text{ in } \Omega, \\
            \partial_n \varphi & = g \quad && \text{ on } \partial \Omega,
            \end{aligned}
        \right.
    \end{equation}
    that is, for every $v \in C^1(\overline\Omega)$,
    \[
        \int_\Omega \nabla \varphi \cdot \nabla v
        = - \frac{1}{|\Omega|} \Bigl( \int_{\partial \Omega} g \, \d\sigma \Bigr) \Bigl( \int_\Omega v \Bigr)
        + \int_{\partial \Omega} g \, v \, \d\sigma .
    \]
    Moreover, for almost every $x \in \Omega$,
    \begin{equation}\label{eq:gradient Green solution}
        \nabla \varphi(x) = \int_{\partial \Omega} \nabla_x G(x, y) \, g(y) \, \d\sigma(y).
    \end{equation}
    Finally, if $p \geq 4/3$, then $\varphi \in H^1(\Omega)$,
    and the weak formulation above holds for every $v \in H^1(\Omega)$.
\end{lem}
\begin{proof}
    First notice that $\varphi$ is well defined on $\Omega$:
    for a fixed $x \in \Omega$, \cref{lem:Green David} gives $|G(x, \cdot)| \leq K \, \mathrm{dist}(x, \partial \Omega)^{-1}$ on $\partial \Omega$,
    and $g \in L^p(\partial \Omega) \subset L^1(\partial \Omega)$.

    \paragraph{Weak solution.}
    From \cite[Proposition 2.3]{aramaki2018existence}, 
    since $g \in L^p(\partial \Omega)$ and the compatibility equation is verified, 
    there exists a unique zero-mean $\widehat{\varphi} \in W^{1, p}(\Omega)$
    weak solution to \eqref{eq:weak Neumann problem} and
    \begin{equation}\label{eq:stability phihat}
        \| \widehat{\varphi} \|_{W^{1, p}(\Omega)} \leq C \| g \|_{L^p(\partial \Omega)} .
    \end{equation}
    It remains to show that $\varphi = \widehat{\varphi}$.
    As $p < \infty$,
    there is a sequence $(g_n)_{n \in \mathbb{N}}$ of smooth functions on $\partial \Omega$ such that $g_n \to g$ in $L^p(\partial \Omega)$.
    For each $n \in \mathbb{N}$,
    let $\varphi_n(x) = \int_{\partial \Omega} G(x, y) \, g_n(y) \, \d\sigma(y)$.
    By \eqref{eq:solution Neumann problem},
    $\varphi_n$ is the zero-mean classical solution of \eqref{eq:smooth Neumann problem} with $g_n$ instead of $g$,
    so that, by the Green formula,
    it is also the zero-mean $W^{1, p}(\Omega)$ weak solution associated with $g_n$.
    By linearity, \eqref{eq:stability phihat} applied to $g_n - g$ gives
    \[
        \| \varphi_n - \widehat{\varphi} \|_{W^{1, p}(\Omega)}
        \leq C \| g_n - g \|_{L^p(\partial \Omega)}
        \underset{n \to \infty}{\longrightarrow} 0,
    \]
    so that, up to a subsequence, $\varphi_n \to \widehat{\varphi}$ almost everywhere in $\Omega$.
    On the other hand, for a fixed $x \in \Omega$,
    as $G(x, \cdot)$ is bounded on $\partial \Omega$,
    the Hölder inequality gives
    \begin{equation}\label{eq:pointwise phin}
        \left| \varphi_n(x) - \varphi(x) \right|
        \leq \left\| G(x, \cdot) \right\|_{L^{p'}(\partial \Omega)} \left\| g_n - g \right\|_{L^p(\partial \Omega)}
        \underset{n \to \infty}{\longrightarrow} 0 ,
    \end{equation}
    where $p' = \frac{p}{p-1}$ is the conjugate exponent of $p$.
    Hence $\varphi = \widehat{\varphi}$ almost everywhere in $\Omega$.

    \paragraph{Gradient.}
    Let $u \in C_c^\infty(\Omega)$.
    By definition of the weak gradient,
    \begin{equation}\label{eq:computation gradient Green}
        \left\langle \nabla \varphi, \, u \right\rangle_{\mathcal{D}'(\Omega), C_c^\infty(\Omega)}
        = - \int_{\Omega} \varphi \, \nabla u
        = - \int_{\Omega} \Bigl( \int_{\partial \Omega} G(x, y) \, g(y) \, \d\sigma(y) \Bigr) \nabla u(x) \, \d x .
    \end{equation}
    To exchange the integrals,
    notice that, as $\Omega \subset \real^3$ is bounded,
    $M = \sup_{y \in \partial \Omega} \int_\Omega \bigl( |x - y|^{-1} + |x - y|^{-2} \bigr) \d x < \infty$.
    Hence, by the Tonelli theorem and \cref{lem:Green David},
    \begin{equation}\label{eq:fubini justification}
        \int_{\partial \Omega} \int_{\Omega} \bigl( |G(x, y)| + |\nabla_x G(x, y)| \bigr) |g(y)| \, \d x \, \d\sigma(y)
        \leq K M \| g \|_{L^1(\partial \Omega)}
        < \infty .
    \end{equation}
    Moreover, for a fixed $y \in \partial \Omega$,
    $\Delta_x G(\cdot, y)$ is constant in $\Omega$,
    so that $G(\cdot, y)$ is smooth in $\Omega$ by Weyl's lemma,
    and, as $u$ is compactly supported in $\Omega$, an integration by parts gives
    $- \int_\Omega G(x, y) \nabla u(x) \, \d x = \int_\Omega \nabla_x G(x, y) \, u(x) \, \d x$.
    Therefore, as $u$ and $\nabla u$ are bounded,
    \eqref{eq:fubini justification} and the Fubini theorem in \eqref{eq:computation gradient Green} lead to
    \begin{equation*}
        \begin{aligned}
        \left\langle \nabla \varphi, \, u \right\rangle_{\mathcal{D}'(\Omega), C_c^\infty(\Omega)}
        & = - \int_{\partial \Omega} \Bigl( \int_{\Omega} G(x, y) \, \nabla u(x) \, \d x \Bigr) g(y) \, \d\sigma(y) \\
        & = \int_{\partial \Omega} \Bigl( \int_{\Omega} \nabla_x G(x, y) \, u(x) \, \d x \Bigr) g(y) \, \d\sigma(y) \\
        & = \int_{\Omega} \Bigl( \int_{\partial \Omega} \nabla_x G(x, y) \, g(y) \, \d\sigma(y) \Bigr) u(x) \, \d x .
        \end{aligned}
    \end{equation*}
    Hence \eqref{eq:gradient Green solution} holds in the sense of distributions.
    As both sides of \eqref{eq:gradient Green solution} belong to $L^1(\Omega)$,
    the left one as $\varphi \in W^{1, p}(\Omega)$ and the right one by \eqref{eq:fubini justification},
    they coincide almost everywhere in $\Omega$.

    \paragraph{$H^1$ regularity.}
    Assume now that $p \geq 4/3$, so that $g \in L^{4/3}(\partial \Omega)$ as $\partial \Omega$ is bounded.
    As $\partial \Omega$ is a smooth compact surface,
    the Sobolev embedding $H^{\frac{1}{2}}(\partial \Omega) \subset L^4(\partial \Omega)$ holds
    (see \cite[Theorem 6.5]{DiNezza2012} together with local charts).
    Combined with the trace theorem and the Hölder inequality,
    \[
        \Bigl| \int_{\partial \Omega} g \, v \, \d\sigma \Bigr|
        \leq \| g \|_{L^{4/3}(\partial \Omega)} \| v \|_{L^4(\partial \Omega)}
        \leq C \| g \|_{L^{4/3}(\partial \Omega)} \| v \|_{H^1(\Omega)},
        \qquad v \in H^1(\Omega),
    \]
    so that the right-hand side of the weak formulation is a continuous linear form on $H^1(\Omega)$.
    By the Lax--Milgram theorem on the zero-mean functions of $H^1(\Omega)$,
    together with the Poincaré--Wirtinger inequality on $\Omega$,
    there exists a unique zero-mean $\widetilde{\varphi} \in H^1(\Omega)$ satisfying the weak formulation for every $v \in H^1(\Omega)$.
    As $\Omega$ is bounded, $\widetilde{\varphi} \in W^{1, 4/3}(\Omega)$,
    and $\widetilde{\varphi}$ is a zero-mean $W^{1, 4/3}(\Omega)$ weak solution of \eqref{eq:weak Neumann problem}.
    Applying the first part of this lemma with the exponent $4/3$,
    the uniqueness of this solution gives $\varphi = \widetilde{\varphi} \in H^1(\Omega)$.
\end{proof}

Let us conclude this section with a uniform bound on the $L^2$ norm of the Green function $G$. 

\begin{lem}\label{lem:Green L2 bound}
    There exists a constant $M$ depending only on $\Omega$ such that 
    \begin{equation}\label{eq:green L2 bound}
        \sup_{y \in \partial \Omega} \int_\Omega G(x,y)^2 \, \d x < M.
    \end{equation}
\end{lem}
\begin{proof}
    Let $D > 0$, so that $\Omega \subset B(y, D)$ for every $y \in \overline\Omega$. 
    Using \cref{lem:Green David}, one has that 
    \[
        \int_\Omega G(x,y)^2 \, \d x \leq  K^2 \int_\Omega \frac{\d x}{|x-y|^2}.
    \]
    Then, passing to spherical coordinates centered at $y$ and using $\Omega \subset B(y,D)$,
    \[
        \int_\Omega \frac{\d x}{|x-y|^2} \leq \int_{B(y,D)} \frac{\d x}{|x-y|^2} = \int_0^{2\pi}\!\!\int_0^\pi\!\! \int_0^D \frac{1}{r^2} \, r^2 \sin\theta \, \d r \, \d \theta \, \d \phi = 4\pi D < \infty,
    \]
    which is indeed bounded uniformly in $y \in \partial \Omega$.
\end{proof}

\section{Computation of the asymptotics of the integrals}
 \label{sec:app:integrals}
This section is devoted to the computation of the asymptotics of the integrals appearing in the proof of \cref{prop:quasi-mode spherical holes} and \cref{lem:edp quasi-mode slit}.

\subsection{In the case of spherical holes from \texorpdfstring{\cref{sub:spherical}}{Section 3.2}}
In this subsection, 
we go back to the notation of \cref{sub:spherical} 
and compute the asymptotics of the integrals appearing in the proof of \cref{prop:quasi-mode spherical holes}.

For this let us introduce a local parametrization of the boundary $\partial \Omega$. 
For $r > 0$ and $k\in\{1,\ldots,N\}$ we use the Riemannian
exponential map of $\partial\Omega$ at $\xk$,
\[
    y(\cdot)\colon B(0,r)\subset\real^2\longrightarrow\partial\Omega,
    \qquad y(z)\coloneqq\Exp^{\partial\Omega}_{\xk}(z),
\]
where $T_{\xk}\partial\Omega\cong\real^2$ is identified through an orthonormal
basis $(e_1,e_2)$ and $n(\xk)$ is the outward unit normal to $\Omega$ at $\xk$.
As $\partial \Omega$ is compact and smooth, 
there exists $r>0$ independent of $k$ such that for any $\rek<r$, $y(\cdot)$ is a diffeomorphism from a neighborhood of $0$
onto its image.
Let us state the three properties on $y$ that are needed for the asymptotic analysis of the integrals in \cref{sub:spherical}.
 
\paragraph{Gauss's lemma.}
Radial lines through the origin are unit-speed geodesics issued from $\xk$, so
\begin{equation}\label{eq:gauss}
    d_{\partial\Omega}\bigl(y(z),\xk\bigr)=|z|,
    \qquad\text{hence}\qquad
    y^{-1}(\Gke)=B(0,\rek)\ \ \text{exactly}.
\end{equation}
 
\paragraph{Jacobian.}
Writing the surface measure on $\Gke$ as $\d\sigma(y)=J_k(z)\,\d z$, the area element of
geodesic polar coordinates satisfies, uniformly in $k$,
\begin{equation}\label{eq:jacobian}
    \left| J_k(z) - 1 \right|\le C|z|^2,
    \qquad C=C(\Omega)>0.
\end{equation}

\paragraph{Extrinsic distance.}
The map $\Exp^{\partial\Omega}_{\xk}$ has the second-order expansion
using the second fundamental form $\II_{\xk}$ of $\partial\Omega$ at $\xk$:
\begin{equation}\label{eq:extrinsic expansion}
    y(z)=\xk+z+\tfrac12\,\II_{\xk}(z,z)\,n(\xk)+\mathrm{O}(|z|^3).
\end{equation}
Notice that in \eqref{eq:extrinsic expansion}, 
the quadratic
correction is normal to $T_{\xk}\partial\Omega$; consequently, for
$w_0,w\in\overline{B(0,1)}$,
\begin{equation}\label{eq:extrinsic}
    \bigl|y(\rek w_0)-y(\rek w)\bigr|=\rek\,|w_0-w|\,\bigl(1+\mathrm{O}((\rek)^2)\bigr),
    \qquad\text{uniformly}.
\end{equation}
 
Let us now state the main result of this subsection, 
which is a direct computation of the integrals needed in \cref{sub:spherical}. 
\begin{lem}\label{lem:integrals circular}
    Fix $k\in\{1,\ldots,N\}$ and let $\rho_\varepsilon^{(k)}$ be the function defined on $\Gke$ by
    \begin{equation}\label{eq:weight}
        \rho_\varepsilon^{(k)} : y \mapsto \bigl((\rek)^2-d_{\partial\Omega}(y,\xk)^2\bigr)^{-1/2}.
    \end{equation}
    Then, as
    $\varepsilon\to0$, it holds that 
    \begin{enumerate}[label=\textup{(\alph*)}]
        \item \textup{\textbf{(Integrability)}} for every $p \in [1, 2)$,\;
        $\displaystyle\int_{\Gke}\rho_\varepsilon^{(k)}(y)^p\,\d\sigma(y)=\frac{2\pi}{2-p}\,(\rek)^{2-p}\bigl(1+\mathrm{O}((\rek)^2)\bigr)$,
        while $\rho_\varepsilon^{(k)} \notin L^2(\Gke)$;
        \item \textup{\textbf{(Mass)}} in particular, for $p = 1$,\;
        $\displaystyle\int_{\Gke}\rho_\varepsilon^{(k)}(y)\,\d\sigma(y)=2\pi\rek\bigl(1+\mathrm{O}((\rek)^2)\bigr)$;
        \item \textup{\textbf{(Newtonian potential)}} for every $x\in\Gke$,\;
        $\displaystyle\int_{\Gke}\rho_\varepsilon^{(k)}(y)\,\frac{1}{2\pi |x-y|}\,\d\sigma(y)=\frac{\pi}{2} +\mathrm{O}(\rek)$;
        \item \textup{\textbf{(Logarithmic potential)}} for every $x\in\Gke$,\;
        $\displaystyle\int_{\Gke}\rho_\varepsilon^{(k)}(y)\,\log|x-y|\,\d\sigma(y)=\mathrm{O}\bigl(\rek\,|\log\rek|\bigr)$.
    \end{enumerate}
\end{lem}
 
\begin{proof}
    Fix $k\in\{1,\ldots,N\}$ and take $y \in \Gke$.
    The function $\rho_\varepsilon^{(k)}$ in \eqref{eq:weight} can be expressed, after
    the change of variables $y=y(z)$, by \eqref{eq:gauss} with
    $z\in B(0,\rek)$, and after the rescaling $z=\rek w$,
    \begin{equation}\label{eq:gk exact}
        \rho_\varepsilon^{(k)}\bigl(y(\rek w)\bigr)
        =\bigl((\rek)^2-|\rek w|^2\bigr)^{-1/2}
        =\frac{1}{\rek}\,\frac{1}{\sqrt{1-|w|^2}}.
    \end{equation}
    Thus, for any $F \colon \Gke \to \real$ regular enough so the integrals are well defined,
    \begin{equation}\label{eq:master}
        \int_{\Gke}\rho_\varepsilon^{(k)}(y)\,F(y)\,\d\sigma(y)
        =\rek\int_{B(0,1)}\frac{J_k(\rek w)}{\sqrt{1-|w|^2}}\,F\bigl(y(\rek w)\bigr)\,\d w.
    \end{equation}
    Using \eqref{eq:jacobian}, 
    we can expand the integrand in \eqref{eq:master} as
    \begin{equation}\label{eq:master expanded}
        \int_{\Gke}\rho_\varepsilon^{(k)}(y)\,F(y)\,\d\sigma(y)
        =\rek\int_{B(0,1)}\frac{F\bigl(y(\rek w)\bigr)}{\sqrt{1-|w|^2}}\,\d w
        +\mathrm{O}\bigl((\rek)^3\bigr)\int_{B(0,1)}\frac{\bigl|F\bigl(y(\rek w)\bigr)\bigr|}{\sqrt{1-|w|^2}}\,\d w.
    \end{equation}
 
    \paragraph{(a) Integrability.}
    For $p \in [1, 2)$, the change of variables $y = y(\rek w)$ together with \eqref{eq:gk exact} and \eqref{eq:jacobian} gives
    \[
        \int_{\Gke}\rho_\varepsilon^{(k)}(y)^p\,\d\sigma(y)
        = (\rek)^{2-p} \int_{B(0,1)} \frac{J_k(\rek w)}{(1-|w|^2)^{p/2}} \, \d w
        = (\rek)^{2-p} \bigl(1+\mathrm{O}((\rek)^2)\bigr) \, 2\pi \int_0^1 \frac{s \, \d s}{(1-s^2)^{p/2}}.
    \]
    The last integral equals $1/(2-p)$, which gives the claimed formula,
    and it diverges for $p = 2$, so that $\rho_\varepsilon^{(k)} \notin L^2(\Gke)$.

    \paragraph{(b) Mass.}
    This is (a) with $p = 1$.

    For the two remaining estimates, we apply \eqref{eq:master expanded} with $F(y)=|x-y|^{-1}$ and
    $F(y)=\log|x-y|$, 
    where $x=y(\rek w_0)\in\Gke$ with~$w_0\in\overline{B(0,1)}$.

    \paragraph{(c) Newtonian potential.}
    By \eqref{eq:extrinsic}, 
    $|x-y(\rek w)|^{-1}=\dfrac{1}{\rek|w_0-w|}\bigl(1+\mathrm{O}((\rek)^2)\bigr)$
    uniformly for~$w \neq w_0$, so \eqref{eq:master expanded} yields
    \begin{equation}\label{eq:self potential rescaled}
        \int_{\Gke}\rho_\varepsilon^{(k)}(y)\,\frac{\d\sigma(y)}{2\pi|x-y|}
        =\bigl(1+\mathrm{O}((\rek)^2)\bigr)\int_{B(0,1)}\frac{\d w}{2\pi\sqrt{1-|w|^2}\,|w_0-w|}.
    \end{equation}
    The integral in the right-hand side of \eqref{eq:self potential rescaled} is well known 
    in potential theory, 
    and its value is independent of~$w_0\in\overline{B(0,1)}$ (see \cite[Equation 4.6]{Tzou2021}).
    Hence, 
    the integral can be evaluated as 
    \[
        \int_{\Gke}\rho_\varepsilon^{(k)}(y)\,\frac{\d\sigma(y)}{2\pi|x-y|}
        = \frac{\pi}{2} \left( 1 + \mathrm{O}((\rek)^2)\right), \qquad \text{ uniformly in } x \in \Gke.
    \]
    \paragraph{(d) Logarithmic potential.}
    By \eqref{eq:extrinsic}, $\log|x-y(\rek w)|=\log\rek+\log|w_0-w|+\mathrm{O}((\rek)^2)$
    uniformly for~$w \neq w_0$, so
    \eqref{eq:master expanded} gives
    \[
        \int_{\Gke}\rho_\varepsilon^{(k)}(y)\,\log|x-y|\,\d\sigma(y)
        =\rek\log\rek\underbrace{\int_{B(0,1)}\frac{\d w}{\sqrt{1-|w|^2}}}_{=\,2\pi}
        +\rek\int_{B(0,1)}\frac{\log|w_0-w|}{\sqrt{1-|w|^2}}\,\d w+\mathrm{O}\bigl((\rek)^3\bigr).
    \]
    The last integral is finite and bounded uniformly in
    $w_0\in\overline{B(0,1)}$: the singularities $(1-|w|^2)^{-1/2}$ and
    $\log|w_0-w|$ lie in $L^p_{\mathrm{loc}}$ for every $p<2$, resp.\ every
    $p<\infty$, so their product is integrable by Hölder's inequality, and
    continuity in $w_0$ with compactness gives the uniform bound. Therefore
    \[
        \int_{\Gke}\rho_\varepsilon^{(k)}(y)\,\log|x-y|\,\d\sigma(y)
        =2\pi\rek\log\rek+\mathrm{O}(\rek)=\mathrm{O}\bigl(\rek\,|\log\rek|\bigr),
    \]
    which is the claimed estimate.
\end{proof}

The previous estimates also give the continuity of the quasi-mode on the holes, required in \cref{hyp:quasi-mode}.
\begin{lem}\label{lem:continuity flat holes}
    Fix $k\in\{1,\ldots,N\}$.
    The function $\psi_\varepsilon^{(k)}$ defined in \eqref{eq:quasi-mode flat holes} is continuous on $\overline{\Gde}$.
\end{lem}

\begin{proof}
    Recall that, by local elliptic regularity,
    $G(\cdot, y)$ is smooth on $\overline{\Omega} \setminus \{y\}$ for every $y \in \partial \Omega$.
    For $j \neq k$, the hole $\overline{\Gamma_j^\varepsilon}$ is at a positive distance from $\overline{\Gke}$,
    so that $\psi_\varepsilon^{(k)}$ is continuous on $\overline{\Gamma_j^\varepsilon}$ by dominated convergence,
    as $\rho_\varepsilon^{(k)} \in L^1(\Gke)$.
    On $\overline{\Gke}$,
    we use the parametrization $y(\cdot)$ of $\partial \Omega$ around $\xk$ introduced at the beginning of this subsection,
    for which $y^{-1}(\Gke) = B(0, \rek)$ by \eqref{eq:gauss},
    and we decompose $G = \frac{1}{2\pi |x - y|} + g(x, y)$ as in \cref{lem:Green scaling exact}.
    As $y(\cdot)$ is a smooth embedding with $\d y(0) = \I$,
    one can write $|y(z_0) - y(z)| = |z_0 - z| \, Q(z_0, z)^{1/2}$ with $Q$ smooth and positive on $\overline{B(0, \rek)}^2$.
    Denoting $\rho_{\mathrm{flat}}(z) = ((\rek)^2 - |z|^2)^{-1/2}$ and $\mathcal{K}(z_0, z) = J_k(z) \, Q(z_0, z)^{-1/2}$,
    the Newtonian part of $\psi_\varepsilon^{(k)}$ at $x = y(z_0)$ is, up to the factor $1/\pi^2$,
    \[
        \int_{B(0, \rek)} \rho_{\mathrm{flat}}(z) \, \frac{\mathcal{K}(z_0, z)}{|z_0 - z|} \, \d z
        = \mathcal{K}(z_0, z_0) \int_{B(0, \rek)} \frac{\rho_{\mathrm{flat}}(z)}{|z_0 - z|} \, \d z
        + \int_{B(0, \rek)} \rho_{\mathrm{flat}}(z) \, \frac{\mathcal{K}(z_0, z) - \mathcal{K}(z_0, z_0)}{|z_0 - z|} \, \d z .
    \]
    The first integral equals $\pi^2$ for every $z_0 \in \overline{B(0, \rek)}$ (see \cite[Equation 4.6]{Tzou2021}),
    and $z_0 \mapsto \mathcal{K}(z_0, z_0)$ is continuous.
    In the second integral, $\mathcal{K}$ is Lipschitz in its second variable uniformly in $z_0$,
    so that the integrand is bounded by $C \rho_{\mathrm{flat}} \in L^1(B(0, \rek))$ and continuous in $z_0$ for $z \neq z_0$:
    this integral is continuous in $z_0$ by dominated convergence.
    Finally, $x \mapsto g(x, y)$ is continuous on $\overline{\Omega} \setminus \{y\}$,
    and $|g(x, y)| \leq C (1 + |\log|x - y||)$ by \cref{lem:Green scaling exact}.
    As $\rho_\varepsilon^{(k)} \in L^{p}(\Gke)$ for some $p \in (1, 2)$ by the (Integrability) estimate of \cref{lem:integrals circular}
    and $y \mapsto \log|x - y|$ is bounded in $L^{p'}(\Gke)$ uniformly in $x \in \overline{\Gke}$, with $p' = p/(p-1)$,
    the family $\bigl(\rho_\varepsilon^{(k)} g(x, \cdot)\bigr)_{x \in \overline{\Gke}}$ is uniformly integrable,
    and the Vitali convergence theorem shows that $x \mapsto \int_{\Gke} \rho_\varepsilon^{(k)}(y) g(x, y) \, \d\sigma(y)$ is continuous on $\overline{\Gke}$.
    Hence $\psi_\varepsilon^{(k)}$ is continuous on $\overline{\Gde}$.
\end{proof}
 
\subsection{In the case of the slit holes from \texorpdfstring{\cref{sec:slit}}{Section 4}}
In this subsection, 
we go back to the notation of \cref{sec:slit}
and prove a technical lemma needed in the proof of \cref{lem:edp quasi-mode slit}.

\begin{lem}\label{lem:computation slit}
    For a fixed $k\in\{1,\ldots,N\}$ and $\varepsilon > 0$, 
    let $\Gke$ be the slit hole defined in \cref{sec:slit}, 
    and $G$ the Green function defined in \cref{def:Neumann Green}.
    Then for any $x \in \Gke$,
    it holds that 
    in the limit $\varepsilon \to 0$,
    \[
        \int_{\Hk} G(x, y) \, \d \sigma(y) = - \frac{1}{\pi}\log\left(w_\varepsilon^{(k)}\right) + \mathrm{O}(1).
    \]
\end{lem}

\begin{proof}
    In this proof we fix $k \in \{1, \ldots, N\}$, 
    we fix $x \in \Gke$, drop the index $k$ for clarity
    and denote $u$ the integral in the statement of the lemma.
    First notice that $G$ can be decomposed as 
    \begin{equation}\label{eq:estimate of u first step}
        u(x) =  \int_{\Hk} G(x, y) \, \d \sigma(y)= 
        \int_{\mathcal{H}} \left[ \frac{1}{2\pi|x-y|} + r(x, y)\right] \, \d \sigma(y),
    \end{equation}
    with $r : \Gke \times \mathcal{H} \to \mathbb{R}$ the remainder.  
    We postpone the bound on the integral over $r$ to the end of the proof,
    and first treat the singular contribution
    \begin{equation}\label{eq:def u0}
        u_0(x) = \frac{1}{2\pi} \int_{\mathcal{H}} \frac{\d \sigma(y)}{|x - y|}
        = \frac{1}{2\pi} \int_{-\ell}^{\ell} \frac{\d s}{|x - \gamma(s)|},
    \end{equation}
    where we used the parametrization of $\mathcal{H}$ introduced in \cref{sub:hole shapes notation}.
    Notice that as $\mathcal{H}$ is a smooth curve, 
    its parametrization $\gamma$ can be chosen to be smooth and can be Taylor expanded around 
    $s_x \in [-\ell, \ell]$, 
    the point such that $\gamma(s_x)$ is the nearest point on $\mathcal{H}$ to $x$,
    as 
    \begin{equation}\label{eq:frenet expansion}
        \gamma(s) = \gamma(s_x) + \sigma \tau(s_x) + E(s, s_x),
        \qquad \text{ with } |E(s, s_x)| \leq K_1 \sigma^2,
    \end{equation}
    writing $\sigma = s - s_x$ and with $K_1 > 0$ depending only on the curvature of $\mathcal{H}$.
    Then, 
    it holds that 
    \begin{equation}\label{eq:x minus gamma frenet}
        x - \gamma(s) = (t - \sigma)\tau(s_x) + \rho - E(s, s_x).
    \end{equation}     
    It is now natural to introduce the flat slit distance, 
    defined as 
    \begin{equation}\label{eq:def b}
        b(s) = \bigl| (t - \sigma)\tau(s_x) + \rho \bigr|
        = \sqrt{\rho^2 + (t - \sigma)^2} = \sqrt{\rho^2 + (z - s)^2},
    \end{equation}
    where $z = t + s_x$, 
    and to split the integral in \eqref{eq:def u0} as
    \begin{equation}\label{eq:split u0}
        u_0(x) = \frac{1}{2\pi} \underbrace{\int_{-\ell}^{\ell} \frac{\d s}{b(s)}}_{=: u_{\mathrm{flat}}(x)}
        + \frac{1}{2\pi} \underbrace{\int_{-\ell}^{\ell} \left( \frac{1}{|x - \gamma(s)|} - \frac{1}{b(s)} \right) \d s}_{=: u_{\mathrm{rest}}(x)},
    \end{equation}
    and estimate the two terms separately.

    \paragraph{The flat integral $u_{\mathrm{flat}}$.}
    By \eqref{eq:def b}, 
    notice that 
    $u_{\mathrm{flat}}$ can be computed directly in the case $\rho > 0$, 
    \begin{equation}\label{eq:prolate arcsinh}
        \int_{-\ell}^{\ell} \frac{\d s}{\sqrt{\rho^2 + (z - s)^2}}
        = \arcsinh\left( \frac{\ell - z}{\rho}\right) + \arcsinh\left( \frac{\ell + z}{\rho}\right)
        = \log\frac{(\ell - z) + r_-}{\rho} + \log\frac{(\ell + z) + r_+}{\rho},
    \end{equation}
    where $r_- = \sqrt{\rho^2 + (z - \ell)^2}$ and $r_+ = \sqrt{\rho^2 + (z + \ell)^2}$
    are the distances from the point to the two foci.
    Notice that, using $(a-b)(a+b) = a^2 - b^2$, we have
    \[
        \Big( r_\pm(\rho, z) - (\ell \pm z)\Big)\Big( r_\pm(\rho, z) + (\ell \pm z)\Big) = \rho^2.
    \]
    Summing the previous equation for $r_+$ and $r_-$, we obtain the following identity (dropping the dependence in $\rho$ and $z$ for readability):
    \[
        r_+ + r_- - 2\ell 
        = \frac{\rho^2}{r_+ + (\ell + z)} + \frac{\rho^2}{r_- + (\ell - z)}
        =\rho^2 \frac{r_+ + r_- + 2\ell}{\left(r_+ + (\ell + z)\right)\left(r_- + (\ell - z)\right)}.
    \]
    Finally, the explicit expression of $u_{\mathrm{flat}}$ in \eqref{eq:prolate arcsinh} can be rewritten as:
    \begin{equation}\label{eq:u explicit cylindrical 2}
        \begin{aligned}
            u_{\mathrm{flat}}(x)  
            & = \log\left( \frac{\left(r_+(\rho, z)  + (\ell + z)\right)\left(r_-(\rho, z) + (\ell - z)\right)}{\rho^2} \right)\\
            & = - \log\left( \frac{r_+(\rho, z)  + r_-(\rho, z)  - 2\ell}{r_+(\rho, z)  + r_-(\rho, z)  + 2\ell}\right)
        \end{aligned}
    \end{equation}
    By \eqref{eq:slit def}, 
    the point $x \in \Gke$ lies on the ellipse of semi-axes $a = \sqrt{\ell^2 + w^2}$ and $w$ in the $(z, \rho)$ plane,
    where $w = w_\varepsilon^{(k)}$ and $a = a_\varepsilon^{(k)}$ is defined in \eqref{eq:slit major axis}.
    As $a^2 - w^2 = \ell^2$, its foci are $z = \pm \ell$, 
    and the sum of the distances to the foci is the major axis length, namely
    \[
        r_+(\rho, z)  + r_-(\rho, z)  = 2 a, \qquad \text{ for } x \in \Gke. 
    \]
    Thus, 
    on $\Gke$,
    using \eqref{eq:slit major axis},
    we have 
    \begin{equation}\label{eq:u flat final}\begin{aligned}
        u_{\mathrm{flat}}(x)  
        & = - \log\left( \frac{2a - 2\ell}{2a + 2\ell}\right)
        = \log\left( \frac{(a + \ell)^2}{w^2}\right) \\
        & = - 2 \log(w) + 2 \log\left( a + \ell \right)
        = - 2 \log(w) + 2 \log(2\ell) + \mathrm{O}(w^2),
    \end{aligned}
    \end{equation}
    which is the expected result.
    For the case $\rho = 0$, 
    the exact same computation can be done, 
    noticing that $x \not \in \mathcal{H}$ and $\rho = 0$ implies that $|z| > \ell$,
    then the integral can be computed directly
    \begin{align*}
        u_{\mathrm{flat}}(x)  
            & = \int_{-\ell}^{\ell} \frac{1}{|z - s|}\,  \d s\\
            & = \log\left| \frac{z+\ell}{z - \ell}\right|.
    \end{align*}
    On $\Gke$ with $\rho = 0$ one has $|z| = a$, so this equals $\log\frac{a+\ell}{a-\ell} = \log\frac{(a+\ell)^2}{w^2}$,
    which is exactly the first line of \eqref{eq:u flat final}.

    \paragraph{The difference $u_{\mathrm{rest}}$.}
    By the very definition of the foot point $s_x$, the tangential offset $t$ and the
    running variable $\sigma$ are always aligned so that
    \begin{equation}\label{eq:foot point alignment}
        |t - \sigma| \geq |\sigma| \qquad \text{for every admissible } \sigma.
    \end{equation}
    Indeed, if $\gamma(s_x)$ is interior to $\mathcal{H}$, 
    one has that $t = 0$,
    and \eqref{eq:foot point alignment} is an equality;
    if the foot point is an endpoint, say $s_x = \ell$, then $\sigma = s - \ell \leq 0$
    while optimality forces $t = \langle x - \gamma(\ell), \tau(\ell)\rangle \geq 0$,
    so $t - \sigma \geq |\sigma|$.
    Consequently, from \eqref{eq:def b},
    \begin{equation}\label{eq:b lower bound}
        b(s) \geq |t - \sigma| \geq |\sigma|.
    \end{equation}
    Moreover, by \eqref{eq:x minus gamma frenet} and the reverse triangle inequality,
    \begin{equation}\label{eq:a b close}
        \bigl| |x - \gamma(s)| - b(s) \bigr| \leq |E(s, s_x)| \leq K_1 \sigma^2.
    \end{equation}
    Fix now $\sigma_0 = (2K_1)^{-1}$, which depends only on $\mathcal{H}$.
    For $|\sigma| > \sigma_0$, $b(s) \geq |\sigma| > \sigma_0$ by \eqref{eq:b lower bound}.
    Moreover, $\mathcal{H}$ is a compact embedded curve, so there is $c_{\mathcal{H}} > 0$ with
    $|\gamma(s) - \gamma(s')| \geq c_{\mathcal{H}} |s - s'|$,
    and $|x - \gamma(s_x)| \leq w_\varepsilon^{(k)} \to 0$ for $x \in \Gke$, 
    hence $|x - \gamma(s)| \geq c_{\mathcal{H}} \sigma_0 / 2$ for $\varepsilon$ small.
    Both quantities are thus bounded below by a fixed
    constant $c(\sigma_0) > 0$,
    so the integrand is bounded by $2/c(\sigma_0)$.
    For $|\sigma| \leq \sigma_0$, \eqref{eq:b lower bound} and \eqref{eq:a b close} give
    $K_1 \sigma^2 \leq \tfrac{1}{2}|\sigma| \leq \tfrac{1}{2} b(s)$,
    hence $|x - \gamma(s)| \geq \tfrac{1}{2} b(s)$, and therefore
    \begin{equation}\label{eq:diff pointwise close}
        \left| \frac{1}{|x - \gamma(s)|} - \frac{1}{b(s)} \right|
        = \frac{\bigl| |x - \gamma(s)| - b(s)\bigr|}{|x - \gamma(s)|\, b(s)}
        \leq \frac{K_1 \sigma^2}{\tfrac{1}{2} b(s)^2}
        \leq \frac{2 K_1 \sigma^2}{\sigma^2} = 2 K_1.
    \end{equation}
    
    In either case the integrand of $u_{\mathrm{rest}}$ is bounded uniformly in $x$,
    and since $\mathcal{H}$ has finite length,
    \begin{equation}\label{eq:Delta bounded}
        u_{\mathrm{rest}}(x) = \mathrm{O}(1).
    \end{equation}
    Combining \eqref{eq:split u0}, \eqref{eq:u flat final} and \eqref{eq:Delta bounded},
    \begin{equation}\label{eq:u0 final}
        u_0(x) = -\frac{1}{\pi} \log(w) + \mathrm{O}(1).
    \end{equation}

    \paragraph{The remainder r.}
    From \cref{lem:Green scaling exact},
    there exist constants $\delta > 0$ and $C_\delta > 0$ such that for any $ y \in \mathcal{H}$ with $|x - y| < \delta$,
    one has that 
    \[  
            |r(x, y)| \leq C_\delta\Bigl(1 + \bigl|\log|x - y|\bigr|\Bigr).
    \]
        Furthermore, 
        from \cref{lem:Green David},
        for any $y \in \mathcal{H}$ with $|x - y| \geq \delta$,
        one has that 
    \[
            |r(x, y)| \leq |G(x,y)| + \frac{1}{2\pi|x-y|} \leq \Bigl(K + \frac{1}{2\pi}\Bigr) \delta^{-1}.
    \]
        Then, 
        the integral of the function $r$ on $\mathcal{H}$ can be bounded as
    \begin{equation}\label{eq:estimate of r}
    \int_{\mathcal{H}} |r(x, y)| \, \d \sigma(y) 
    \leq C_\delta \int_{\mathcal{H} \cap B(x, \delta)} \Bigl(1 + \bigl|\log|x - y|\bigr|\Bigr) \, \d \sigma(y) + \Bigl(K + \frac{1}{2\pi}\Bigr) \delta^{-1} |\mathcal{H}|.
    \end{equation}

    It remains to check that the right-hand side of \eqref{eq:estimate of r} is $\mathrm{O}(1)$.
    The far term $(K + \frac{1}{2\pi})\delta^{-1}|\mathcal{H}|$ is a fixed constant, and so is $C_\delta |\mathcal{H}|$.
    For the near term, we may take $\delta$ smaller than the tubular radius of $\mathcal{H}$,
    so that $\mathcal{H} \cap B(x, \delta)$ is a single subarc around $s_x$ with $|\sigma| \leq C\delta < 1$;
    on it $|x - \gamma(s)| < \delta < 1$, hence $\bigl|\log|x-\gamma(s)|\bigr| = -\log|x - \gamma(s)|$.
    Using $|x - \gamma(s)| \geq \tfrac{1}{2} b(s)$ from the previous paragraph,
    \[
        \int_{\mathcal{H} \cap B(x, \delta)} \bigl|\log|x - y|\bigr| \, \d \sigma(y)
        \leq \log 2 \cdot |\mathcal{H}| - \int_{|\sigma| \leq C\delta} \log b(s) \, \d s.
    \]
    By \eqref{eq:def b} and the substitution $\eta = t - \sigma$, the last integral equals
    \[
        - \int \log b(s) \, \d s
        = -\frac{1}{2} \int \log\bigl( \rho^2 + \eta^2 \bigr) \, \d \eta,
    \]
    whose explicit primitive $\eta \log(\rho^2 + \eta^2) - 2\eta + 2\rho \arctan(\eta / \rho)$
    stays bounded over the fixed compact $\eta$-range as $\rho \to 0$,
    since $\eta \log(\rho^2 + \eta^2)$ is bounded near $\eta = 0$ and $\arctan$ is bounded.
    Hence the near term is also $\mathrm{O}(1)$, and $\int_{\mathcal{H}} r(x, y)\, \d\sigma(y) = \mathrm{O}(1)$.

    \medskip
    Collecting \eqref{eq:u0 final} and this last bound in \eqref{eq:estimate of u first step},
    \[
        u(x) = -\frac{1}{\pi} \log(w) + \mathrm{O}(1),
    \]
    which is the announced estimate.
\end{proof}

\printbibliography

\end{document}